\documentclass{amsart}

\usepackage{amssymb, amsmath,mathtools,mathabx}
\usepackage{mathrsfs}
\usepackage{amscd}
\usepackage{verbatim}
\usepackage{stmaryrd}
\usepackage{bbm}
\usepackage[dvipsnames]{xcolor}

\usepackage{enumerate}

\usepackage[colorlinks,linkcolor={blue},citecolor={blue},urlcolor={purple},]{hyperref}

\usepackage[colorinlistoftodos,prependcaption,textsize=tiny]{todonotes}

\usepackage{tabularx}
\newcolumntype{L}{>{\arraybackslash}X}
\usepackage{multirow}

\theoremstyle{plain}
\newtheorem{theorem}{Theorem}[section]
\theoremstyle{remark}
\newtheorem{remark}[theorem]{Remark}
\newtheorem{example}[theorem]{Example}
\theoremstyle{plain}
\newtheorem{corollary}[theorem]{Corollary}
\newtheorem{lemma}[theorem]{Lemma}
\newtheorem{proposition}[theorem]{Proposition}
\newtheorem{definition}[theorem]{Definition}

\newtheorem{assumption}[theorem]{Assumption}

\numberwithin{equation}{section}

\def\N{{\mathbb N}}
\def\Z{{\mathbb Z}}

\def\R{{\mathbb R}}

\renewcommand{\P}{\mathbf{P}}
\newcommand{\E}{\mathbf{E}}
\newcommand{\F}{\mathscr{F}}

\newcommand{\p}{\mathbb{P}}
\newcommand{\q}{\mathbb{Q}}
\newcommand{\T}{\mathbb{T}}
\renewcommand{\S}{\mathbb{S}}

\newcommand{\calL}{\mathscr{L}}

\newcommand{\h}{\mathrm{H}}
\renewcommand{\l}{\mathrm{L}}
\newcommand{\Ls}{\mathbb{L}}
\newcommand{\Hs}{\mathbb{H}}
\renewcommand{\H}{\Hs^\sigma_0} 
\newcommand{\HL}{\mathcal{H}} 
\newcommand{\HS}{\mathscr{L}_2}

\newcommand{\Borel}{\mathscr{B}}

\renewcommand{\emptyset}{\varnothing}

\newcommand{\g}{\gamma}

\newcommand{\om}{\omega}
\renewcommand{\O}{\Omega}

\newcommand{\z}{\mathcal{Z}}

\newcommand{\Dom}{\mathcal{O}}
\renewcommand{\t}{\mathcal{T}}
\newcommand{\A}{\mathcal{A}}
\renewcommand{\v}{V}

\newcommand{\vu}{\underline{v}}
\newcommand{\vut}{\underline{v_t}}
\newcommand{\wt}{\widetilde}
\newcommand{\wc}{\widecheck}
\newcommand{\wh}{\widehat}
\newcommand{\supp}{\mathrm{supp}\,}
\renewcommand{\sp}{\mathrm{span}}
\newcommand{\one}{\mathbf{1}}
\newcommand{\dd}{\mathrm{d}}
\newcommand{\Do}{\mathsf{D}}
\newcommand{\X}{\mathcal{X}}
\newcommand{\XX}{\mathbb{X}}
\newcommand{\x}{\mathbf{x}}
\newcommand{\qq}{\mathbf{q}}
\newcommand{\kk}{\mathbf{k}}
\renewcommand{\a}{\mathbf{a}}
\newcommand{\n}{n}
\newcommand{\Id}{\mathrm{Id}}
\newcommand{\loc}{{\rm loc}}

\newcommand{\Tcal}{\mathcal{O}}
\newcommand{\plane}{\mathcal{P}}
\newcommand{\svec}{\mathrm{span}}
\renewcommand{\Pr}{\mathrm{Pr}_2}
\newcommand{\Ran}{\mathcal{R}}
\newcommand{\Ma}{\mathscr{M}^{{\rm L}}_t}
\newcommand{\Mat}{\mathscr{M}^{{\rm L}}_T}
\renewcommand{\L}{{\rm L}}
\newcommand{\Leb}{\mathrm{Leb}}
\newcommand{\zl}{\mathcal{J}_\l}
\newcommand{\Jv}{\wt{J}}
\newcommand{\sone}{\gamma}
\newcommand{\vd}{v^{\rm{det}}}
\newcommand{\invl}{\wt{\mu}}
\newcommand{\B}{\mathcal{S}}
\newcommand{\Op}{\mathcal{O}}
\newcommand{\Pl}{L}
\newcommand{\Ito}{\mathscr{J}}
\newcommand{\Df}{\mathcal{O}_f}
\newcommand{\Fun}{\Phi}

\newcommand{\Energy}{\mathcal{E}}
\newcommand{\fvel}{F^{{\rm vel}}}
\newcommand{\fpr}{F^{{\rm pr}}}

\newcommand{\D}{\mathscr{D}}

\newcommand{\hh}{\mathscr{H}}
\newcommand{\ww}{\mathscr{W}}
\newcommand{\WW}{\mathbb{W}}
\newcommand{\V}{\mathcal{V}}
\newcommand{\U}{\mathcal{U}}
\newcommand{\uu}{\mathscr{U}}
\newcommand{\KK}{\mathcal{N}}
\newcommand{\Set}{\mathscr{Q}}
\newcommand{\dist}{\mathrm{dist}_{\partial\Dom}}
\newcommand{\sign}{\mathrm{sign}}

\newcommand{\embed}{\hookrightarrow}

\newcommand{\vin}{\eta}
\newcommand{\vf}{\xi}

\allowdisplaybreaks

\begin{document}

\author{Antonio Agresti}
\address{Delft Institute of Applied Mathematics\\
Delft University of Technology \\ P.O. Box 5031\\ 2600 GA Delft\\The
Netherlands} 
\curraddr{Department of Mathematics Guido Castelnuovo, Sapienza University of Rome,
P.le Aldo Moro 2, 00185 Rome, Italy}
\email{antonio.agresti92@gmail.com}
\thanks{The author has received funding from the VICI subsidy VI.C.212.027 of the Netherlands Organisation for Scientific Research (NWO). The author is a member of GNAMPA (IN$\delta$AM)}

\date\today

\title[Lagrangian chaos and unique ergodicity for stochastic PE\lowercase{s}]{Lagrangian chaos and unique ergodicity\\ for stochastic primitive equations}



\begin{abstract}
We show that the Lagrangian flow associated with the stochastic 3D primitive equations (PEs) with nondegenerate noise is chaotic, i.e., the corresponding top Lyapunov exponent is strictly positive almost surely. This result builds on the landmark work by Bedrossian, Blumenthal, and Punshon-Smith \cite{BBPS2022} on Lagrangian chaos in stochastic fluid mechanics. Our primary contribution is establishing an instance where Lagrangian chaos can be proven for a fluid flow with supercritical energy, a key characteristic of 3D fluid dynamics.
For the 3D PEs, establishing the existence of the top Lyapunov exponent is already a challenging task. We address this difficulty by deriving new estimates for the invariant measures of the 3D PEs, which capture the anisotropic smoothing in the dynamics of the PEs.
As a by-product of our results, we also obtain the first uniqueness result for invariant measures of stochastic 3D PEs.
\end{abstract}

\maketitle
\setcounter{tocdepth}{1}

\vspace{-0.85cm}

\tableofcontents
\vspace{-1.3cm}

\section{Introduction}
\label{s:intro}
In this paper, we consider the three-dimensional stochastic primitive equations (PEs) with additive noise on $\Dom=\T^2_{x,y}\times (0,1)$:
\begin{equation}
\label{eq:primitive_full}
\left\{
\begin{aligned}
\partial_t v_t+ (u_t\cdot \nabla)v_t  &= -\nabla_{x,y} p_t +\Delta v_t    + Q\dot{W}_t, \quad &\text{ on }&\Dom,\\
\partial_{z} p_t&=0, \quad &\text{ on }&\Dom,\\
\nabla\cdot u_t&=0, \quad &\text{ on }&\Dom,\\
v_0&=v, \quad &\text{ on }&\Dom,
\end{aligned}
\right.
\end{equation}
where $\T^2_{x,y}$ denotes the two-dimensional torus in the horizontal variables $(x,y)$, $\nabla_{x,y}=(\partial_x,\partial_y)$ is the gradient in the horizontal variables, $u_t=(v_t,w_t)$ the unknown velocity field, $v_t: \O\times \Dom\to \R^2$ and $w_t: \O\times \Dom\to \R$ are the corresponding horizontal and vertical components, respectively; while $p_t: \O\times\T^2_{x,y}\to \R$ denotes the pressure. Finally, $Q\dot{W}_t$ is a given  (nondegenerate) white in time and colored in space Gaussian noise on a given filtered probability space (see Subsection \ref{ss:solutions_def} for details). Throughout the manuscript, the above system is complemented with the following boundary conditions:
\begin{align}
\label{eq:primitive_full1}
v_t (\cdot,0)=v_t (\cdot,1)\quad \text{ and }\quad w_t(\cdot,0)=w_t(\cdot,1)=0 \quad \text{ on }\ \T^2_{x,y}.
\end{align}
Hence, the horizontal velocity field $v_t$ is periodic in all directions, while the Dirichlet condition on $w_t$ ensures the impermeability of the boundary $\partial\Dom$.
In particular, $v_t$ can be considered as a map on the three-dimensional torus $\T^3$.

\smallskip
The PEs \eqref{eq:primitive_full} are one of the fundamental models for geophysical flows used to describe oceanic and atmospheric dynamics, where, in a typical region, the vertical scale is much smaller than the horizontal ones. 
In such domains, the PEs well approximate the 3D Navier-Stokes equations. In the deterministic case (i.e., $Q=0$),
rigorous justifications by means of the small aspect ratio limit can be found in \cite{MR1884725,KGHHKW20,LT19}. Detailed information on the geophysical background for the various versions of the deterministic PEs can be found in, e.g., \cite{K21_global,Ped,Vallis06}. 

The PEs have been widely studied in recent years, therefore, it is impossible to give a complete overview here. However, we highlight some key references. In the deterministic setting, we point to the seminal work by Lions, Temam, Wang \cite{LiTeWa1,LiTeWa2}, as well as the breakthrough result by Cao and Titi \cite{CT07} on global well-posedness in $H^1$. Extensions in various directions can be found in \cite{GGHHK20_bounded,GGHK21_scaling,HH20_fluids_pressure,HK16,Ju17,Kukavica_2007}.
In the stochastic setting, global well-posedness was first obtained by Debussche, Glatt-Holtz, Temam and Ziane in  \cite{Debussche_2012} (see also \cite{DEBUSSCHE20111123}). 
For further developments, the reader is referred to \cite{A23_primitive3, primitive1, primitive2, BS21, debussche2025stochastic, hu2025local} and the references therein.

\smallskip

In contrast to well-posedness, the dynamical properties of solutions to PEs are far less understood. The primary goal of this manuscript is to prove that the Lagrangian trajectories associated with the solution $u_t$ of the stochastic PEs \eqref{eq:primitive_full} are \emph{chaotic} in the sense that they are highly sensitive to initial conditions. This phenomenon, often referred to as \emph{Lagrangian chaos}, is established in the following.

\begin{theorem}[Lagrangian chaos PEs -- Informal statement]
\label{t:intro}
Under suitable smoothness and nondegeneracy assumptions on the noise $QW_t$, there exists a \emph{deterministic} constant $\lambda_+>0$ such that for all sufficiently regular initial data $v$ the (random) Lagrangian flow $\phi^t$ associated with the stochastic {\normalfont{PEs}} \eqref{eq:primitive_full} is chaotic, i.e., 
\begin{equation}
\label{eq:lagrangian_flow}
\frac{\dd}{\dd t}\phi^t(\x)= u_t(\phi^t(\x)), \qquad \phi^0(\x)=\x\in \Dom,
\end{equation}
satisfies
\begin{equation}
\label{eq:positive_lyap_exp}
\lim_{t\to \infty} \frac{1}{t}\log|\nabla \phi^t(\x)|=\lambda_+>0  \text{ for all $\x\in \Dom$ with probability }1.
\end{equation}
\end{theorem}

The rigorous formulation of the above is given in Theorem \ref{t:positive_lyapunov_exponent}. 
Equation \eqref{eq:positive_lyap_exp} shows that nearby particles are exponentially separated by the flow $\phi^t$.
The exponential expansion rate $\lambda_+$ is usually called the (top) \emph{Lyapunov exponent} of $\phi^t$.
Consequences on scalar turbulence of Theorem \ref{t:intro} or, more precisely, a version of \eqref{eq:positive_lyap_exp} with $\nabla \phi^t(\x)$ replaced by $(\nabla \phi^t (\x))^{-\top}$, are given in Corollary \ref{cor:gradient_expansion}.

Chaotic behaviours are commonly observed in the study of the ocean and play a crucial role in its dynamics, see e.g.,  \cite{abraham2002chaotic, brown1991ocean, prants2014chaotic, sivakumar2004chaos, yang1996chaotic,yang1997three}. To the best of our knowledge, Theorem \ref{t:intro} provides the first rigorous result concerning the chaotic behaviour of a \emph{three-dimensional} oceanic model, such as the stochastic PEs \eqref{eq:primitive_full}.

\smallskip

The proof of Theorem \ref{t:intro} builds upon the foundational work of Bedrossian, Blumenthal, and Punshon-Smith \cite{BBPS2022} on Lagrangian chaos in stochastic fluid dynamics models. To avoid interrupting the presentation, we postpone to Subsection \ref{ss:discussion_intro} the discussion of the novelties of this work and its connection to existing results.

\smallskip

A key element in the proof of Theorem \ref{t:intro} is the existence of a \emph{unique} invariant measure for the stochastic PEs \eqref{eq:primitive_full}. Invariant measures play a crucial role in understanding the long-time behaviour of solutions to a given SPDE \cite{DPZ_ergodicity}, and their uniqueness has significant implications, including ergodicity and the complete characterisation of long-time dynamics through the invariant measure.
For definitions, the reader is referred to Subsection \ref{ss:global_well_statement}. While the existence of ergodic measures for PEs was established by Glatt-Holtz, Kukavica, Vicol and Ziane in \cite{GHKVZ14}, their uniqueness remains an open problem (see the comments below Theorem C in \cite{BGN23} and the recent work \cite{lin2025averaging}). 
The main obstacle in establishing uniqueness is the weak bounds available for PEs in their natural state space $H^\sigma$ with $\sigma\geq 1$, cf., \cite[Section II]{GHKVZ14} and \cite[Theorem 3.6 and Remark 3.10]{primitive2}. Here, the weakness is with respect to integrability in the probability space and invariant measures.

In this manuscript, by leveraging on new high-order energy estimates for \eqref{eq:primitive_full} established in Section \ref{s:well_posedness_Hsigma}, we are able to overcome these challenges. The following is an informal version and a special case of Proposition \ref{prop:regularity_invariant_measure} and Theorem \ref{t:uniqueness_regularity_invariant}.

\begin{theorem}[Uniqueness and regularity of invariant measures -- Informal statement]
\label{t:uniqueness_invariant_intro}
The following assertions hold:
\begin{enumerate}[{\rm(1)}]
\item\label{it:uniqueness_invariant_intro1} If $QW_t$ takes values in $H^\sigma(\T^3;\R^2)$ a.s.\ for some $\sigma\geq 2$, then all invariant measure $\mu$ for the stochastic {\normalfont{PEs}} \eqref{eq:primitive_full} are \emph{concentrated} on 
$
\textstyle\bigcap_{\sigma'<\sigma+1}H^{\sigma'}(\T^3;\R^2).
$
\item  If $QW_t$ takes values in $H^\sigma(\T^3;\R^2)$ a.s.\ for some $\sigma\geq 2$ and is nondegenerate, then there exists a \emph{unique} invariant measure for the stochastic {\normalfont{PEs}} \eqref{eq:primitive_full}.
\end{enumerate}
\end{theorem}

The regularity result for invariant measures of \eqref{eq:primitive_full} in Theorem \ref{t:uniqueness_invariant_intro}\eqref{it:uniqueness_invariant_intro1} is optimal from a parabolic point of view. 
From the above, it follows that $\mu$ is concentrated on smooth functions, provided $QW_t$ is smooth in space. The latter fact answers positively to the question posed below
\cite[Theorem 1.9]{GHKVZ14}, at least in the case of additive noise. Further details on this point are given below Proposition \ref{prop:regularity_invariant_measure}.

To the best of our knowledge, Theorem \ref{t:uniqueness_invariant_intro} presents the \emph{first} uniqueness result for invariant measures of PEs \eqref{eq:primitive_full}. Our proof of uniqueness follows the well-established strategy involving strong Feller and irreducibility properties of transition Markov kernels, as used in \cite{FlaMas1995} for the 2D Navier-Stokes equations. For the relevant terminology, we refer the reader to Subsection \ref{ss:global_well_statement} or \cite{DPZ_ergodicity}.
In our setting, the high-order energy estimates established in Section \ref{s:well_posedness_Hsigma} play a central role. It is worth noting that, due to the use of the strong Feller approach, we require that the noise acts on \emph{all} Fourier modes, cf., \eqref{eq:generalization2} and Subsection \ref{sss:probabilistic_set_up}.
Establishing the uniqueness of invariant measures in the case of degenerate noise--where some Fourier modes are not randomly forced--appears to be extremely challenging due to the above-mentioned weak energy bounds. In particular, known methods, such as asymptotic couplings \cite{GHMR17,hairer2011asymptotic} (see \cite[Subsection 3.2]{GHMR17} for the two-dimensional variant of \eqref{eq:primitive_full}), asymptotic strong Feller \cite{HM06_annals}, or the $e$-property \cite{MR2663632}, do not seem to be applicable in this context.

\smallskip

Below, we focus on describing the challenges in the proof of Theorem \ref{t:intro} and highlight its connections with existing literature. In Subsections \ref{ss:future_challenges} and \ref{ss:overview}, we describe future challenges and provide an overview of the manuscript, respectively. 
The main results of this work are collected in Section \ref{ss:global_unique_erg}, where we also present additional findings. Subsection \ref{ss:proofs_further_results} offers a detailed discussion of the proofs.

\subsection{Background and contribution}
\label{ss:discussion_intro}
In the influential work \cite{BBPS2022}, a sufficient condition was established for the positivity of the top Lyapunov exponent
\begin{equation}
\label{eq:top_Lyapunov_exponent}
\lim_{t\to \infty} \frac{1}{t}\log|\nabla \phi^t(\x)|,
\end{equation}
where $\phi^t$ is the flow generated by the solution of a stochastic fluid model. The same authors then used this result as a foundation to study mixing and enhanced dissipation for passive scalars advected by a stochastic fluid \cite{BBPS2022_PTRF,BBPS2022_AOP}. These results were subsequently employed to prove a cumulative version of the Batchelor prediction in passive scalar turbulence \cite{BBPS2022_CPAM}.
Although the framework presented in the above-mentioned references is rather general, several questions remain open, particularly regarding the applications of these results to more realistic fluid models, such as the 3D PEs \eqref{eq:primitive_full}, which, as noted earlier, are commonly used in the study of oceanic and atmospheric dynamics. Further discussion of these issues can be found in Subsection \ref{ss:future_challenges} below and in \cite[Subsection 1.4]{BBPS2022_CPAM}. 
In particular, fluid models usually manifest the following complications: 
\begin{itemize} 
\item Energy supercriticality. 
\item Presence of boundary conditions. 
\end{itemize} 
Recall that energy supercriticality is typical of three-dimensional models and produces weak $\O$-bounds for the solution of the fluid flow in the well-posedness space. Let us emphasise that the latter also appears in two-dimensional models with boundary noise (see Subsection \ref{ss:future_challenges} for more details). 

The stochastic PEs \eqref{eq:primitive_full} serve as a prototypical example of globally well-posed 3D fluid models in which both of the aforementioned conditions are simultaneously satisfied. While the boundary conditions are explicitly given in  \eqref{eq:primitive_full1}, the energy supercriticality can be seen from a scaling analysis; see \cite[Subsection 1.2]{A23_primitive3} for details. It is important to note that the PEs \eqref{eq:primitive_full} possess the scaling of the 3D Navier-Stokes equations, and thus, they are energy supercritical.
The energy supercriticality of the PEs is essentially reflected in the weak bounds mentioned above Theorem \ref{t:uniqueness_invariant_intro}, and therefore related difficulties are already present at the level of the fluid equation. Meanwhile, the presence of the impermeability condition \eqref{eq:primitive_full1} forces the flow $\phi^t$ in \eqref{eq:lagrangian_flow} to live in the non-compact space $\Dom$. 
The latter fact will be particularly relevant when dealing with ergodic measures of the so-called Lagrangian process $(u_t,\phi^t(\x))_t$, as compactness methods (e.g., Krein-Milman's theorem) cannot be applied to ensure ergodicity of the invariant measure $\mu\otimes \Leb$ (here, $\mu$ is as in Theorem \ref{t:uniqueness_invariant_intro}). 
To address this difficulty, we establish auxiliary results on Markov semigroups in Appendix \ref{app:useful_results}.

\smallskip

The essential ingredients in the application of the general framework in \cite[Section 4]{BBPS2022} to \eqref{eq:primitive_full} can be (roughly) summarised as follows: 
\begin{enumerate}
\item\label{it:top_positive_1} An integrability condition for the ergodic measures $\mu$ of the fluid model.
\item\label{it:top_positive_2} Smoothing of the transition kernels of the Lagrangian process
$(u_t,\phi^t(\x))_t$.
\item\label{it:top_positive_11} Ergodicity of $\mu\otimes \Leb$ for the Lagrangian process $(u_t,\phi^t(\x))_t$.
\item\label{it:top_positive_3} Approximate controllability results for:
\begin{itemize}
\item The Jacobian process
$
\big(u_t,\phi^t(\x), \nabla \phi^t(\x)\big)_t .
$
\item The projective process
$
\big(u_t,\phi^t(\x), \nabla \phi^t(\x)\xi/ |\nabla \phi^t(\x)\xi|\big)_t .
$
\end{itemize}
\end{enumerate} 
In the above list, $\mu$ is the unique invariant measure for the stochastic PEs \eqref{eq:primitive_full} provided by Theorem \ref{t:uniqueness_invariant_intro}, $\x\in\Dom$ and $\xi\in\R^3\setminus\{0\}$.

The precise details are provided in Subsection \ref{ss:proofs_further_results}, and the reader is referred to \cite[Proposition 4.17]{BBPS2022} for the general framework.
We now provide a brief discussion of the above points and highlight the main difficulties encountered in their proofs in case $u_t$ is the solution to the stochastic PEs \eqref{eq:primitive_full}.

\smallskip

\eqref{it:top_positive_1}: 
According to the multiplicative ergodic theorem \cite[Theorem 3.4.1]{A98_Random}, the \emph{existence} of the top Lyapunov exponent \eqref{eq:top_Lyapunov_exponent} is (essentially) equivalent to the following condition on the ergodic measure $\mu$ of \eqref{eq:primitive_full}:
\begin{equation}
\label{eq:integrability_intro}
\E_{\mu}\int_{\T^3}|\nabla u(\x)|\,\dd \x<\infty,
\end{equation}
where $\E_{\mu}$ is the expected values with respect to $\mu$.
Unlike the 2D Navier-Stokes equations (or the hyperviscous models discussed in \cite{BBPS2022}), where the condition \eqref{eq:integrability_intro} follows directly from the standard energy balance, proving \eqref{eq:integrability_intro} for the stochastic PEs \eqref{eq:primitive_full} is far from trivial. We now clarify the main difficulty behind the proof of \eqref{eq:integrability_intro} in the context of PEs. First, note that \eqref{eq:primitive_full} does \emph{not} include an evolution equation for the third component of the velocity field 
$u$, which we denoted by $w$. The latter can, however, be recovered from the divergence-free condition in \eqref{eq:primitive_full} and the boundary condition \eqref{eq:primitive_full1}:
\begin{equation}
\label{eq:def_w_intro}
w_t(\x)=-\int_0^{z}\nabla_{x,y} \cdot v_t(x,y,z')\,\dd z', \ \quad \x=(x,y,z)\in \Dom.
\end{equation}
Here, $\nabla_{x,y}\cdot $ is the divergence operator with respect to the horizontal variables $(x,y)$; see below for the notation.
Consequently, \eqref{eq:integrability_intro} requires the integrability of quantities involving \emph{second-order} derivatives of the horizontal velocity $v_t$. However, in the existing literature, only logarithmic bounds on second-order derivatives are available; see \cite[Theorem 1.9]{GHKVZ14}. The latter estimates are too weak to ensure \eqref{eq:integrability_intro}, therefore ensuring the existence of the Lyapunov exponent \eqref{eq:top_Lyapunov_exponent}.

In this work, we address these issues by using $L^p$-estimates and exploiting the \emph{anisotropic} smoothing effect associated with the nonlinearity $w\partial_z v$, where $w$ is as in \eqref{eq:def_w_intro}. 
For further details, the reader is referred to Theorem \ref{t:integrability_MET} and the comments below it. Additionally, we obtain \eqref{eq:integrability_intro} with spatial integrability $L^1$ replaced by $L^p$ for any $p<\frac{4}{3}$. We also expect that the $L^2$-case of \eqref{eq:integrability_intro} does \emph{not} hold.

\smallskip

\eqref{it:top_positive_2}: 
There are several ways to capture the smoothing effects of transition kernels of Markov semigroups. In this work, as in \cite{BBPS2022}, we will demonstrate that the Lagrangian process induces a \emph{strong Feller} semigroup on its natural state space; see Proposition \ref{prop:strong_feller_projective} for the precise statement. The reader is referred to either Subsection \ref{ss:proofs_further_results} or \cite[Chapter 4]{DPZ_ergodicity} for the definition of strong Feller.

As in \cite{BBPS2022}, we assume that the noise $Q\dot{W}_t$ acting on the horizontal velocity is nondegenerate in Fourier space (see Assumption \ref{ass:Q}), and thus the main challenge in achieving strong Feller for the Lagrangian process $(u_t,\phi^t(\x))_t$ is the lack of noise in the flow equation \eqref{eq:lagrangian_flow}. In other words, there is no direct source of noise in \eqref{eq:lagrangian_flow}. However, randomness will be transferred to $\phi^t(\x)$ from $v_t$ on which the noise acts directly, see \eqref{eq:primitive_full}. 
Compared to the fluid models considered in \cite{BBPS2022}, our situation is even more degenerate because, in the Lagrangian flow equation \eqref{eq:lagrangian_flow}, the third component is driven by 
$w_t$, which is \emph{not} directly driven by the noise (see \eqref{eq:def_w_intro} and the preceding discussion).
To address this, in Section \ref{s:strong_Feller}, we will investigate the algebraic properties of the system described by \eqref{eq:primitive} and \eqref{eq:lagrangian_flow} through an H\"ormander-type condition. The details are provided in Subsection \ref{ss:nondegeneracy}.

Before discussing the subsequent item, let us comment on other possible choices of smoothing for the Markov kernel of the Lagrangian process. 
In the recent work \cite{CR24_degenerate}, the authors proved that in the context of the 2D Navier-Stokes equations the results in \cite{BBPS2022} (as well as subsequent works \cite{BBPS2022_PTRF,BBPS2022_AOP,BBPS2022_CPAM}) can be extended to the case where the noise acts on only \emph{finitely} many Fourier modes (i.e., degenerate noise). This extension is achieved through the use of the asymptotic strong Feller property, which was introduced in \cite{HM06_annals} and used to show unique ergodicity of the 2D Navier-Stokes equations with degenerate noise. However, as noted below Theorem \ref{t:uniqueness_invariant_intro}, it remains unclear how to establish asymptotic strong Feller for the stochastic PEs \eqref{eq:primitive_full} itself in the case of degenerate noise. Therefore, the extension presented in \cite{CR24_degenerate} of the results in \cite{BBPS2022} to the case of degenerate noise does not currently seem applicable in the context of the stochastic PEs \eqref{eq:primitive_full}.

\smallskip

\eqref{it:top_positive_11}: 
Although the invariance of the measure $\mu \otimes \Leb$ under the Lagrangian dynamics $(u_t, \phi^t(\x))_t$ follows directly from the divergence-free condition on $u_t$, the question of ergodicity is more delicate. Unlike in \cite{BBPS2022}, due to the non-compactness of $\Dom$, it is difficult to obtain uniqueness of invariant measures for the Lagrangian process $(u_t,\phi^t(\x))_t$, and hence, ergodicity of $\mu\otimes \Leb$. 
Indeed, in \cite{BBPS2022}, the compactness of the underlined domain allows one to invoke the Krein-Milman theorem along with the uniqueness and extremality of ergodic measures (see, e.g., \cite[Proposition 3.2.7]{DPZ_ergodicity}). 
Moreover, this issue cannot be solved by considering the flow $\phi^t$ on the compact set $\overline{\Dom}$. Indeed, in the latter situation, it is possible to construct ergodic measures for the Lagrangian process that are supported on $\supp \mu \times \partial \Dom$. Combined with the Krein-Milman theorem and the invariance of $\mu \otimes \Leb$, this implies the existence of multiple ergodic measures for the Lagrangian process $(u_t, \phi^t(\x))_t$ where $\phi^t$ is considered on $\overline{\Dom}$.
Nevertheless, by employing a topological argument along with the strong Feller property of the process $(u_t, \phi^t)_t$ with $\x\in \Dom$, we can still establish ergodicity of $\mu\otimes \Leb$ via the results in Appendix \ref{app:useful_results}.

\smallskip

\eqref{it:top_positive_3}: The approximate controllability of the Jacobian and projective processes, as formulated in Proposition \ref{prop:approximate_controllability}, is also influenced by the noise degeneracy in the vertical component 
$w_t$ of the velocity field $u_t$, as described in \eqref{eq:def_w_intro}. To address this challenge, we construct ad-hoc flows that control the evolution of the PEs. Our approach extends the ``shear flows'' method used in \cite[Section 7]{BBPS2022} to the context of PEs. For further details, the reader is referred to Section \ref{s:irreduc_controllability}.

\smallskip

To conclude, let us emphasise that, as in \cite{BBPS2022}, we will obtain a stronger result than \eqref{eq:positive_lyap_exp} in Theorem \ref{t:intro}. Specifically, using the random multiplicative ergodic theorem \cite[Theorem III.1.2]{K86_ergodic_theory} (also used in \cite{BBPS2022}, see Proposition 3.16 therein), one can derive from Theorem \ref{t:intro} the following stronger result:
For all $\xi\in \R^3\setminus\{0\}$,
\begin{equation}
\label{eq:intro_expansion}
\lim_{t\to \infty} \frac{1}{t}\log|\nabla \phi^t(\x)\xi|=\lambda_+>0  \text{ for all $\x\in \Dom$ with probability }1.
\end{equation}
Hence, \eqref{eq:intro_expansion} shows the trajectory of the flow $\phi^t$ induced by the stochastic PEs \eqref{eq:primitive_full} expands exponentially \emph{all} (deterministic) directions $\xi$.
By means of the random multiplicative ergodic theorem, \eqref{eq:intro_expansion} follows from Theorem \ref{t:intro} if we obtain
\begin{enumerate}
\setcounter{enumi}{4}
\item\label{it:top_positive_5} Uniqueness of invariant measures for the projective process, whose projection onto the first two components coincides with $\mu \otimes \Leb$.
\end{enumerate}
As in \eqref{it:top_positive_11}, we are unable to prove uniqueness of the invariant measure for the projective process in general, but only within the class of measures whose projection onto the first two components coincides with $\mu \otimes \Leb$. This limitation again stems from the non-compactness of $\Dom$.
The claim \eqref{it:top_positive_5} will be established by proving uniqueness of ergodic measures for the projective process, together with the ergodicity of $\mu \otimes \Leb$ for the Lagrangian dynamics. The uniqueness of ergodic measures is obtained via the strong Feller property of the projective process, as shown in Propositions \ref{prop:strong_feller_projective} and \ref{prop:weak_irreducibility}.
We emphasise that the strong Feller property established here is a stronger result than what is required in \eqref{it:top_positive_2} from the previous list.

\subsection{Future challenges}
\label{ss:future_challenges}
In this subsection, we discuss potential future extensions of our results, with a particular focus on PEs and, more generally, geophysical flows.
\begin{itemize}
\item \emph{Unique ergodicity with degenerate noise.} The extension of the uniqueness result for invariant measures in Theorem \ref{t:uniqueness_invariant_intro} (or, more precisely, Theorem \ref{t:uniqueness_regularity_invariant}) to the case where 
the noise $Q\dot{W_t}$ acts only on finitely many Fourier modes seems at the moment out of reach. As noted above, the primary challenge, in this case, arises from the relatively weak energy bounds for the stochastic PEs \eqref{eq:primitive_full}, which might require a refinement of the known strategies for proving the uniqueness of invariant measures in the degenerate case (for the latter see the comments below Theorem \ref{t:uniqueness_invariant_intro}).

\item \emph{Mixing and enhanced dissipation of passive scalars.} 
As shown in the works \cite{BBPS2022_PTRF,BBPS2022_AOP}, the positivity of the top Lyapunov exponent of the flow map associated with a stochastic fluid model often implies mixing and enhanced dissipation results for passive and diffusive scalars, respectively. It is important to note that these results are stronger than the exponential growth for passive scalars presented in Corollary \ref{cor:gradient_expansion}. While the positivity of the top Lyapunov exponent for the induced flow map is a key element in \cite{BBPS2022_PTRF,BBPS2022_AOP}, these results also rely on the existence of Lyapunov functions that provide quantitative control over high deviations of the velocity
$u_t$, as discussed in \cite[Subsection 2.4]{BBPS2022_AOP}. The construction of such Lyapunov functions remains an open problem for stochastic PEs \eqref{eq:primitive_full}.

\item \emph{Other geophysical flows: Boundary noise.}
The interaction between the wind and the ocean surface plays a crucial role in ocean dynamics, see  \cite{gill2016atmosphere,Ped,pedlosky1996ocean}. In the case of highly oscillating wind, and under the assumption of a time-scale separation, the wind's effect can be idealised as noise acting on the boundary. Recent well-posedness results for fluid dynamics models with boundary noise are presented in \cite{ABL24_boundary,AL24_boundary,MR4701784}. The physical relevance of extending results on Lagrangian chaos and scalar advection, as obtained in \cite{BBPS2022_PTRF,BBPS2022_AOP,BBPS2022_CPAM,BBPS2022}, to boundary-driven models was discussed in \cite[p. 1254]{BBPS2022_CPAM}.

Establishing Lagrangian chaos for a fluid model with white-in-time boundary noise presents significant challenges. Besides the presence of the boundary, even in a two-dimensional setting, the energy estimates are weak in the probability space (cf., \cite{AL24_boundary}), and therefore integrability condition \eqref{eq:integrability_intro} for the existence of the top Lyapunov exponent \eqref{eq:top_Lyapunov_exponent} is non-trivial. We expect that some of the techniques developed in this work could be applied to this context as well.

Finally, note that a colored-in-time boundary noise--such as replacing the white noise with an Ornstein-Uhlenbeck process--can simplify the proof of the condition \eqref{eq:integrability_intro}. However, in many situations, the white-in-time noise often leads to important insights for the colored-in-time and space case; see \cite[Subsection 2.12]{CR24_degenerate} and \cite[p. 249]{BBPS2022_AOP}.
\end{itemize}

\subsection{Overview}
\label{ss:overview}
This manuscript is organised as follows.

\begin{itemize}
\item Section \ref{s:Preliminaries} introduces preliminary concepts for PEs, including the reformulation of PEs as an evolution equation for the horizontal velocity $v_t$ and the definition of $H^\sigma$-solutions to the stochastic PEs \eqref{eq:primitive_full}, along with the probabilistic framework.
\item Section \ref{ss:global_unique_erg} presents the main results of this paper: Unique ergodicity, Lagrangian chaos, and scalar advection by stochastic PEs in Subsection \ref{ss:global_well_statement} and Section \ref{ss:main_results}, with an overview of the proofs in Section \ref{ss:proofs_further_results}.
\item Section \ref{s:integrability_inv_measure_u} proves the integrability condition \eqref{eq:integrability_intro} for invariant measures of the stochastic PEs \eqref{eq:primitive_full}.
\item Section \ref{s:well_posedness_Hsigma} proves Theorem \ref{t:uniqueness_invariant_intro} and complementary global well-posedness results for the stochastic PEs in high-order Sobolev spaces.
\item Section \ref{s:strong_Feller} establishes the strong Feller property for the projective processes.
\item Section \ref{s:irreduc_controllability} proves approximate controllability for the Jacobian and projective processes, along with weak irreducibility for the projective processes.
\item Appendix \ref{app:useful_results} provides complementary results on Markov semigroups.
\end{itemize}

\subsubsection*{Notation}
Here, we collect the main notations used in this manuscript. For two quantities $x$ and $y$, we write $x\lesssim y$ (resp.\ $x\gtrsim y$), if there exists a constant $C$ independent of the relevant parameters such that $x\leq Cy$ (resp.\ $x\geq C y $). We write $x\eqsim y$, whenever $x\gtrsim y$ and $y\lesssim x$.

An element in the domain $\Dom$ is denoted by $\x=(x,y,z)$ where $(x,y)\in \T^2_{x,y}$ and $z\in (0,1)$ denotes the horizontal and the vertical directions, respectively. For $\x=(x,y,z)\in \Dom$, we denote by $\dist(\x)$ its distance from the boundary:
$$
\dist (\x)\stackrel{{\rm def}}{=}z \wedge (1-z).
$$
In the above, $\T=\R/\Z$ and $\T^2_{x,y}=\T\times \T$ and the subscript `$x,y$' serving solely to indicate the corresponding coordinates. 
A similar decomposition is only employed on $\T^3= \T^2_{x,y}\times \T_z$ and $\T^2_{x,y}=\T_x\times \T_y$. 
The above splitting also induced a decomposition of the standard operators $\nabla=(\partial_x,\partial_y,\partial_z)$ into the horizontal and vertical components given by $\nabla_{x,y}=(\partial_x,\partial_y)$ and $\partial_z$, respectively. In particular, for a vector field $v=(v_x,v_y)$ on either $\Dom$ or $\T^3$, the operator $\nabla_{x,y}\cdot$ denotes the horizontal divergence and the action on $v$ is defined as
$
\nabla_{x,y}\cdot v =\partial_x v_x+\partial_y v_y.
$

Let $H$ and $K$ be Hilbert spaces. $\Borel(H)$ denotes the Borel sets of $H$, and 
$\calL_2(H_1,H_2)$ stands for the space of Hilbert-Schmidt operators from $H$ to $K$. 

Throughout this manuscript, $(\Omega, \mathscr{A},(\mathscr{F}_t)_{t\geq 0}, \P)$ denotes a filtered probability space. Further details on the probability setting are given in Subsection \ref{ss:solutions_def} below. As usual, we write $\E$ for the expectation on $(\Omega, \mathscr{A}, \P)$ and $\mathscr{P}$ for the progressive $\sigma$-field. A process $\phi:[0,\infty)\times \O\to H$ is progressively measurable if $\phi|_{[0,t]\times \O}$ is $\Borel([0,t])\otimes \F_t$ measurable for all $t\geq 0$, where $\Borel$ is the Borel $\sigma$-algebra on $[0,t]$.

Finally, we give a list of some symbols which will appear in the manuscript.
\begin{itemize}
\item Hydrostastic Helmholtz projection $\p$, $\q=\Id-\p$ and associated spaces $\Ls^2(\T^3)$, $\Ls^2_0(\T^3)$, $\Hs^\rho(\T^3)$ and $\Hs^\rho_0(\T^3)$ of vector-fields -- Subsection \ref{ss:reformulation}.
\item $\z$, $\plane$, $\g_{(\kk,\ell)}$ and $e_\kk$ -- Subsection \ref{sss:probabilistic_set_up}. 
\item Auxiliary processes $\x_t$, $\A_t$, $\xi_t$ and $\wc{\xi}_t$ -- Subsection \ref{ss:proofs_further_results}.
\item Augmented processes $V_t$ and $V_t^{(r)}$, and corresponding state space $\HL$ -- Subsection \ref{ss:strong_feller_via_cut_off}.
\item Projective Markov kernel or semigroups $\Pl_t$, $P_t$ and $P_t^{(r)}$ -- Subsection \ref{ss:proofs_further_results} and \ref{ss:strong_feller_via_cut_off}, respectively.
\item Malliavin derivatives $\D$ and $\D_g$, Skorohod operator $\delta$, and  partial Malliavin matrix $\Ma$ -- Subsection \ref{ss:Malliavin_calculus} and \ref{ss:gradient_estimates}, respectively. 
\end{itemize}

\section{Preliminaries and set-up}
\label{s:Preliminaries}
In this section, we gathered the main assumptions on the noise and definitions.

\subsection{Reformulation and function analytic setting}
\label{ss:reformulation}
As is well-known in the study of PEs, they can be conveniently reformulated in terms of the only unknown $v$. 
To see this, note that the boundary condition at $w_t=0$ on $\T_{x,y}^2\times \{0\}$ in \eqref{eq:primitive_full}, combined with the incompressibility condition $\nabla_{x,y}\cdot v_t +\partial_z w_t=0$ leads to
\begin{equation}
\label{eq:def_w}
w_t=[w(v_t)] \stackrel{{\rm def}}{=}
-\int_{0}^{z} (\nabla_{x,y}\cdot v_t)(x,y,z')\,\dd z', 
\end{equation}  
for $\x=(x,y,z)\in \T^3$. Moreover, the boundary condition $w_t=0$ on $\T_{x,y}^2\times \{1\}$ in \eqref{eq:primitive_full} in \eqref{eq:primitive_full} also yields
\begin{equation}
\label{eq:incompressibility_v}
\int_{\T_z} \nabla_{x,y}\cdot v_t(\cdot,z)\,\dd z=0 \ \ \text{ on }\T^2_{x,y},
\end{equation}
that is a divergence free condition for the \emph{barotropic} mode $\int_{\T_z} v_t(\cdot,z)\,\dd z=0$.

Note also that the second of \eqref{eq:primitive_full} implies that the pressure $p_t=p_t(x,y)$ is two-dimensional. As it does not depend on the vertical variable $z$, the pressure $p_t$ is often referred to as \emph{surface pressure}.

Summarising the above discussion, the PEs \eqref{eq:primitive_full} can be reformulated in terms of the horizontal component of the velocity field  $v:\R_+\times \O\times \T^3\to \R^2$ and the surface pressure $p:\R_+\times \O\times \T_{x,y}^2 \to \R$ as
\begin{equation}
	\label{eq:primitive_p}
\left\{
\begin{aligned}
		&\partial_t v_t+ (v_t\cdot \nabla_{x,y})v_t +w(v_t)\partial_z v_t 
		=-\nabla_{x,y} P_t+  \Delta v_t   + Q\dot{W}_t,  &\text{on }&\T^3,\\
&
\textstyle\int_{\T_z}\nabla_{x,y}\cdot v_t(\cdot,z)\,\dd z=0,  &\text{on }&\T^2_{x,y},\\
&v_0=v,  &\text{ on }&\T^3.
\end{aligned}
\right.
\end{equation}

Next, we reformulate the problem only in terms of $v$ as the incompressibility condition, in the second of \eqref{eq:primitive_p}, solely determines $p$. To this end, we employ the \emph{hydrostatic Helmholtz projection}.
For each $f\in L^2(\T^2;\R^2)$, let $\Psi_f\in H^1(\T^2)$ be the unique solution of the following elliptic problem:
\begin{equation}
\label{eq:def_q_xy}
\textstyle
\Delta_{x,y} \Psi_f =\nabla_{x,y}  \cdot f\  \text{ in }\D'(\T_{x,y}^2)
\quad \text{ and } \quad \int_{\T_{x,y}^2} \Psi_f \,\dd x\,\dd y=0.
\end{equation}
Set $\q_{x,y} f\stackrel{{\rm def}}{=}\nabla_{x,y}\Psi_f$. 
One can readily check that $\q_{x,y}\in \calL(L^2(\T^2_{x,y};\R^2))$.
The hydrostatic Helmholtz projection $\p:L^2(\T^3;\R^2)\to L^2(\T^3;\R^2)$ is defined as
\begin{equation}
\label{eq:def_p_q}
\textstyle
\p f\stackrel{{\rm def}}{=} f- \q f \quad \text{ where }\quad \q f \stackrel{{\rm def}}{=} \q_{x,y}\Big[\int_{\T_z} f(\cdot,z)\,\dd z\Big],
\end{equation}
where $f\in L^2(\T^3;\R^2)$.
One can check that $\p$ and $\q$ can be restrcted or uniquely extended to orthogonal projections $H^\rho(\T^3;\R^2)\to H^\rho(\T^3;\R^2)$ for all $\rho\in\R$, which will still be denote by $\p$ and $\q$, respectively.

By applying the hydrostatic Helmholtz projection in the first of \eqref{eq:primitive_p} we can reformulate the primitive equations \eqref{eq:primitive_p} in terms of the only unknown $v$:
\begin{equation}
\label{eq:primitive}
\left\{
\begin{aligned}
\partial_t v_t &= \Delta v_t- \p[(v_t\cdot \nabla_{x,y})v_t +w(v_t)\partial_z v_t] + Q\dot{W}_t,\qquad & \text{on }&\T^3,\\
v_0&=v, \qquad &\text{ on }&\T^3.
\end{aligned}
\right.
\end{equation}
To see the equivalence between \eqref{eq:primitive} and \eqref{eq:primitive_p} note that 
$\Delta$ commutes with $\p$, $\p (Q\dot{W})=Q\dot{W}$ (see the assumptions in Subsection \ref{sss:probabilistic_set_up} below) and the incompressibility condition \eqref{eq:incompressibility_v} is satisfied as 
$\int_{\T_z}\nabla_{x,y}\cdot v(\cdot,z)\,\dd z=0$ 
and such a condition is preserved by the flow induced via \eqref{eq:primitive}. 

From now on, we will consider the system \eqref{eq:primitive} instead of \eqref{eq:primitive_full}.
More precisely, we will consider \eqref{eq:primitive} as a stochastic evolution equation on the Sobolev-type space $\Hs_0^\sigma(\T^3)$ with an appropriate $\sigma\geq 2$ depending on the smoothness of the noise. Here, $\Hs_0^\sigma(\T^3)$ is a Sobolev space of vector field on $\T^3$ with values in $\R^2$, mean zero and divergence-free barotropic mode, i.e., satisfying \eqref{eq:incompressibility_v}. More precisely, we let 
\begin{align}
\label{eq:def_Ls}
\Ls^2(\T^3)
&\stackrel{{\rm def}}{=} \Big\{f\in L^2(\T^3;\R^2)\,:\,   \nabla_{x,y}\cdot \Big(\int_{\T_z}f(\cdot,z)\,\dd z\Big)=0\text{ in }\D'(\T^2_{xy}) \Big\},\\
\label{eq:def_Ls_0}
\Ls^2_0(\T^3)&\stackrel{{\rm def}}{=}\Big\{f\in \Ls^2(\T^3)\,:\, \int_{\T^3} f(\x)\,\dd \x=0\Big\},\\
\label{eq:def_Ls_1}
\Hs_0^\rho(\T^3)&\stackrel{{\rm def}}{=}H^{\rho}(\T^3;\R^2)\cap \Ls_0^2(\T^3)\ \  \text{ for }\  \rho>0,
\end{align}
endowed with natural norms. Similarly, to \eqref{eq:def_Ls_1}, one defines $\Hs^\rho(\T^3)$ as $H^{\rho}(\T^3;\R^2)\cap \Ls^2(\T^3)$. Moreover, if no confusion seems likely, we use the shorthand notation $\Ls^2$, $\Ls^2_0$ and $\Hs^\rho$ in place of $\Ls^2(\T^3)$, $\Ls^2_0(\T^3)$ and $\Hs^\rho(\T^3)$, respectively.

\subsection{Probabilistic set-up and $H^\sigma$ solutions}
\label{ss:solutions_def}
Let $\Z^3_+\subseteq \Z^3$ be such that 
\begin{equation*}
\Z^3_+\cup \Z^3_- =\Z^3_0\qquad \text{ and }\qquad \Z^3_+\cap \Z^3_-=\emptyset,
\end{equation*}
where 
$\Z^3_-\stackrel{{\rm def}}{=}-\Z^3_+$ and $\Z^3_0\stackrel{{\rm def}}{=}\Z^3\setminus \{0\}$.  
Similar to \cite{BBPS2022,EM01_finite}, we define the following real Fourier basis on $\T^3$ (as above, $\T=\R/\Z$):
\begin{equation*}
e_{\kk}(\x)
\stackrel{{\rm def}}{=}2
\left\{
\begin{aligned}
&\cos (2\pi \kk\cdot \x)  &\ \ \  &\kk\in \Z^3_+\ \text{ and } \ \x\in \T^3,\\
&\sin (2\pi \kk\cdot \x)  &\ \ \  &\kk\in \Z^3_-\ \text{ and }\ \x\in \T^3.
\end{aligned}\right.
\end{equation*}
Let $\plane\stackrel{{\rm def}}{=}\{\kk\in \Z^3_0\,:\,k_z=0 \}$ where $k_z$ denotes the third component of the vector $\kk$, i.e., $\kk=(k_x,k_y,k_z)$.
Correspondingly, we define the set of indices
\begin{equation}
\label{eq:index}
\z \stackrel{{\rm def}}{=} \big[(\Z^3_0\setminus \plane) \times \{1,2\}\big]\cup \big[\plane\times \{1\}\big].
\end{equation}
Thus, an element of $\z$ is a pair $(k,\ell)$ where $k\in \Z^d_0$ and $\ell\in \{1,2\}$. The cartesian product between $\plane$  and the singleton $\{1\}$ in \eqref{eq:index} is only for notational convenience.

For $(\kk,\ell)\in\z$, we let
\begin{equation}
\label{eq:definition_gamma_kk_ell}
\g_{(\kk,\ell)}\stackrel{{\rm def}}{=}
\left\{
\begin{aligned}
&a_{\ell}  &\qquad \text{ if }\ &(\kk,\ell)\in (\Z_0^3\setminus \plane)\times \{1,2\},\\
 &\sign(\kk)\kk^{\perp}/|\kk| &\qquad\text{ if }\  &(\kk,\ell)\in  \plane\times \{1\},
\end{aligned}
\right.
\end{equation}
where $\kk^\perp=(k_y,-k_x)^\top$ for $\kk\in \plane$ and $a_1=(1,0)^\top$, $a_2=(0,1)^\top$ denotes the standard basis of $\R^2$, and $\sign(\kk)=1 $ if $  \kk\in \Z^3_+$  and $\sign(\kk)=-1$ otherwise.
In particular, $\g_{(\kk,\ell)}=\g_{(-\kk,\ell)}$.
One can check that 
$$(\g_{(\kk,\ell)}e_{\kk})_{(\kk,\ell)\in \z}$$ constitutes an complete orthonormal basis of $\Ls_0^2(\T^3)$.

\smallskip

Before proceeding, we comment on the anisotropic behaviour of the index set $\z$. 
The lack of a second component for indices in  
$\plane$ reflects the fact that the deterministic 3D PEs reduce to the 2D Navier-Stokes when the initial data is two-dimensional (corresponding to fields of the form $\g e_{\kk}$ with $\g\cdot (k_x,k_y)=0$).

\subsubsection{$H^\sigma$-solutions}
Throughout the manuscript, $(W_t)_t$ denotes an $\Ls_0^2(\T^3)$-cylindri\-cal Brownian motion on a given filtered probability space $(\O,\mathscr{A},(\F_t)_{t\geq 0},\P)$, see e.g., \cite[Definition 2.11]{AV19_QSEE_1} or \cite{DPZ}.  
In particular, there exists a family of standard independent Brownian motions $(W^{(\kk,\ell)}_t)_{(\kk,\ell)\in\z}$ on $(\O,\F,(\F_t)_{t\geq 0},\P)$ such that 
\begin{equation}
\label{eq:L20cylindrical_noise}
\textstyle
W_t=\sum_{(\kk,\ell)\in\z} \g_{(\kk,\ell)}e_{\kk}\,W^{(\kk,\ell)}_t.
\end{equation}
We can now define $H^\sigma$-solutions to \eqref{eq:primitive_full} with $\sigma\geq 1$.

\begin{definition}[Local, unique and global $H^\sigma$-solutions]
\label{def:solution}
Assume that $\sigma\geq 1$,
$Q\in \calL_2(\Ls_0^2(\T^3),\Hs_0^\sigma(\T^3))$ and $v\in L^2_{\F_0}(\O;\Hs_0^\sigma(\T^3))$.
Let
$$
\tau:\O\to [0,\infty]\qquad \text{ and }\qquad 
(v_t)_t:[0,\tau)\times \O\to \Hs^{1+\sigma}(\T^3)
$$ 
be a stopping time and a stochastic process, respectively. 
\begin{itemize}
\item $((v_t)_t ,\tau)$ is said to be a \emph{local $H^\sigma$-solution} to \eqref{eq:primitive_full} if the following hold.
\vspace{0.1cm}
\begin{itemize}
\item $(v_t)_t\in L^2_{\loc}([0,\tau);H^{1+\sigma}(\T^3;\R^3))$ a.s.;
\vspace{0.1cm}
\item $((v_t\cdot\nabla_{x,y}) v_t+ w(v_t)\partial_z v_t)_t \in L^2_{\loc}([0,\tau);H^{-1+\sigma}(\T^3;\R^3))$ a.s.;
\vspace{0.1cm}
\item a.s.\ for all $t\in [0,\tau)$,
$$
v_t= v + \int_0^t \big(\Delta v_r + \p[(v_r\cdot\nabla_{x,y}) v_r+ w(v_r)\partial_z v_r ]\big)\,\dd r+ Q W_t.
$$
\end{itemize}
\item $((v_t)_t,\tau)$ is said to be a \emph{maximal unique $H^\sigma$-solution} if it is a local $H^\sigma$ solution and for any other local $H^\sigma$-solution
$((v_t')_t,\tau')$ we have $\tau'\leq \tau$ a.s.\ and $v_t=v_t'$ a.s.\ for all $t\in [0,\tau')$.
\item A maximal unique $H^\sigma$-solution $((v_t)_t,\tau)$ is called \emph{global unique $H^\sigma$-solution} if 
$\tau=\infty$ a.s.\ In this case, we simply write $(v_t)_t$ instead of $((v_t)_t,\tau)$.
\end{itemize}
\end{definition}

In the above, $QW_t=\sum_{(\kk,\ell)\in \z} q_{(\kk,\ell)} W^{(\kk,\ell)}_t$, where $q_{(\kk,\ell)}\stackrel{{\rm def}}{=} Q (\g_{(\kk,\ell)} e_{\kk})$.
The lower bound $\sigma\geq 1$ is consistent with the local well-posedness literature on stochastic PEs \eqref{eq:primitive_full}, cf., \cite{primitive1,primitive2,BS21,Debussche_2012}. 
Note also that for any $H^\sigma$-local solution $((v_t)_t,\tau)$, we have 
$
\int_{\T^3} v_t(\x)\,\dd \x=0\text{  a.s.\ for all }t<\tau,
$ due to the mean-zero condition of $v_0$ and $Q$.

\subsubsection{Probabilistic set-up}
\label{sss:probabilistic_set_up}
In the following, we formulate the assumption on $Q$ needed to prove Lagrangian chaos for PEs; see Theorem \ref{t:intro} and Subsection \ref{ss:main_results}. 

\begin{assumption} 
\label{ass:Q}
There exists 
$\alpha>6$ such that the operator $Q$ is diagonalisable with respect to the basis $( \g_{(\kk,\ell)}e_{\kk})_{(\kk,\ell)\in \z} $ of $\Ls^2_0(\T^3)$ with singular values $( q_{(\kk,\ell)})_{(\kk,\ell)\in \z} $ satisfying
\begin{equation}
\label{eq:coloring_assumption_noise}
q_{(\kk,\ell)} \eqsim  |\kk|^{-\alpha}  \ \ \text{ for all }\ (\kk,\ell)\in\z.
\end{equation}
\end{assumption}

All the results of the current work hold under Assumption \ref{ass:Q}. However, we can weaken it for unique ergodicity and regularity results; see Subsection \ref{ss:global_well_statement} below.

Comparing the above assumption with the one used in \cite{BBPS2022} (see Assumptions 1 and 2 in Subsection 1.1.1 there), one can expect Lagrangian chaos for PEs to hold even if $q_{(\kk,\ell)}=0$ for $1< |\kk|< N$ with $N>1$ (i.e., mid-frequency noise degeneracy).
For simplicity, we do not pursue this level of generality here. However, we expect that the methods presented here and the existing literature on SPDEs with finite-dimensional noise degeneracy \cite{EckHai2001,RX11,EM01_finite} can be combined to extend the results of the current manuscript to the latter situation.  
Let us mention that we will, anyhow, encounter finite-dimensional noise degeneracy in the study of strong Feller for the Lagrangian and projective processes; see Section \ref{s:strong_Feller}. 
\smallskip

If Assumption \ref{ass:Q} holds, then we fix $\sigma$ such that 
\begin{equation}
\label{eq:coloring_assumption_noise_2}
\sigma>\frac{9}{2}, \quad \text{ and }\quad \alpha-2 <\sigma<\alpha-\frac{3}{2}.
\end{equation}
As Proposition \ref{prop:global_well_posedness} below shows, the choice of $\sigma$ determines the pathwise regularity of the solution to the stochastic PEs \eqref{eq:primitive_full}.
More precisely,
the condition $\sigma<\alpha-\frac{3}{2}$ ensures  $Q\in \calL_2(\Ls^2_0(\T^3),\Hs^\sigma_0(\T^3))$ and therefore
 the solution $(v_t)_t$ to \eqref{eq:primitive_full} takes values in $C([0,\infty);H^\sigma(\T^3;\R^2))$. The latter combined with $\sigma>\frac{9}{2}$ and $
H^\sigma(\T^3) \subseteq C^{\eta}(\T^3)$, where $ \eta =\sigma-\tfrac{3}{2}>3$,
ensures the well-definedness of the projective process in \eqref{eq:projective_1}-\eqref{eq:projective_2}. Finally, $\sigma>\alpha-\frac{3}{2}$ guarantees the uniqueness of the invariant measure for on \eqref{eq:primitive_full} on $\Hs^\sigma(\T^3)$; see Theorem \ref{t:uniqueness_regularity_invariant} and Example \ref{ex:operator_Q}.

\section{Statement of the main results and strategy}
\label{ss:global_unique_erg}
This section presents the main results and is organised as follows:
\begin{itemize}
\item Subsection \ref{ss:global_well_statement} -- Global well-posedness and unique ergodicity of PEs. 
\item Subsection \ref{ss:main_results} -- Lagrangian chaos and scalar advection. 
\item Subsection \ref{ss:proofs_further_results} -- Proofs outline and auxiliary results.
\end{itemize}
The proofs are given in the later sections. 

\subsection{Global well-posedness in $H^\sigma$ and unique ergodicity} 
\label{ss:global_well_statement}
We begin by stating the global well-posedness of \eqref{eq:primitive_full} in high-order Sobolev spaces.
Recall that the Markov process $(v_t)_t$ is said to be Feller on $\Hs_0^\sigma(\T^3)$ if the mapping
$$
\Hs_0^\sigma(\T^3)\ni v\mapsto \E_v[\phi(v_t)]  \ \text{ is continuous for all $\phi\in C_{{\rm b}}(\Hs_0^\sigma(\T^3);\R)$},
$$
where the subscript in the expected value emphasises that $v_0=v$.

\begin{proposition}[Global well-posedness in higher-order Sobolev spaces]
\label{prop:global_well_posedness}
Let $Q\in \calL_2(\Ls_0^2(\T^3),\Hs_0^\sigma(\T^3))$ for some $\sigma\geq 2$. Then, for each $v\in \Hs_0^\sigma(\T^3)$, there exists a \emph{global unique $H^\sigma$-solution} to \eqref{eq:primitive_full} and satisfies
\begin{equation}
\label{eq:pathwise_regularity_w}
(v_t)_t\in C([0,\infty);H^\sigma(\T^3;\R^2)) \text{ a.s.}
\end{equation}
Finally, $(v_t)_{t}$ is a Feller process on $\Hs_0^\sigma(\T^3)$.
\end{proposition}

From the proof, it is clear that the above also holds without the mean-zero conditions on $v_0$ and $Q$. 

Although the global well-posedness of PEs with initial data in $\Hs^1(\T^3)$ has been extensively studied in the literature, see e.g., \cite{primitive2,primitive1,BS21,Debussche_2012}, the above is new. 
The main difficulty behind the proof of Proposition \ref{prop:global_well_posedness} is the criticality of $H^1$ for the PEs; see \cite[Subsection 1.2]{A23_primitive3}. In particular, the above result \emph{cannot} be obtained by a standard bootstrap method, and the Feller property requires the proof of new energy estimates for the stochastic PEs \eqref{eq:primitive_full}, which will be presented in Section \ref{s:well_posedness_Hsigma}. 
The obtained estimate in the non-critical space $H^\sigma$ with $\sigma\geq 2$ will be also of fundamental importance in the later proofs of leading to Lagrangian chaos of Theorem \ref{t:intro} (cf., Lemma \ref{l:strong_feller_reduction} in Section \ref{s:strong_Feller}). The case $\sigma\in (1,2)$ seems more complicated.

\smallskip

Next, we present our results on the regularity and uniqueness of invariant measures for stochastic PEs \eqref{eq:primitive_full}. Let $\rho\geq0$. 
A Borel probability measure $\mu$ on $\Hs_0^\rho(\T^3)$ is called an \emph{invariant} (or stationary) measure  for \eqref{eq:primitive_full} on $\Hs_0^\rho(\T^3)$ if 
for all $t>0$ and for all continuous observable  
$\phi:\Hs_0^\rho(\T^3)\to \R$, it holds that 
$$
\textstyle\int_{\Hs_0^\rho(\T^3)} \phi(v)\,\dd \mu(v)= \int_{\Hs_0^\rho(\T^3)} \E_v[\phi(v_t)]\,\dd \mu(v) .
$$

\begin{proposition}[Regularity of invariant measures]
\label{prop:regularity_invariant_measure}
Let $Q\in \calL_2(\Ls_0^2(\T^3),\Hs_0^\sigma(\T^3))$ for some $\sigma\geq 2$, and fix $\rho\in [1,\sigma]$.
Then there exists a strictly increasing continuous mapping $\Fun:[0,\infty)\to [0,\infty)$ such that $\lim_{R\to \infty} \Fun(R)=\infty$ and for all invariant measures $\mu$ of the stochastic {\normalfont{PEs}} \eqref{eq:primitive_full} on $\Hs^{\rho}_0(\T^3)$ satisfy, for all $\sigma'<\sigma+1$,
\begin{equation}
\label{eq:integrability_invariant_measure_Hrho}
\int_{\Hs^\rho_0(\T^3)} \Fun(\|v\|_{\Hs^{\sigma'}(\T^3)})\,\dd \mu(v)\leq C(\|Q\|_{\calL_2(\Ls_0^2(\T^3),\Hs_0^\sigma(\T^3))}).
\end{equation}
In particular, all invariant measures of the stochastic {\normalfont{PEs}} \eqref{eq:primitive_full} on $\Hs^\rho_0(\T^3)$ are concentrated on $\bigcap_{\sigma'<1+\sigma}\,\Hs^{\sigma'}_0(\T^3)$.
\end{proposition}

By parabolic regularity, the above is optimal.
We will obtain \eqref{eq:integrability_invariant_measure_Hrho} with $\Fun(R)= \log^{[ N]} (1+R)$ where $N\geq 1$ depends only on $\sigma$ and $\log^{[N]} (1+R)$ denotes the $N$-time composition of the mapping $R\mapsto \log(1+R)$. In case $\sigma$ is close to $2$, the previous bound is weaker than the one in \cite[Theorems 1.7 and 1.9]{GHKVZ14}.
We do not expect this to be optimal, but it is sufficient for our purposes.

Proposition \ref{prop:regularity_invariant_measure} and Sobolev embeddings also yield
\begin{equation*}
Q\in\cap_{\sigma\geq 1} \,\calL_2(\Ls_0^2(\T^3),\Hs_0^{\sigma}(\T^3))
\quad \Longrightarrow \quad
\mu(C^{\infty}(\T^3;\R^2)\cap \Ls_0^2(\T^3))=1.
\end{equation*}
The previous result provides an affirmative answer to a question posed by Glatt-Holtz, Kukavica, Vicol, and Ziane, stated below in \cite[Theorem 1.9]{GHKVZ14}.

Next, we turn our attention to the uniqueness of invariant measures.
As mentioned before Theorem \ref{t:uniqueness_invariant_intro}, the following is the first uniqueness result for invariant measure for \eqref{eq:primitive_full}. Recall that 
$
v\in \supp\mu$ if $\mu(B_{\Hs^\rho_0(\T^3)}(v,\delta))>0$ for all $\delta>0$.

\begin{theorem}[Uniqueness and support of invariant measures]
\label{t:uniqueness_regularity_invariant}
Suppose that for some $\sigma\geq 2$ we have $Q\in \calL_2(\Ls_0^2(\T^3),\Hs_0^\sigma(\T^3))$, the range of $Q$ is dense, and there exists $\sone<\sigma+2$ such that 
\begin{align}
\label{eq:generalization2}
\|f\|_{\Ls_0^2(\T^3)}\lesssim \|Q f\|_{\Hs_0^{\sone}(\T^3)} \ \text{ for all }f\in  \Ls_0^2(\T^3).
\end{align}
Then the following hold:
\begin{itemize}
\item\label{it:uniqueness_regularity_invariant_uniqueness} 
There exists a unique invariant measure $\mu$ for \eqref{eq:primitive_full} on $\Hs_0^\sigma(\T^3)$.

\item The invariant measure $\mu$ has full support on $\Hs_0^\sigma(\T^3)$, i.e.,  
$
\supp\,\mu=\Hs_0^\sigma(\T^3).
$
\end{itemize}  
\end{theorem}

Additional results on the transition Markov kernels associated with the stochastic PEs \eqref{eq:primitive_full} are given in Remark \ref{r:strongly_mixing_etc}.
Proposition \ref{prop:regularity_invariant_measure} and Theorem \ref{t:uniqueness_regularity_invariant} show that the ergodic measure in $\Hs_0^1(\T^3)$ constructed in \cite[Theorem 1.6]{GHKVZ14} for \eqref{eq:primitive} is unique in case $Q$ satisfies the assumptions in Theorem \ref{t:uniqueness_regularity_invariant}.

Next, we provide an example of $Q$ satisfying the assumptions of Theorem \ref{t:uniqueness_regularity_invariant}.

\begin{example}
\label{ex:operator_Q}
With the notation of Subsection \ref{sss:probabilistic_set_up}, the diagonal $Q$ defined as 
$$
\textstyle
Qf =\sum_{(\kk,\ell)\in \z} q_{(\kk,\ell)} \g_{(\kk,\ell)} e_\kk (f,\g_{(\kk,\ell)} e_\kk )_{L^2},
$$ 
where $ |q_{(\kk,\ell)}|\eqsim |\kk|^{-\gamma}$ and $\g>\frac{7}{2}$, satisfies the assumption of Theorem \ref{t:uniqueness_regularity_invariant} on $\Hs^\sigma_0(\T^3)$ with $\sigma\geq 2$ satisfying 
$
\gamma-2<\sigma<\gamma-\frac{3}{2}.
$
\end{example}

The proofs of Propositions \ref{prop:global_well_posedness}, \ref{prop:regularity_invariant_measure} and Theorem \ref{t:uniqueness_regularity_invariant} are given in Section \ref{s:well_posedness_Hsigma}.
In the next subsection, we state the main result of this manuscript. 

\subsection{Lagrangian chaos for stochastic PEs}
\label{ss:main_results}
Let $(\phi_t)_t$ be the random flow on $\T^3$ associated to $u_t=(v_t,w(v_t))$ via \eqref{eq:lagrangian_flow} where $v_t$ and $w(v_t)$ satisfy the stochastic PEs \eqref{eq:primitive_full} and is as in \eqref{eq:def_w}, respectively. Note that, due to \eqref{eq:pathwise_regularity_w}, the random flow $(\phi_t)_t$ is a.s.\ well-defined. 
Our first main result proves that the latter is chaotic.

\begin{theorem}[Positive top Lyapunov exponent]
\label{t:positive_lyapunov_exponent}
Let $Q$ be as in Subsection \ref{sss:probabilistic_set_up}. Then there exists a deterministic constant $\lambda_+>0$ such that, for every initial position $\x\in \Dom$ and initial velocity $v\in \Hs^\sigma_0(\T^3)$, 
\begin{align}
\label{eq:positive_lyapunov_exponent1}
\lambda_+=\lim_{t\to \infty} \frac{1}{t}\log |\nabla \phi_t (\x) |>0 \ \text{ a.s.\ } 
\end{align}
Moreover, the flow expands exponentially in \emph{all} directions with rate $\lambda$, i.e., for every $\x\in \Dom$, $v\in \Hs^\sigma_0(\T^3)$ and $\xi\in \R^3\setminus\{0\}$, 
\begin{align}
\label{eq:positive_lyapunov_exponent2}
\lambda_+=\lim_{t\to \infty} \frac{1}{t}\log |\nabla \phi_t (\x)\xi | >0\ \text{ a.s. }  
\end{align}
\end{theorem}

Of course, \eqref{eq:positive_lyapunov_exponent2} is stronger than \eqref{eq:positive_lyapunov_exponent1}.
Comments on the physical description of the above have already been given below Theorem \ref{t:intro} and \eqref{eq:intro_expansion}.

\smallskip

Following \cite[Subsection 1.2.1]{BBPS2022}, we study the growth of the gradient of a passive scalar $(f_t)_t$ advected by the solution $(u_t)=(v_t,w(v_t))$ of the stochastic PEs \eqref{eq:primitive_full}:
\begin{equation}
\label{eq:linear_scalar}
\left\{
\begin{aligned}
&\partial_t f_t + (u_t\cdot\nabla f_t)=0 &\text{ on }&\Dom,\\
&f_0=f&\text{ on }&\Dom,
\end{aligned}
\right.
\end{equation}
where $f$ is given. Due to \eqref{eq:pathwise_regularity_w} and $\sigma>\frac{9}{2}$, we have $(\phi_t)_t\in C(\R_+;C^2(\overline{\Dom}))$ a.s.\ and therefore the unique regular (in the Sobolev sense) solution to \eqref{eq:linear_scalar} is given by $f_t(\x)=f(\phi_t^{-1}(\x))$. Since $\phi_t$ is volume preserving, it follows that
\begin{equation}
\label{eq:L1_passive_scalar}
\textstyle
\|\nabla f_t\|_{L^1(\Dom)}= \int_{\Dom} |(\nabla \phi_t(\x))^{-\top} \nabla f(\x)|\,\dd \x .
\end{equation}
The RHS of the above suggests investigating the properties of $(\nabla \phi_t(\x))^{-\top}$ (here, $\top$ denotes the transpose). In analogy with Theorem \eqref{t:positive_lyapunov_exponent}, we have

\begin{theorem}[Positive top Lyapunov exponent II]
\label{t:positive_lyapunov_exponent_II}
Let $Q$ be as in Subsection \ref{sss:probabilistic_set_up}. Then there exists a deterministic constant $\overline{\lambda}_+>0$ for which the statements in \eqref{eq:positive_lyapunov_exponent1} and \eqref{eq:positive_lyapunov_exponent2} hold with $(\nabla \phi_t(\x),\lambda_+)$ replaced by $((\nabla \phi_t(\x))^{-\top},\overline{\lambda}_+)$. 
\end{theorem}

As a consequence of \eqref{eq:L1_passive_scalar} and Theorem \ref{t:positive_lyapunov_exponent_II} we have (cf.,
\cite[Subsection 2.6]{BBPS2022})

\begin{corollary}[Gradient growth for passive scalars]
\label{cor:gradient_expansion}
Let $Q$ and $\overline{\lambda}_+>0$ be as in Subsection \ref{sss:probabilistic_set_up} and Theorem \ref{t:positive_lyapunov_exponent_II}, respectively. Then, for all $v\in \H(\T^3)$, $\eta\in (0,\overline{\lambda}_+)$ and a non-zero $f\in W^{1,1}(\Dom)$ such that $\int_{\Dom}f(\x)\,\dd \x=0$, there exists a constant $R=R(v,f,\eta)>0$ such that, for all $t>0$,
\begin{equation}
\label{eq:exponential_growth_passive_scalar}
\E\int_{\Dom} f_t(\x)\,\dd \x\geq R\, e^{\eta t}.
\end{equation}
\end{corollary}

\begin{proof}
Let $\Df\stackrel{{\rm def}}{=} \{\x\in \Dom\,:\, \nabla f(\x)\neq 0\}$. 
By assumption, $\Leb(\Df)>0$. 
Theorem \ref{t:positive_lyapunov_exponent_II} ensures that, for all $v\in \H(\T^3)$ and $\x\in \Dom$, 
\begin{equation}
\label{eq:rewriting_lemma_cocycle_T}\textstyle{
\overline{\lambda}_+=\lim_{t\to \infty} \frac{1}{t}\log |\nabla \phi^t_{\cdot,v} (\x)\nabla f(\x) | >0\ \text{ a.s. }  }
\end{equation}
where $\overline{\lambda}_+>0$ is deterministic. For all $\eta\in (0,\overline{\lambda}_+)$, $\om\in \O$ and $\x\in \Dom$, set
$$
r(\om,\x)\stackrel{{\rm def}}{=} \one_{\Df}(\x)\,\textstyle{\inf}_{t>0}\big( |\nabla \phi^t_{\om,v} (\x)\nabla f(\x) | e^{-\eta t}\big).
$$
It follows from \eqref{eq:rewriting_lemma_cocycle_T} that 
$
\P\times \Leb (r>0)=\Leb( \Df)>0.
$
Due to \eqref{eq:L1_passive_scalar} and the above reasoning, the estimate \eqref{eq:exponential_growth_passive_scalar} holds with $R=\E \int_{\Dom} r(\x)\, \dd \x >0$.
\end{proof}

Corollary \ref{cor:gradient_expansion} and the arguments in \cite[Section 8]{BBPS2022} imply the validity of Yaglom's law in the study of scalar turbulence in the Batchelor regime; see \cite[Subsection 1.2.1]{BBPS2022}. 
To avoid repetitions, we leave the details to the interested reader.

\smallskip

The proofs of Theorems \ref{t:positive_lyapunov_exponent} and \ref{t:positive_lyapunov_exponent_II} are given in Subsection \ref{ss:proofs_further_results} based on sufficient conditions for positivity of the top Lyapunov exponents proven in \cite[Section 4]{BBPS2022} and further results on the stochastic PEs \eqref{eq:primitive_full}. Details are given in Subsection \ref{ss:proofs_further_results} below.
Sections \ref{s:integrability_inv_measure_u}-\ref{s:irreduc_controllability} are devoted to the proofs of the above-mentioned additional results on \eqref{eq:primitive_full}.

\subsection{Outline and proofs of Theorems \ref{t:positive_lyapunov_exponent} and \ref{t:positive_lyapunov_exponent_II}}
\label{ss:proofs_further_results}
As mentioned in Section \ref{s:intro}, the proofs of Theorems \ref{t:positive_lyapunov_exponent}  and \ref{t:positive_lyapunov_exponent_II} are based on the sufficient condition for the positivity of the top Lyapunov exponent proven in 
\cite[Section 4]{BBPS2022}. 
For brevity, here we limit ourselves to explaining how to apply their framework to the current situation rather than presenting the results in \cite{BBPS2022} in full generality.
In this subsection, we will state several results which will be proven in the later sections. Combining the latter with the results in \cite{BBPS2022}, we will prove Theorems \ref{t:positive_lyapunov_exponent} and \ref{t:positive_lyapunov_exponent_II}. We begin by focusing on checking the assumptions in \cite[Proposition 4.17]{BBPS2022}. 
Let us begin by introducing some notation and terminology.

\smallskip

Without loss of generality (see also the proof of Proposition \ref{prop:RDSII}), here and in the following we assume that $\O=C_0([0,\infty);\WW)$ where $\WW$ is a separable Hilbert space such that $\Ls^2_0(\T^3)\embed \WW$ is Hilbert-Schmid endowed with the standard Weiner measure, which we still denote by $\P$. 
Note that, for all $t>0$ and $\om\in \O$, the mapping $\theta:[0,\infty)\times \O\to \O$ given by $(t,\om)\mapsto \theta^t\om$ where
\begin{equation}
\label{eq:definition_thetat}
\theta^t \om(s)\stackrel{{\rm def}}{=}\om(t+s)-\om(s),
\end{equation}
is measurable, measure preserving (i.e., $\P\circ (\theta^t)^{-1}=\P$) and satisfies the semigroup property $\theta^t\circ \theta^s=\theta^{s+t}$ for all $t,s\geq 0$. 

A mapping $\Ran:[0,\infty)\times \Omega\times Z\to Z$ on a polish space $Z$, written as $(t,\om,z)\mapsto \Ran_\om^t(z)$, is said to be a (continuous) \emph{random dynamical system} (RDS in the following) if the following are satisfied (recall that $\Borel$ denotes the Borel $\sigma$-algebra):
\begin{itemize}
\item (Measurability) $\Ran$ is measurable from $\Borel ([0,\infty))\otimes \F \otimes \Borel(Z)$ to $ \Borel(Z)$.
\item (Cocycle property) For all $\om\in \O$ and $s,t\geq 0$, it holds that $\Ran_\om^0=\Id_Z$ and 
$$
\Ran^{t+s}_\om =
\Ran^{t}_{\theta^s(\om)} \circ \Ran^{s}_\om.
$$
\item (Boundedness and uniform continuity) For all $\om\in \O$ and all \emph{bounded} set $U\subseteq [0,\infty)\times Z$, the set $\Ran_\om(U)$ is bounded in $Z$ and uniformly continuous.
\end{itemize}

We say that the solution operator $v\mapsto (v_t)_t$ of the stochastic PEs \eqref{eq:primitive} induces an RDS on $\Hs_0^\sigma(\T^3)$ if there exists an RDS on the same space such that, for all $t\geq 0$ and $v\in Z$, we have $v_t(\om)=\Ran_\om^t (v)$ for a.a.\ $\om\in \O$.

\begin{proposition}[RDS for PEs]
\label{prop:RDS}
Let Assumption \ref{ass:Q} be satisfied and let $\sigma$ be as in \eqref{eq:coloring_assumption_noise_2}.
Then the solution operator $v\mapsto v_t$ associated to the stochastic {\normalfont{PEs}}  induces a {\normalfont{RDS}} over $\Hs_0^\sigma(\T^3)$. Moreover, denoting by $\V$ the corresponding mapping, we have
$ \V^t_{\theta^s\om}(v)$ is independent of $\F_s$ for all $t\geq s\geq 0$ and $v\in \Hs^\sigma_0(\T^3)$.
\end{proposition}

The above is a special case of Proposition \ref{prop:RDSII} whose proof is given in Subsection \ref{ss:RDS}.
We now return to the proof of Theorems \ref{t:positive_lyapunov_exponent} and \ref{t:positive_lyapunov_exponent_II}. We begin with the proof of the former, as the latter can be derived by replacing $\nabla\phi_t (\x)$ with $(\nabla\phi_t (\x))^{-\top}$ in the arguments presented below.

To prove Theorem \ref{t:positive_lyapunov_exponent}, we apply \cite[Propositions 3.16 and 4.17]{BBPS2022} on $  \H(\T^3) \times \Dom$, with the above choice of $(\theta^t)_t$ and
\begin{align}
\label{eq:TA_definition_1}
\t^t_{\om}(v,\x)&=(\phi^t_{\om,v}(\x),\V^t_\om(v)), \\ 
\label{eq:TA_definition_2}
\A_{\om,(v,\x)}^t&= \nabla \phi^t_{\om,v}(\x),
\end{align}
where $\om\in\O$, $t\geq 0$, $\x\in \Dom$ and $v\in \H(\T^3)$. Moreover, for $\om\in\O$, $(\phi^t_\om)_t$ denotes the Lagrangian flow associated with the RDS in Proposition \ref{prop:RDS}:
\begin{equation}
\label{eq:lagrangian_flow_strategy}
\partial_t\phi^t_\om(\x)= \U^t_\om(\phi^t_\om(\x))\  \text{ for } t\geq 0, \qquad \phi^0_\om(\x)=\x\in \Dom,
\end{equation}
where
\begin{equation}
\U^t_\om\stackrel{{\rm def}}{=}(\V^t_\om,w(\V^t_\om))
\end{equation}
and $w(\cdot)$ is as in \eqref{eq:def_w}.
As commented below \eqref{eq:coloring_assumption_noise_2}, the matrix $\A_{\om,(v,\x)}$ is well defined.
Moreover, the fact that $\nabla \cdot \U^t_\om=0$ implies $\A_{\om,(v,\x)}^t\in \mathrm{SL}_3(\R)$  (i.e., $\A_{\om,(v,\x)}$ is invertible) for a.a.\ $\om\in \O$ and all $(v,\x)\in \Hs^\sigma_0(\T^3)\times \Dom$.

Let us recall that, for all $v\in \Hs^\sigma_0(\T^3)$, we have $\V^t_\om (v)=v_t(\om)$ for a.a.\ $\om\in \O$ and $t\geq 0$. Hence, the above definition of $\phi^t_\om$ coincides with the one in Theorem \ref{t:intro} whenever the $\O$-expectional sets are allowed to depend on the initial data $v$ (as in the case of the statement in Theorem \ref{t:intro}). 
In the following, if no confusion seems likely, we will not distinguish between $(v_t)_t$ and the generated RDS $ (\V^t(v))_t$, and similarly for $(u_t)_t$ and $(\U^t(v))_t$.

Proposition \ref{prop:RDS} implies that $\t$ is a random dynamical system on $\H(\T^3)\times \Dom$ as well, and $\A$ is a (three-dimensional) linear cocycle over $\t$ (\cite[Definition 3.11]{BBPS2022}). Moreover, Proposition \ref{prop:RDS} also ensures that the assumptions \cite[(H1) and (H3)]{BBPS2022} are satisfied. Next, we verify the integrability conditions \cite[(H2)]{BBPS2022} for the linear cocycle $\A$, i.e., (here $\log^+\cdot=\max\{\log\cdot,0\}$)
\begin{align}
\label{eq:integrability_At}
\E\int_{\H(\T^3)}\int_{\Dom} \sup_{0\leq t\leq 1} \log^+ |\A_{(v,\x)}^t|\,\dd \x\,\dd \mu(v)&<\infty,\\
\label{eq:integrability_At1}
\E\int_{\H(\T^3)}\int_{\Dom} \sup_{0\leq t\leq 1} \log^+ |(\A_{(v,\x)}^t)^{-1}|\,\dd \x\,\dd \mu(v)&<\infty.
\end{align}
In the above, we used that $\mu\otimes \Leb$ is the unique invariant (hence ergodic) measure for the Lagrangian process $(v_t,\phi^t)$; see Corollary \ref{cor:uniqueness_projective} below.

Let us stress that the conditions \eqref{eq:integrability_At}-\eqref{eq:integrability_At1} are \emph{necessary} to obtain the existence of the Lyapunov exponents for the linear cocycle $\A$ over $\t$ via the multiplicative ergodic theorem (MET), see e.g., \cite[Theorem 3.4.1]{A98_Random}.
To check the above, we employ the following special case of Theorem \ref{t:integrability_MET}.

\begin{proposition}[$L^1$-Integrability of the gradient of the velocity field]
\label{prop:integrability_MET_L1}
Let Assumption \ref{ass:Q} be satisfied and set $\Energy\stackrel{{\rm def}}{=}\|Q\|_{\calL(\Ls_0^2(\T^3),\Hs_0^1(\T^3))}$. 
Then
\begin{equation}
\label{eq:integrability_mu_L1}
\int_{\Hs^\sigma_0(\T^3)}\int_{\Dom}|\nabla u(\x)|\,\dd \x\,\dd \mu(v)\leq C(\Energy),
\end{equation}
where $u=(v,w(v))$ and $w(v)$ is as in \eqref{eq:def_w}.
\end{proposition}

Next, we show how Proposition \ref{prop:integrability_MET_L1} implies \eqref{eq:integrability_At}-\eqref{eq:integrability_At1} with $\A_{(v,\x)}^t=\A_{\cdot,(v,\x)}^t$ as in \eqref{eq:TA_definition_2}. Differentiating \eqref{eq:lagrangian_flow_strategy}, we obtain  
\begin{equation}
\label{eq:equation_A_cocycle}
\partial_t \A_{(v,\x)}^t
= \nabla u_t(\phi^t(\x))\, \A_{(v,\x)}^t\  \text{ for } t\geq 0, 
 \qquad  \A_{(v,\x)}^0= \Id.
\end{equation}
For notational convenience, let us set $\x_t\stackrel{{\rm def}}{=}\phi^t(\x)$. 
By Gr\"onwall's lemma,
$$
|\A^t_{(v,\x)}|\leq \exp \Big(\int_0^1 |\nabla u_t(\x_t)|\,\dd t\Big) \ \ \text{ for }\   t\in [0,1].
$$
Thus, the left hand side of \eqref{eq:integrability_At} can be estimate by 
\begin{align*}
\E\int_{\H(\T^3)}\int_{\Dom} \int_0^1 |\nabla u_t(\x_t)|\,\dd t\,\dd \mu(v)\,\dd \x
&\stackrel{(i)}{=}\E\int_0^1 \int_{\H(\T^3)}\int_{\Dom}  |\nabla u_t(\x)|\,\dd \mu(v)\,\dd \x\, \,\dd t\\
&\stackrel{(ii)}{=} \int_{\H(\T^3)}\int_{\Dom}  |\nabla u(\x)|\,\dd \mu(v)\,\dd \x\stackrel{\eqref{eq:integrability_mu_L1}}{<}\infty.
\end{align*}
where in $(i)$ and $(ii)$ we used that transformation $\x\mapsto \x_t=\phi_t(\x)$ is volume preserving and the invariance of $\mu$, respectively.
Finally, \eqref{eq:integrability_At1} follows similarly as
$$
\partial_t [(\A^t_{(v,\x)})^{-1}]= (\A^t_{(v,\x)})^{-1} \partial_t \A^t_{(v,\x)} (\A^t_{(v,\x)})^{-1}
\stackrel{\eqref{eq:equation_A_cocycle}}{=} (\A^t_{(v,\x)})^{-1}\nabla u_t(\x_t).
$$

Instead of the choice \eqref{eq:TA_definition_2}, to prove Theorem \ref{t:positive_lyapunov_exponent_II}, we consider the linear cocycle $\wc{\A}_{\om,(v,\x)}\stackrel{{\rm def}}{=}(\nabla \phi_{\om,v}^t (\x))^{-\top}$ which, due to the previous formula, solves
\begin{equation}
\label{eq:cocycle_hat}
\partial_t \wc{\A}_{\om,(v,\x)}^t = (\nabla \U^t_\om (\x_t))^{-\top} \wc{\A}_{\om,(v,\x)}^t \ \text{ for }t>0, \qquad   \wc{\A}_{\om,(v,\x)}^0= \Id.
\end{equation}
As for the linear cocycle $\A$, the conditions \cite[(H1) and (H3)]{BBPS2022} and \cite[(H2)]{BBPS2022} follows respectively from Propositions \ref{prop:RDS} and \ref{prop:integrability_MET_L1}, respectively.

\smallskip

Now, we continue with checking the condition of \cite[Proposition 4.17]{BBPS2022} for the linear cocycles $\A$ and $\wc{\A}$. To this end, we need the so-called projective processes. As explained in \cite[Subsections 3.2.3 and 3.3.1]{BBPS2022}, the processes 
$$
\xi_t \stackrel{{\rm def}}{=} \A^t_{(v,\x)} \xi/|\A^t_{(v,\x)} \xi|\qquad \text{ and } \qquad \wc{\xi}_t \stackrel{{\rm def}}{=} \A^t_{(v,\x)} \wc{\xi}/|\A^t_{(v,\x)}  \wc{\xi}|,
$$ 
where $\xi,\wc{\xi}\in \S^2$ are fixed directions, and play a central role in understanding the positivity of the top Lyapunov exponent. For instance, \eqref{eq:positive_lyapunov_exponent2} and its corresponding statement in Theorem \ref{t:positive_lyapunov_exponent_II} are concerned with the time asymptotic property of such processes. 
It is routine to check that $\xi_t$ and $\wc{\xi}_t$ satisfy 
\begin{align}
\label{eq:projective_1}
\partial_t \xi_t&= \Pi_{\xi_t} (\nabla u_t(\x_t)\xi_t),\\
\label{eq:projective_2}
\partial_t \wc{\xi}_t&= -\Pi_{\wc{\xi}_t} ([\nabla u_t(\x_t)]^{\top}\wc{\xi}_t),
\end{align}
with initial data $\xi_0=\xi\in\S^2$ and $\wc{\xi}_0=\wc{\xi}\in \S^2$; respectively. Here, $\S^2\subseteq\R^3$ is the two-dimensional sphere and $\Pi_{\xi}= \Id- (\xi\otimes \xi)$ is the orthogonal projection from $\R^3$ onto the tangent space of $\S^2$ at the point $\xi$.
With a slight abuse of notation, we still denote by $\xi_t$ (resp., $\wc{\xi}_t$) the equivalent class in the two-dimensional projective space $\Pr$ given by $\pi(\xi_t)$ (resp., $\pi(\wc{\xi}_t)$) where $\pi:\S^2\to \Pr$ is the usual quotient map which identifies opposite points on $\S^2$.

\smallskip

The remaining ingredients needed to apply \cite[Proposition 4.17]{BBPS2022} are provided by the following three results. 
In the following, for a metric space $Z$, we denote by $B_{Z}(z,\varepsilon)$ the ball of radius $\varepsilon>0$ with center $z\in Z$. The superscript ${\rm c}$ indicates the complement of a set. 
We denote by $\Pl_t((v,\x),\cdot)$ the Markov transition function of the Lagrangian process $(v_t,\x_t)_t$, i.e., for all $v\in\H(\T^3)$, $\x\in \Dom$ and $t>0$,
\begin{equation}
\label{eq:Pl_definition}
\Pl_t((v,\x),A)= \P((v_t,\x_t)\in A\,|\, (v_0,\x_0)=(v,\x)).
\end{equation}

\begin{proposition}[Approximate controllability]
\label{prop:approximate_controllability}
Let Assumption \ref{ass:Q} be satisfied and let $\sigma$ be as in \eqref{eq:coloring_assumption_noise_2}.
The following results hold for all $t> 0$ and $\a\in\Tcal$.
\begin{enumerate}[{\rm(1)}]
\item\label{it:approximate_controllability0}
The support of $\Pl_{t}((0,\a),\cdot)$ contains $\{0\}\times \Dom$.
\item\label{it:approximate_controllability1} For any $\eta,\eta'\in \Pr$ and $\varepsilon>0$, 
\begin{align*}
\P\big((v_t,\x_t,\xi_t)\in B_{\Hs^\sigma}(0,\varepsilon)\times B_{\Dom}(\a,\varepsilon)\times B_{\Pr}(\eta',\varepsilon)\,|\, 
(v_0,\x_0,\xi_0)=(0,\a,\eta)\big)&>0,\\
\P\big((v_t,\x_t,\wc{\xi}_t)\in B_{\Hs^\sigma}(0,\varepsilon)\times B_{\Dom}(\a,\varepsilon)\times B_{\Pr}(\eta',\varepsilon)\,|\, 
(v_0,\x_0,\wc{\xi}_0)=(0,\a,\eta)\big)&>0.
\end{align*}
\item\label{it:approximate_controllability2} For all $M>0$ and $\varepsilon>0$,
\begin{align*}
\P\big((v_t,\x_t,\A_t)\in B_{\Hs^\sigma}(0,\varepsilon)\times B_{\Dom}(\a,\varepsilon)\times
 B_{\R^{3\times 3}}^{{\rm c}}(0,M)\,|\,
 (v_0,\x_0,\A_0)=(0,\a,\Id)\big)>0.
\end{align*}
\end{enumerate}
\end{proposition}

Before stating the next result, let us recall the definition of strong Feller property. A Markov process $(z_t)_t$ on a Polish space $Z$ with Markov transition function $Q_t(z,\cdot)=\P(z_t\in \cdot\,|\, z_0=z)$ is said to \emph{strong Feller} if the mapping 
$$
\textstyle{Z\ni z\mapsto \E_z[
\varphi(z_t)]=\int_{Z} \varphi(z')\,\dd Q_t (z,\dd z')}$$ 
is \emph{continuous} for all \emph{bounded} measurable observable $\varphi:Z\to \R$.

\begin{proposition}[Strong Feller -- Projective process]
\label{prop:strong_feller_projective}
Let Assumption \ref{ass:Q} be satisfied and let $\sigma$ be as in \eqref{eq:coloring_assumption_noise_2}.
Then the Markov semigroups associated with the processes $(v_t,\x_t,\xi_t)_t$ and $(v_t,\x_t,\wc{\xi}_t)_t$ are strong Feller in 
$
\H(\T^3)\times \Dom\times \Pr.
$
\end{proposition}

\begin{proposition}[Weak irreducibility -- Auxiliary processes]
\label{prop:weak_irreducibility}
Let Assumption \ref{ass:Q} be satisfied and let $\sigma$ be as in \eqref{eq:coloring_assumption_noise_2}.
Then the following holds.
\begin{itemize}
\item The support of any stationary measure for the Lagrangian process $(u_t,\x_t)$ on $\H(\T^3)\times  \Dom$ must contain $\{0\}\times \Dom$.
\item The support of any stationary measure for the projective processes $(u_t,\x_t,\xi_t)$ and $(u_t,\x_t,\wc{\xi}_t)$ on $\H(\T^3)\times \Dom\times \Pr$ must contain $\{0\}\times \Dom\times \Pr$.
\end{itemize}
\end{proposition}

Propositions \ref{prop:strong_feller_projective}, \ref{prop:weak_irreducibility}, and results from Appendix \ref{app:useful_results} yield the following uniqueness result for ergodic measures of the Lagrangian and projective processes.

\begin{corollary}[Uniqueness of ergodic measures -- Auxiliary processes]
\label{cor:uniqueness_projective}
Let Assumption \ref{ass:Q} be satisfied and let $\sigma$ be as in \eqref{eq:coloring_assumption_noise_2}.
The following are satisfied.
\begin{itemize}
\item
Let $\mu$ be the unique invariant measure for \eqref{eq:primitive_full} on $\H(\T^3)$ given by Theorem \ref{t:uniqueness_regularity_invariant}. The unique ergodic measure on $\H(\T^3)\times \Dom$ for the Lagrangian process $(v_t,\x_t)$ is given by $\mu\otimes \Leb$.
\item
The projective process $(v_t,\x_t,\xi_t)_t$ has a unique ergodic measure on $\Hs^\sigma_0(\T^3)\times \Dom\times \Pr$, and it is the unique element of the set
$$
\big\{\nu\text{ is an invariant measure of $(v_t,\x_t,\xi_t)_t$ satisfying $  \nu(\cdot\times \Pr)=\mu\otimes \Leb$}\big\}.
$$
The same holds with $(v_t,\x_t,\xi_t)_t$ replaced by $(v_t,\x_t,\wc{\xi}_t)_t$.
\end{itemize}
\end{corollary}

\begin{proof}
We first consider the Lagrangian process. The invariance of $\mu\otimes \Leb$ is immediate due to $\nabla \cdot u_t=0$. Moreover, due to Proposition \ref{prop:strong_feller_projective}, $\mu\otimes \Leb$ is also ergodicity due to Lemma \ref{l:ergodicity_supp_product} (or via Corollary \ref{cor:ergodicity_supp} and Theorem \ref{t:uniqueness_regularity_invariant}).
The uniqueness part now follows from a well-known variant of Khasminskii's uniqueness result (see, for example, \cite[Proposition 4.1.1]{DPZ_ergodicity}) and Propositions \ref{prop:strong_feller_projective}-\ref{prop:weak_irreducibility}.

Second, we consider the projective process $(v_t,\x_t,\xi_t)_t$, the other being similar. 
As above, the uniqueness of ergodic measures of  $(v_t,\x_t,\xi_t)_t$ follows from Propositions \ref{prop:strong_feller_projective}-\ref{prop:weak_irreducibility} and the above-mentioned variant of the Khasminskii theorem. Let $\Set$ be the set of invariant measures of $(v_t,\x_t,\xi_t)_t$ satisfying $  \nu(\cdot\times \Pr)=\mu\otimes \Leb$. This set is not empty due to the Krylov-Bogoliubov procedure, cf., \cite[Lemma 3.15]{BBPS2022} with $\mu$ replaced by $\mu\otimes \Leb$. By the Krein-Milman and Prokhorov theorems, the convex set $\Set$ has extreme points, and $\Set$ coincides with the closed convex hull of its extreme points. By Lemma \ref{l:uniqueness_z_factor} and the first part of this corollary, it follows that the extreme points of $\Set$ are ergodic measures. Hence, an ergodic measure for $(v_t,\x_t,\xi_t)_t$ exists, is unique among ergodic measures, and is the unique extreme point of $\Set$. 
\end{proof}

Building upon Propositions \ref{prop:RDS}–\ref{prop:weak_irreducibility} and Corollary \ref{cor:uniqueness_projective}, the general construction presented in \cite[Section 4]{BBPS2022} can be applied to prove 
Theorems \ref{t:positive_lyapunov_exponent} and \ref{t:positive_lyapunov_exponent_II}.

\begin{proof}[Proof of Theorems \ref{t:positive_lyapunov_exponent} and \ref{t:positive_lyapunov_exponent_II}]
We begin by proving \eqref{eq:positive_lyapunov_exponent1} (resp., \eqref{eq:positive_lyapunov_exponent2}). The latter follows from Propositions \ref{prop:RDS}-\ref{prop:approximate_controllability} and Theorem \ref{t:uniqueness_regularity_invariant} by applying combining the multiplicative ergodic theorem  \cite[Theorem 3.4.1]{A98_Random} (resp., random multiplicative ergodic theorem \cite[Theorem III.1.2]{K86_ergodic_theory}) and
\cite[Proposition 4.17]{BBPS2022} to the linear cocycle $\A_{\om,(v,\x)}^t$ induced by \eqref{eq:primitive_full}. 
Let us note that multiplicative ergodic theorems ensure the existence of the limits \eqref{eq:positive_lyapunov_exponent1} and \eqref{eq:positive_lyapunov_exponent2} only for $\mu\otimes \Leb$-a.a.\ $(v,x)\in \H(\T^3)\times\Dom$ (here we also used Theorem \ref{t:uniqueness_regularity_invariant}). 
As noticed in \cite[Remark 2.2]{BBPS2022}, the strong Feller property as in Proposition \ref{prop:strong_feller_projective} allow to upgrade the statement to for \emph{all} $(v,\x)\in \H(\T^3)\times \Dom$. For the reader's convenience, we include some details. 
We only comment on the case \eqref{eq:positive_lyapunov_exponent1} as the other is similar. Suppose that \eqref{eq:positive_lyapunov_exponent1} holds for a.a.\ $(v,\x)\in  \H(\T^3)\times \Dom$. 
Let 
$
\KK\stackrel{{\rm def}}{=}\{(\om,v,\x)\in \O\times \H(\T^3)\times \Dom
\,:\, \lim_{t\to \infty}t^{-1}\log |\nabla \phi^t_{\om,v}(\x)|=\lambda\}
$
and set 
$$
\psi_{\KK}(v,\x)\stackrel{{\rm def}}{=} \E [\one_{\KK}(\cdot,v,\x)]=\E [\one_{\KK}(\theta^t\cdot,v,\x)] \ \ \text{ for all }t>0,
$$  
where the last equality follows from $\P\cdot(\theta^{t})^{-1}=\P$. From the assumption it follows that $\psi_{\KK}(v,\x)=1$ for $\mu\otimes \Leb$-a.a.\ $(v,\x)\in \H(\T^3)\times \Dom$.
By the latter and invariance of $\mu\otimes \Leb$ for the Lagrangian process (with $\Pl_t$ as in \eqref{eq:Pl_definition}), 
$$
\textstyle{1=\int_{\H(\T^3)\times \T^3} 
\psi_{\KK}(v,\x)\, \dd \mu(v)\, \dd \x = 
\int_{\H(\T^3)\times \T^3} 
\Pl_t \psi_{\KK}(v,\x)\, \dd \mu(v)\, \dd \x} $$
 for all $t>0$ (in the following, we pick $t=1$).
Hence, $\Pl_1 \psi_{\KK}(v,\x)=1$ for $\mu\otimes \Leb$-a.a.\ $(v,\x)\in \H(\T^3)\times \Dom$, and by strong Feller property of Proposition \ref{prop:strong_feller_projective}, we have $\Pl_1 \psi_{\KK}(v,\x)=1$ for \emph{all} $(v,\x)\in \H(\T^3)\times \Dom$. Now, the independence of $\theta^1\om$ and $(u_1(\om),\x_1(\om))$ yield
$$
\textstyle{\int_{\O}\one_{\KK}(\theta^1\om,v_1(\om),\x_1(\om))\,\dd \P(\om)=
\Pl_1 \psi_{\KK}(v,\x)= 1}.
$$
Thus, $\one_{\KK}(\theta^1\cdot,v_1,\x_1)$ a.s. 
Hence, the conclusion follows by noticing that 
$(\om,v,\x)\in \KK$ if and only if $(\theta^1\om,v_1(\om),\x_1(\om))\in \KK$ due to the invertibility of $\nabla\phi^t_{\cdot,\cdot}(\cdot)$.

Theorem \ref{t:positive_lyapunov_exponent_II} follows similarly by considering the cocycle $\wc{\A}^t_{\om,(v,\x)}$ in \eqref{eq:cocycle_hat}. Let us emphasise that an analogue of Proposition \ref{prop:approximate_controllability}\eqref{it:approximate_controllability2} is not needed for $\wc{\A}_{\om,(v,\x)}^t$ because of the relation between $\wc{\A}_{\om,(v,\x)}^t$ and $\A_{\om,(v,\x)}^t$, see  \cite[Theorem 5.1.1]{A98_Random} and \cite[Proposition 3.17]{BBPS2022}. However, Propositions \ref{prop:strong_feller_projective} and \ref{prop:weak_irreducibility} are needed also for the process $(u_t,\x_t,\wc{\xi}_t)$  via Corollary \ref{cor:uniqueness_projective} to apply \cite[Theorem III.1.2]{K86_ergodic_theory}.
\end{proof}

In summary, we have derived Theorems \ref{t:positive_lyapunov_exponent} and \ref{t:positive_lyapunov_exponent_II} under the assumption that Propositions \ref{prop:RDS}-\ref{prop:weak_irreducibility} hold. The proofs of these propositions are presented in the following sections. Specifically, Proposition \ref{prop:RDS} is proved in Subsection \ref{ss:RDS}, Propositions \ref{prop:approximate_controllability} and \ref{prop:weak_irreducibility} are covered in Section \ref{s:irreduc_controllability}, and Proposition \ref{prop:strong_feller_projective} is established in Subsection \ref{ss:strong_feller_via_cut_off}. The results in Subsection \ref{ss:global_well_statement} are proven in Section \ref{s:well_posedness_Hsigma}. 
The subsequent section is devoted to the proof of Proposition \ref{prop:integrability_MET_L1}.

\section{Integrability of invariant measures}
\label{s:integrability_inv_measure_u}
In this section, we prove Proposition \ref{prop:integrability_MET_L1}. 
As discussed in Subsection \ref{ss:proofs_further_results}, Proposition \ref{prop:integrability_MET_L1} was used to obtain the existence of the (top) Lyapunov exponent in Theorems \ref{t:intro} and \ref{t:positive_lyapunov_exponent}.
Here, we prove a more refined result compared to Proposition \ref{prop:integrability_MET_L1} under minimal requirements on $Q$.
More precisely, in this subsection, we work under the following assumption.

\begin{assumption}
\label{ass:assumption_noise_integrability}
Suppose that $W$ and $Q$ are an $\Ls^2_0(\T^3)$-cylindrical Brownian motion and a Hilbert-Schmidt operator in $ \calL_2(\Ls_0^2(\T^3),\Hs_0^1(\T^3))$, respectively.
\end{assumption}

Existence of global unique $H^1$-solution to the stochastic PEs \eqref{eq:primitive_full} (see Definition \ref{def:solution}) under Assumption \ref{ass:assumption_noise_integrability} is standard, see e.g., 
\cite{primitive1,Debussche_2012}. 
For notational convenience, as in Proposition \ref{prop:integrability_MET_L1}, we let
$$
\Energy\stackrel{{\rm def}}{=}\|Q\|_{\calL_2(\Ls_0^2(\T^3),\Hs_0^1(\T^3))}.
$$
The following stronger version of Proposition \ref{prop:integrability_MET_L1} is the main result of this section.

\begin{theorem}[$L^p$-Integrability of the velocity field]
\label{t:integrability_MET}
Suppose that Assumption \ref{ass:assumption_noise_integrability} holds.
Let $\mu$ be an invariant measure on $\Hs_0^1(\T^3)$ for the stochastic \emph{PEs} \eqref{eq:primitive_full}. Then, for all $1\leq p<\frac{4}{3}$,
\begin{equation}
\label{eq:integrability_nabla_u_p}
\int_{\Hs_0^1(\T^3)}\int_{\Dom} |\nabla u(\x)|^p\,\dd \x\,\dd \mu(v)\leq C(p,\mathcal{E})
\end{equation}
where $u=(v,w(v))$ and $w(v)$ is uniquely determined by $v$ via the formula \eqref{eq:def_w}.
\end{theorem}

We expect that the estimate \eqref{eq:integrability_nabla_u_p} to fail in the usual case $p=2$ (the proof also suggests the optimality of the threeshold $p<\frac{4}{3}$). 
This contrasts with the case of the 2D Navier-Stokes equations, where \eqref{eq:integrability_nabla_u_p} with $p=2$ is a consequence of the energy inequality (see Proposition \ref{prop:exponential_integrability_mu} below).
Due to \eqref{eq:def_w}, another (at first glance) more direct approach to prove 
Proposition \ref{prop:integrability_MET_L1}
is to estimate the following quantity:
\begin{equation}
\label{eq:integrability_second_order_derivatives_v}
\int_{\Hs_0^1(\T^3)}\int_{\Dom} |\nabla^2 v(\x)|\,\dd \x\,\dd \mu(v)
\end{equation}
where $\mu$ is an invariant measure for the stochastic PEs \eqref{eq:primitive_full}. 
We do not know whether the above integral is finite. Indeed, there are serious problems when dealing with \eqref{eq:integrability_second_order_derivatives_v} related to the failure of maximal $L^1$-regularity estimates. The reader is referred to Remark \ref{r:L1_integrability_gradient_second_derivatives} for further comments. 
In our case, working directly with $u=(v,w(v))$, we can exploit the smoothing effects of $v\mapsto w(v)$, making Theorem \ref{t:integrability_MET} easier than understanding the boundedness of \eqref{eq:integrability_second_order_derivatives_v}.

\smallskip

The proof of Theorem \ref{t:integrability_MET} is given in Subsection \ref{ss:proof_integrability_MET} and requires several preliminary results. Here, we state some of them as they are of independent interest. As usual, the starting point is the following estimates, which are a consequence of the standard energy balances (see e.g.\ \cite[Appendix A]{HM06_annals} for a similar situation).

\begin{proposition}
\label{prop:exponential_integrability_mu}
Suppose that Assumption \ref{ass:assumption_noise_integrability} holds.
Let $\mu$ be an invariant measure on $\Hs_0^1(\T^3)$ for the stochastic \emph{PEs} \eqref{eq:primitive_full} and denote by $\Energy_0\stackrel{{\rm def}}{=}\|Q\|_{\calL_2(\Ls_0^2(\T^3),\Ls_0^2(\T^3))}$ the injected energy by the noise per unit of mass. Then,
\begin{align}
\label{eq:exponential_integrability_mu_1}
\int_{\Hs_0^1(\T^3)} \int_{\T^3} |\nabla v(\x)|^2\,\dd \x \,\dd \mu(v)&=\frac{\Energy_0}{2} , & &
\text{\emph{(Gradient integrability)}}\\
\label{eq:exponential_integrability_mu_2}
\int_{\Hs_0^1(\T^3)} \exp\Big(\eta \int_{\T^3} | v(\x)|^2\,\dd \x \Big)\,\dd \mu(v)&\leq 2,&& \text{\emph{(Exponential integrability)}}
\end{align}
where $\eta=1/(2\Energy_0)$.
\end{proposition}

The above might be known to experts. However, for the sake of completeness, we provide a sketch of the proof in Subsection \ref{ss:control_partialz_mu} below. 

\smallskip

Next, we state another result which lies at the core of the proof of Theorem \ref{t:integrability_MET}. 
It shows that any invariant measure for the stochastic PEs \eqref{eq:primitive_full} has more regularity in the vertical direction compared to the horizontal ones. 
This is in accordance with the central role of the vertical derivative $(\partial_z v_t)_t$ in the global well-posedness of 3D PEs \cite{A23_primitive3,CT07}. The control of $(\partial_z v_t)_t$ is specific to PEs and can be thought of as a consequence of the $z$-independence of the surface pressure $(p_t)_t$.

\begin{proposition}[Mixed integrability and smoothness of the vertical derivative]
\label{prop:regularity_partial_z_inv_measure}
Under the assumptions of Theorem \ref{t:integrability_MET}, for all $p\in (1,2)$ and $\gamma <\frac{2}{p}$,
\begin{equation}
\label{eq:vertical_derivative_mu_H}
\int_{\Hs_0^1(\T^3)}\int_{\T^2_{x,y}}\|v(x,y,\cdot)\|_{H^{\gamma}(\T_z)}^p\,\dd x \,\dd y\,\dd \mu(v)\leq C(p,\gamma,\Energy).
\end{equation}
In particular, for all $p<\frac{4}{3}$,
\begin{equation}
\label{eq:vertical_derivative_mu_sup}
\int_{\Hs_0^1(\T^3)} \int_{\T^2_{x,y}} \sup_{\T_z} |\partial_z v(x,y,\cdot)|^p\,\dd x \,\dd y\,\dd \mu(v)\leq  C(p,\Energy).
\end{equation}
\end{proposition}

The estimate \eqref{eq:vertical_derivative_mu_sup} follows from \eqref{eq:vertical_derivative_mu_H} by noticing that  $p<\frac{4}{3}$ ensures the existence of $\g<\frac{2}{p}$ satisfying $\g-1>\frac{1}{2}$ and thus $H^{\gamma}(\T_z)\embed C_{{\rm b}}^{1}(\T_z)$ by Sobolev embeddings.

Proposition \ref{prop:regularity_partial_z_inv_measure} is proven in Subsection \ref{ss:control_partialz_mu}. Its proof is based on $L^q$-maximal regularity results for the heat operator on anisotropic spaces of the form $L^{p}(\T^2_{x,y};L^{2}(\T_z))$; see Lemma \ref{l:max_reg_anisotropic_Laplacian}. A similar situation already appeared in \cite{A23_primitive3}, although with another aim.
The use of the anisotropic spaces with integrability $2$ in $z$ comes from the smoothing effect of the mapping $v\mapsto w(v)$ in the vertical direction. Indeed, for sufficiently smooth $v$, by the H\"older inequality it holds
\begin{align}
\label{eq:L1_estimate_mixed_explaination}
\|w(v)\partial_z v\|_{L^{1}(\T^2_{x,y};L^2(\T_z))}
&\lesssim \|w(v)\|_{L^2(\T_{x,y};L^\infty(\T_z))}
 \|\partial_zv\|_{L^2(\T^3)}\\
 &\nonumber
 \lesssim \|\nabla v\|_{L^2(\T^3)}^2.
\end{align}
The last term can be estimated by means of the energy inequality, i.e., Proposition \ref{prop:exponential_integrability_mu}. 
From the estimate \eqref{eq:L1_estimate_mixed_explaination}, one sees that we have at our disposal more integrability in the vertical direction than in the horizontal ones. The above inequality, along with Proposition \ref{prop:exponential_integrability_mu}, suggests the validity of \eqref{eq:vertical_derivative_mu_H} with $p$ close to $ 1$. 

\smallskip

Throughout this section, to shorthand the notation, we write $H^{s_1,p_1}_{x,y}(H^{s_2,p_2}_z)$ instead of $H^{s_0,p_0}(\T^2_{x,y};H^{s_1,p_1}(\T_z))$ and similar if $H^{s_i,p_i}$ is replaced by $L^{p_i}$.

\subsection{Auxiliary results}
\label{ss:aux_result_Lp}
In this subsection, we collect some useful auxiliary results.
We begin with an elementary lemma concerning the integral operator on the one-dimensional torus $\T=[0,1]$. For $f\in L^1(\T)$, let
\begin{equation}
\label{eq:underlined_operator}
\underline{f}\stackrel{{\rm def}}{=} \int_{\T}\Big(\int_{0}^{z} f(\zeta)\,\dd \zeta\Big)\,\dd z
= \int_{0}^1 f(z)\, (1-z) \,\dd z.
\end{equation}

\begin{lemma}
\label{l:estimate_H_negative_mean_free}
For all mean-zero functions $f\in L^2(\T)$, it holds that
\begin{equation}
\label{eq:estimate_H_negative_mean_free}
\Big\|\int_{0}^{\cdot} f(z)\, \dd z - \underline{f}\Big\|_{L^2(\T)}\lesssim \|f\|_{H^{-1}(\T)}.
\end{equation}
\end{lemma}

Before proving the above result, let us stress that the presence of $\underline{f}$ on the LHS\eqref{eq:estimate_H_negative_mean_free} is essential. Indeed, without $\underline{f}$ on the LHS\eqref{eq:estimate_H_negative_mean_free} the corresponding estimate and Lemma \ref{l:estimate_H_negative_mean_free} imply $|\underline{f}|\lesssim \|f\|_{H^{-1}(\T)}$. However, the latter is false as by duality it implies $z\mapsto z\not\in H^1(\T)\subseteq C(\T)$.

\begin{proof}
Let $\mathcal{F}$ be the Fourier transform on $\Z$. 
Since $\int_0^1 \big(\int_{0}^{\cdot} f(z)\, \dd z-\underline{f}\big)\,\dd z=0$, 
\begin{align*}
\Big\|\int_{0}^{\cdot} f(z)\, \dd z - \underline{f}\Big\|_{L^2(\T)}^2
&= \sum_{k\in \Z\setminus\{0\}} \Big|\mathcal{F}\Big(\int_{0}^{\cdot} f(z)\, \dd z - \underline{f}\Big)(k)\Big|^2\\
&= \sum_{k\in \Z\setminus\{0\}} \Big|\mathcal{F}\Big(\int_{0}^{\cdot} f(z)\, \dd z\Big)(k)\Big|^2\\
&\stackrel{(i)}{\eqsim} \sum_{k\in \Z\setminus\{0\}} \frac{1}{k^2}|\mathcal{F}(f)(k)|^2\eqsim\|f\|_{H^{-1}(\T)}^2,
\end{align*}
where in $(i)$ we used the mean-zero assumption $\int_{0}^1 f(z)\,\dd z=0$.
\end{proof}

Next, we establish maximal $L^q$-regularity estimates for the heat equation in spaces with anisotropic smoothness and integrability. They lie at the core of the proof of Proposition \ref{prop:regularity_partial_z_inv_measure}.

\begin{lemma}[Maximal $L^q$-regularity for the heat equation in anisotropic spaces]
\label{l:max_reg_anisotropic_Laplacian}
Let $1<q,p_0,p_1<\infty$ and $s_0,s_1\in \R$. Set 
$\X\stackrel{{\rm def}}{=}H^{s_0,p_0}_{x,y}(H^{s_1,p_1}_z)$ and let $\Delta_{\X}$ be the realization of the Laplace operator on $\X$ with domain
$$
\Do(\Delta_{\X})\stackrel{{\rm def}}{=} H^{s_0+2,p_0}_{x,y}(H^{s_1,p_1}_z)\cap H^{s_0,p_0}_{x,y}(H^{s_1+2,p_1}_z).
$$
Moreover, let $(e^{t\Delta})_{t\geq 0}$ be the heat semigroup on $\T^3$. 
Then 
\begin{align}
\label{eq:max_reg_anisotropic}
\Big\|t\mapsto \int_{0}^t e^{(t-s)\Delta} f(s)\,\dd s \Big\|_{L^q(0,1;\Do(\Delta_{\X}))}
&\lesssim \|f\|_{L^q(0,1;\X)},\\
\label{eq:max_reg_anisotropic_almost}
\Big\|t\mapsto \int_{0}^t e^{(t-s)\Delta} f(s)\,\dd s \Big\|_{L^1(0,1;\Do((-\Delta_{\X})^{\delta}))}
&\lesssim \|f\|_{L^1(0,1;\X)},  \ \ \delta\in (0,1).
\end{align}
Finally, for all $\delta\in (0,1)$ and $\theta_0,\theta_1>0$ such that $\theta_0+\theta_1=2\delta$,  
$$
\Do((-\Delta_{\X})^{\delta})\embed H^{s_0+\theta_0,q_0}_{x,y}(H^{s_1+\theta_1,q_1}_{z}).
$$
\end{lemma}

The above is known to experts on maximal regularity. We include the proof for the reader's convenience.

\begin{proof}
In the proof below, we use the notion of the $H^{\infty}$-calculus and $\mathcal{R}$-sectoriality for various versions of the Laplace operator. We refer to either \cite[Chapter 10]{Analysis2} or \cite[Chapters 3 and 4]{pruss2016moving} for details and notation.

We begin by proving \eqref{eq:max_reg_anisotropic}. 
Let 
$
A\stackrel{{\rm def}}{=} -\Delta_{x,y}
$
and $B\stackrel{{\rm def}}{=} -\partial^2_{z}$
be the realisation of the (minus) Laplace operator on $H^{s_0,p_0}_{x,y}(H^{s_1,p_1}_{z})$ with domain 
$$
\Do(A)=H^{s_0+2,p_0}_{x,y}(H^{s_1,p_1}_{z}) \quad \text{ and }\quad 
\Do(B)=H^{s_0,p_0}_{x,y}(H^{s_1+2,p_1}_{z}),
$$
respectively. Note that $A$ and $B$ are commuting (in the resolvent sense) and
$$
\Do(A)\cap \Do(B)= \Do(\Delta_{\X}).
$$
Moreover, $A$ and $B$ have $H^{\infty}$-calculus of angle 0 by the periodic version of \cite[Theorem 10.2.25]{Analysis2}. In particular the operator $A+B$ on $\Do(A)\cap \Do(B)$ is $\mathcal{R}$-sectorial of angle $0$ by \cite[Corollary 4.5.9]{pruss2016moving}. The maximal $L^q$-regularity estimate \eqref{eq:max_reg_anisotropic} now follows from the Weis' characterization, see e.g.\ \cite[Theorem 4.4.4]{pruss2016moving} or \cite{We}.

Due to \eqref{eq:max_reg_anisotropic}, the semigroup $(e^{t\Delta})_{t\geq 0}$ is analytic on $\X$ by \cite[Proposition 3.5.2(i)]{pruss2016moving}. Hence \eqref{eq:max_reg_anisotropic_almost} follows from Young's convolution inequality. 

It remains to prove the last assertion of Lemma \ref{l:max_reg_anisotropic_Laplacian}. The embedding essentially uses mixed-derivative estimates. Indeed, by \cite[Corollary 4.5.11]{pruss2016moving} for all $\theta\in (0,1)$,
\begin{align*}
\Do((-\Delta_{\X})^{\delta})
&\embed \Do(A^{\delta})\cap \Do(B^{\delta}) \\
&= H^{s_0+2\delta,q_0}_{x,y}(H^{s_1,q_1}_z)\cap H^{s_0,q_0}_{x,y}(H^{s_1+2\delta,q_1}_z)\\
&\embed  [H^{s_0+2\delta,q_0}_{x,y}(H^{s_1,q_1}_z), H^{s_0,q_0}_{x,y}(H^{s_1+2\delta,q_1}_z)]_{\theta}\\
&= H^{s_0+2\delta\theta}_{x,y} (H^{s_1+2\delta(1-\theta),q_1}_z),
\end{align*}
where $[\cdot,\cdot]_{\theta}$ denotes the complex interpolation functor, see e.g.\ \cite{BeLo}. The last claim follows by setting $\theta_0=2\theta\delta$ and $\theta_1=2\delta(1-\theta)$.
\end{proof}

\subsection{Proof of Propositions \ref{prop:exponential_integrability_mu} and \ref{prop:regularity_partial_z_inv_measure}}
\label{ss:control_partialz_mu}
We begin by providing a sketch of the proof of Proposition \ref{prop:exponential_integrability_mu}, which will then be used to prove Proposition \ref{prop:regularity_partial_z_inv_measure}.

\begin{proof}[Proof of Proposition \ref{prop:exponential_integrability_mu} -- Sketch]
Let $v\in \Hs^1_0(\T^3)$ and let $(v_t)_t$ be the global unique $H^1$-solution to \eqref{eq:primitive_full}, cf., Definition \ref{def:solution}. 
Let $\psi:[0,\infty)\to [0,\infty)$ be a bounded and smooth map. By It\^o formula applied to $v\mapsto \Psi_v\stackrel{{\rm def}}{=}\psi(\|v\|_{L^2}^{2})$ we have, a.s.\ for all $t>0$ and $n\geq 1$,
\begin{align*}
\psi(\|v_t\|_{L^2}^{2}) &+2\int_0^t \psi'(\|v_s\|_{L^2}^2)  \|\nabla v_s\|_{L^2}^{2}\,\dd s =\psi( \|v\|_{L^2}^{2}) \\
&+ \frac{1}{2}
\sum_{(k,\ell)\in \z} \int_0^t  D^2\Psi_{v_s} (Q\g_{(\kk,\ell)}e_{\kk}, Q\g_{(\kk,\ell)}e_{\kk})\,\dd s + \mathrm{M}_t,
\end{align*}
where $\mathrm{M}_t$ is a integrable martingale starting at zero and $(\g_{(\kk,\ell)}e_{\kk})_{(k,\ell)\in\z}$ is the orthonormal basis of $\Ls^2_0(\T^3)$ as described in Subsection \ref{sss:probabilistic_set_up}. 
In obtaining the above, we also used the standard cancellation $\int_{\T^3} (u_t\cdot\nabla) v_t\cdot v_t\,\dd x=0$ a.e.\ on $\R_+\times \O$ since $u_t=(v_t,w(v_t))$ satisfies $\nabla\cdot u_t=0$ and $w(v_t)=0$ on $\T^2\times \{0,1\}$ by \eqref{eq:def_w}. 

Taking $\int_{\Hs^1_0(\T^3)} \E [\cdot] \, \dd \mu(v)$ and using the invariance of $\mu$, we obtain for all $t>0$,
\begin{align*}
&2t\int_{\Hs^1_0(\T^3)}  \psi'(\|v\|_{L^2}^2)  \|\nabla v\|_{L^2}^{2}\, \dd \mu(v) \\
&\qquad \qquad = \frac{1}{2}
\sum_{(k,\ell)\in \z} \int_{\Hs^1_0(\T^3)}  \int_0^t  D^2\Psi(v_s) (Q\g_{(\kk,\ell)}e_{\kk}, Q\g_{(\kk,\ell)}e_{\kk})\,\dd s\,\dd \mu(v).
\end{align*}
Now, let us fix $n\geq 1$ and take a sequence of maps $(\psi_k)_{k\geq 1}$ such that $|\psi_k (x)|\lesssim_k x^2$ and $\lim_{k\to \infty}\psi_k(x)\to x^{n}$ for all $x\in [0,\infty)$. Then, by applying the above identity to such a sequence, one can readily check that, for all $n\geq 1$,
\begin{align}
\label{eq:energy_estimate_k}
\int_{\Hs^1_0(\T^3)} \| v\|_{L^2}^{2(n-1)} \|\nabla v\|_{L^2}^{2}\,  \dd \mu(v)
&= 
(2n-1)\frac{\Energy_0}{2}\int_{\Hs^1_0(\T^3)} \| v\|_{L^2}^{2(n-1)} \,  \dd \mu(v).
\end{align}
In particular, \eqref{eq:exponential_integrability_mu_1} is equivalent to \eqref{eq:energy_estimate_k} with $n=1$. 

Next, letting $\mathcal{I}_n\stackrel{{\rm def}}{=}\int_{\Hs^1_0(\T^3)} \| v\|_{L^2}^{2n}\,  \dd \mu(v)$, the identity \eqref{eq:energy_estimate_k} and the Poincar\'e inequality yield 
$$
\mathcal{I}_n \leq n\,\Energy_0\,\mathcal{I}_{n-1}\leq \dots \leq n! \,\Energy_0^n. 
$$
Hence, if $\eta=1/(2\Energy_0)$, then 
$$
\int_{\Hs_0^1(\T^3)} \exp\Big(\eta \int_{\T^3} | v(\x)|^2\,\dd \x \Big)\,\dd \mu(v)
= \sum_{n\geq 1} \frac{\eta^n}{n!}
\int_{\Hs^1_0(\T^3)} \| v\|_{L^2}^{2n} \,  \dd \mu(v)
\leq 2,
$$
as desired.
\end{proof}

Before going into the proof of Proposition \ref{prop:regularity_partial_z_inv_measure}, let us state the following result that is a straightforward consequence of the definitions in \eqref{eq:def_p_q} and will be used below without further mention.

\begin{lemma}
Let $1<p_0,p_1<\infty$ and $s_0,s_1\in \R$. Then $\q,\p\in \calL(H^{s_0,p_0}_{x,y}(H^{s_1,p_1}_z))$.
\end{lemma}

\begin{proof}[Proof of Proposition \ref{prop:regularity_partial_z_inv_measure}]
As we already noticed that \eqref{eq:vertical_derivative_mu_sup} follows \eqref{eq:vertical_derivative_mu_H}, it suffices to prove the latter. 
The proof is split into three steps.

\smallskip

\emph{Step 1: Let $X$ be a reflexive Banach space and let $d\geq 1$ be an integer. Then for all $r<0$ and $q\in (1,\infty)$ satisfying  $- d > r-\frac{d}{q}$ we have }
\begin{equation}
\label{eq:L1_embeds_negative}
L^1(\T^d;X)\embed H^{r,q}(\T^d;X). 
\end{equation}
The above is known to experts. However, for the reader's convenience, we include a short proof based on a duality argument. Note that the condition $- d > r-\frac{d}{q}$ is equivalent to $-r-\frac{d}{q'}>0$ where $q'= \frac{q}{q-1}$. By the vector-valued Sobolev embeddings (see e.g.\ \cite[Proposition 7.4]{MV12}) we have 
\begin{equation}
\label{eq:embedding_to_dualize}
H^{-r,q'}(\T^d;X^*)\embed C(\T^d;X^*).
\end{equation}
Since $L^1(\T^d;X)\embed (C(\T^d;X^*))^*$,  the claimed embedding follows by taking the dual of the embedding \eqref{eq:embedding_to_dualize} and using $(H^{-r,q'}(\T^d;X^*))^*=H^{r,q}(\T^d;X)$ by reflexivity of $X$, see e.g.\ \cite[Theorem 1.3.21 and Proposition 5.6.7]{Analysis1}.

\smallskip

\emph{Step 2: For $\alpha<2$, it holds that}
\begin{equation}
\label{eq:claim_step2_integrability_vertical_derivative}
\int_{\Hs_0^1(\T^3)}\|v\|_{L^{1}_{x,y}(H^{\alpha}_z)}\,\dd \mu(v)\leq C(\alpha,\Energy).
\end{equation}
Let $(v_t)_t$ be the strong solution to the stochastic PEs \eqref{eq:primitive_full}. One can readily check that it satisfies the mild formulation:
\begin{align}
\label{eq:PE_mild_formulation}
v_t=e^{t\Delta} v_0 + \int_{0}^t e^{(t-s)\Delta} b(v_s,v_s)\,\dd s + \int_{0}^t e^{(t-s)\Delta}Q\, \dd W_s
\end{align}
a.s.\ for all $t\geq 0$, where $b$ is the nonlinearity given by 
\begin{equation*}
b(v_1,v_2)=\p[(v_1\cdot\nabla_{x,y})v_2+w(v_1)\partial_z v_2 ].
\end{equation*}
Next, we would like to apply the estimate \eqref{eq:max_reg_anisotropic_almost} in Lemma \ref{l:max_reg_anisotropic_Laplacian} with an appropriate choice of $(s_0,p_0)$, while we choose $p_1=2$ and $s_1=0$. Fix $s_0\in (0,2-\mu)$ and let $p_0\in (1,2)$ be such that $2<-s_0+\frac{2}{p}$.

Set $\delta=(\alpha+s_0)/2<1$ and let $\X=H^{s_0,p_0}_{x,y}(L^2_z)$ be as in Lemma \ref{l:max_reg_anisotropic_Laplacian}.
By taking $L^1(0,1;\Do((-\Delta_{\X})^{\delta}))$-norms in the mild-formulation \eqref{eq:PE_mild_formulation}, using the estimate \eqref{eq:max_reg_anisotropic_almost} and taking $\int_{\Hs^1(\T^3)}\cdot\,\dd \mu(v)$,
\begin{align*}
\int_{\Hs_0^1(\T^3)} \int_0^1 \|v_t\|_{\Do((-\Delta_{\X})^{\delta})}\,\dd t\,\dd \mu(v) 
& \lesssim_\alpha 
\int_{\Hs_0^1(\T^3)}\|v\|_{H^1}\,\dd \mu(s) \\
&+\int_{\Hs_0^1(\T^3)}\int_0^1 \|b(v_t,v_t)\|_{\X} \,\dd t\,\dd \mu(v)+ C(\Energy).
\end{align*}
Next, we analyse the terms on the RHS of the above equation separately. Firstly, note that $
\int_{\Hs^1(\T^3)}\|v\|_{H^1}\,\dd \mu(s) \leq C(\Energy)$ by Proposition \ref{prop:exponential_integrability_mu}. Secondly, by the invariance of $\mu$,
$$
\int_{\Hs_0^1(\T^3)} \int_0^1 \|v_t\|_{\Do((-\Delta_{\X})^{\delta})}\,\dd t\,\dd \mu(v) 
= \int_{\Hs_0^1(\T^3)} \|v\|_{\Do((-\Delta_{\X})^{\delta})}\,\dd \mu(v) .
$$
Thirdly, by Step 1 and the choice of $(s_0,p_0)$, we obtain (cf., \eqref{eq:L1_estimate_mixed_explaination})
\begin{align*}
&\int_{\Hs_0^1(\T^3)}\int_0^1 \|b(v_t,v_t)\|_{\X} \,\dd t\,\dd \mu(v)\\
&\lesssim
\int_{\Hs_0^1(\T^3)}\int_0^1 \big(\|(v_t\cdot \nabla )v_t \|_{L^1_{x,y}(L^2_z)}+ \|w(v_t)\partial_z v_t\|_{L^1_{x,y}(L^2_z)}\big) \,\dd t\,\dd \mu(v)\\
&\stackrel{(i)}{\lesssim}
\int_{\Hs_0^1(\T^3)}\int_0^1 \big(\|v_t\|_{L^2_{x,y}(L^\infty_z)}\|v_t\|_{H^1}+ 
\|w(v)\|_{L^2_{x,y}(L^\infty_z)} \|v\|_{H^1}\big)\, \dd t \,\dd \mu(v)\\
&\stackrel{(ii)}{\lesssim}
\int_{\Hs_0^1(\T^3)}\int_0^1 \|v_t\|_{H^1}^2\, \dd t \,\dd \mu(v)\\
& \stackrel{(ii)}{=}
\int_{\Hs_0^1(\T^3)} \|v\|_{H^1}^2\, \dd\mu(v)\\
&\stackrel{(iv)}{\leq } C(\Energy)
\end{align*}
where in $(i)$ used that $|w(v)|\lesssim \int_{\T_z} |\nabla_{x,y}\cdot v |\,\dd z$ on $\T^3$ by \eqref{eq:def_w}, in $(ii)$ that $
H^1\embed 
L^2_{x,y}(H^1_z)\embed L^r(L^\infty_{z})
$, in $(iii)$ the invariance of $\mu$ and in $(iv)$ Proposition \ref{prop:exponential_integrability_mu}.

Collecting the above estimates, we have obtained
\begin{equation*}
 \int_{\Hs_0^1(\T^3)} \|v\|_{\Do((-\Delta_{\X})^{\delta})}\,\dd \mu(v)\leq C(\alpha,\Energy).
\end{equation*}
Now, \eqref{eq:claim_step2_integrability_vertical_derivative} follows from the above and the last assertion in Lemma \ref{l:max_reg_anisotropic_Laplacian}.

\smallskip

\emph{Step 3: Proof of \eqref{eq:vertical_derivative_mu_H}}. Again, by Proposition \ref{prop:exponential_integrability_mu} and $H^1\embed L^2_{x,y}(H^1_z)$, we have 
\begin{equation*}
\int_{\Hs_0^1(\T^3)} \|v\|_{L^2_{x,y}(H^1_z)}^2\,\dd \mu(v)\leq C(\Energy).
\end{equation*}
By interpolating the above and \eqref{eq:claim_step2_integrability_vertical_derivative} (see e.g., \cite[Theorem 2.2.6]{Analysis1}), it holds that 
$$
\int_{\Hs_0^1(\T^3)} \|v\|_{L^{p_\theta}_{x,y}(H^{s_\theta}_z)}^{p_\theta}\,\dd \mu(v)\leq C(\theta,\alpha,\Energy).
$$ 
where $\theta\in (0,1)$ and $p_\theta=\frac{2-\theta}{2}$, $s_\theta=\alpha +\theta(1-\alpha)$ with $\alpha<2$.
In particular, $\theta=\frac{2}{p_\theta'}$ and $s_\theta=\alpha +\frac{2}{p_\theta'}(1-\alpha)$ where $p_\theta'$ is the conjugate exponent of $p$. The estimate \eqref{eq:vertical_derivative_mu_H} now follows by relabelling the parameters and choosing $\alpha\sim 2$ depending on $\g$.
\end{proof}

\begin{remark}[On the $\mu\otimes \Leb$-integrability of $\nabla^2 v$]
\label{r:L1_integrability_gradient_second_derivatives}
The above proof of Proposition \ref{prop:regularity_partial_z_inv_measure} would have given us the finiteness of the integral \eqref{eq:integrability_second_order_derivatives_v} in case we could choose $s_0=0$, $p_0=1$ and $\delta=1$ in the proof (still keeping $s_1=0$ and $p_1=2$). By inspecting the proof of Lemma \ref{l:max_reg_anisotropic_Laplacian} and in light of known results on maximal $L^1$-regularity \cite[Subsections 17.3.b and 17.4.b]{Analysis3}, Lemma \ref{l:max_reg_anisotropic_Laplacian} with $s_0=0$, $p_0=1$ and $\delta=1$ does not hold. Thus, it is open to determine whether \eqref{eq:integrability_second_order_derivatives_v} is finite or not. 
\end{remark}

\subsection{Proof of Theorem \ref{t:integrability_MET}}
\label{ss:proof_integrability_MET}
To prove Theorem \ref{t:integrability_MET}, we need the following consequence of Proposition \ref{prop:exponential_integrability_mu}. 

\begin{lemma}
\label{l:mixed_integrability_invariant_mesures}
Under the assumptions of Theorem \ref{t:integrability_MET}, for all $p<4$, 
\begin{equation}
\label{eq:lemma_estimate_L4xy_L2z}
\int_{\Hs_0^1(\T^3)}
\|v\|_{L^p_{x,y}(L^2_z)}^p \,\dd \mu(v)\leq C(p,\Energy).
\end{equation}
\end{lemma}

\begin{proof}
Fix $r<\infty$. 
Note that, by \eqref{eq:exponential_integrability_mu_2}, it follows that 
\begin{equation}
\label{eq:lemma_estimate_L4xy_L2z1}
\int_{\Hs_0^1(\T^3)}  \|v\|_{L^2}^r \,\dd \mu(v)
\leq C(r,\Energy).
\end{equation}
Moreover, by Sobolev embeddings, it follows that $H^1\embed H^1_{x,y}(L^2_z)\embed L^r_{x,y}(L^2_z)$. Hence, \eqref{eq:exponential_integrability_mu_1} implies
\begin{equation}
\label{eq:lemma_estimate_L4xy_L2z2}
\int_{\Hs_0^1(\T^3)}  \|v\|_{L^r_{x,y}(L^2_z)}^2 \,\dd \mu(v)
\leq C(r,\Energy).
\end{equation}
Interpolating \eqref{eq:lemma_estimate_L4xy_L2z1}-\eqref{eq:lemma_estimate_L4xy_L2z2} with $r=r(p)<\infty$ appropriately, we obtain \eqref{eq:lemma_estimate_L4xy_L2z}. 
\end{proof}

We are now ready to prove Theorem \ref{t:integrability_MET}.

\begin{proof}[Proof of Theorem \ref{t:integrability_MET}]
We begin by collecting some useful facts. 
Throghout the proof we fix $p<\frac{4}{3}$.
Recall that $u=(v,w(v))$ where $w(v)$ is as in \eqref{eq:def_w}.
Due to Proposition \ref{prop:exponential_integrability_mu}, to prove Theorem \ref{t:integrability_MET} it suffices to show that
\begin{equation}
\label{eq:claim_integrability_inv_meas}
\int_{\Hs_0^1(\T^3)} \int_{\T^3} \Big|\nabla_{x,y}\big( \nabla_{x,y}\cdot \int_{0}^{\cdot} v (\cdot,z)\,\dd z \big)\Big|^p \,\dd x \leq C(p,\Energy).
\end{equation}
The idea is to use the following splitting:
\begin{align*}
\int_{0}^{\cdot} v (\cdot,z)\,\dd z
&=
\Big(\int_{0}^{\cdot} v (\cdot,z)\,\dd z- \underline{v}\Big)+ \underline{v}\ \ \text{ where } \ \  \underline{v}\stackrel{{\rm def}}{=}\int_{\T_z} v(\cdot,z) z \,\dd z. 
\end{align*}
Hence, by Lemma \ref{l:estimate_H_negative_mean_free} and the divergence-free type-condition $\int_{\T_z}( \nabla_{x,y}\cdot v)(\cdot,z)\,\dd z=0$ on $\T_{x,y}^2$, to prove \eqref{eq:claim_integrability_inv_meas} it is enough to show that
\begin{align}
\label{eq:inv_measure_int_1}
\int_{\Hs_0^1(\T^3)}\int_{\T^2_{x,y}}\|\nabla_{x,y}^2 v(x,y,\cdot)\|_{H^{-1}_{z}}^p\, \dd x \, \dd y\,\dd \mu(v)& \leq C(p,\Energy),\\
\label{eq:inv_measure_int_2}
\int_{\Hs_0^1(\T^3)}\int_{\T^2_{x,y}} |\nabla_{x,y}^2 \vu |^p\,\dd x \, \dd y \,\dd \mu(v)& \leq C(p,\Energy).
\end{align}
We now split the proof into two steps.

\smallskip

\emph{Step 1: Proof of \eqref{eq:inv_measure_int_1}.}
Since $\nabla\cdot u=0$ for $v\in \Hs^1$, we have
$$
(v\cdot\nabla_{x,y}) v+ w(v)\partial_z v =\nabla_{x,y}\cdot ( v\otimes v) + \partial_z (w(v)v ). 
$$
Arguing as in the proof of Proposition \ref{prop:regularity_partial_z_inv_measure}, the mild formulation of the primitive equations and the maximal $L^q$-regularity estimate \eqref{eq:max_reg_anisotropic} of Lemma \ref{l:max_reg_anisotropic_Laplacian} with $s_0=0$, $p_0=q=p<4/3$, $s_1=-1$, $p_1=2$ and the stationary of $\mu$,
\begin{align*}
&\int_{\Hs_0^1(\T^3)} \int_{\T^2_{x,y}}\|\nabla_{x,y}^2 v(x,y,\cdot)\|_{H^{-1}_{z}}^{p}\, \dd x \, \dd y\,\dd \mu(v)
\lesssim_p C(\Energy)+
\int_{\Hs_0^1(\T^3)}\|v\|_{H^1}^2\,\dd\mu(v)\\
&\qquad + 
\int_{\Hs_0^1(\T^3)} \|\nabla_{x,y}\cdot (v\otimes v)\|_{L^{p}_{x,y}(H^{-1}_z)}^{p}\,\dd\mu(v)+
\int_{\Hs_0^1(\T^3)} \|\partial_z (w(v)v)\|_{L^{p}_{x,y}(H^{-1}_z)}^{p}\,\dd\mu(v).
\end{align*}
Next, we estimate the last two terms on the RHS of the above estimate. Firstly, choose $q<4$ such that $\frac{1}{q}+\frac{1}{2}=\frac{1}{p}$. Due \eqref{eq:L1_embeds_negative}, we have
\begin{align*} 
 \|\nabla_{x,y}\cdot (v\otimes v)\|_{L^{p}_{x,y}(H^{-1}_z)}
&\lesssim
\|\nabla_{x,y}\cdot (v\otimes v)\|_{L^{p}_{x,y}(L^1_z)}\\
&\lesssim  \|v\|_{L^q_{x,y}(L^2_z)}\|\nabla v\|_{L^2}. 
\end{align*}
In particular, by \eqref{eq:exponential_integrability_mu_1} and \eqref{eq:lemma_estimate_L4xy_L2z},
$$
\int_{\Hs_0^1(\T^3)} \|\nabla_{x,y}\cdot (v\otimes v)\|_{L^{p}_{x,y}(H^{-1}_z)}^{p}\,\dd\mu(v)
\leq C(p,\Energy).
$$
Similarly, arguing as in \eqref{eq:L1_estimate_mixed_explaination}, for $q$ as above,
\begin{align*}
\|w(v)v\|_{L_{x,y}^{p}(H^{-1}_z)}
&\lesssim
\|w(v)v\|_{L_{x,y}^{p}(L^2_z)}\\
&\lesssim \|w(v)\|_{L^2_{x,y}(L^{\infty}_z)}\|v\|_{L^r_{x,y}(L^2_z)}\\
&\lesssim \|\nabla v\|_{L^2}\|v\|_{L^r_{x,y}(L^2_z)}.
\end{align*}
Hence, $\int_{\Hs_0^1(\T^3)} \|\partial_z (w(v)v)\|_{L^{p}_{x,y}(H^{-1}_z)}^{p}\,\dd\mu(v)
\leq C(p,\Energy)$ by H\"older inequality. 
Thus, \eqref{eq:inv_measure_int_1} follows by collecting the previous estimates.

\smallskip

\emph{Step 2: Proof of \eqref{eq:inv_measure_int_2}.}
Let us begin by deriving an SPDE for $\vu_t\stackrel{{\rm def}}{=} \vut=\int_{\T^3} v_t\, (1-z)\,\dd z$, cf., \eqref{eq:underlined_operator}. From  \eqref{eq:primitive_full}, we obtain 
\begin{align*}
\partial_t \vu_t
&+ \underline{(v_t\cdot \nabla_{x,y})v_t }
+\underline{w(v_t)\partial_z v_t} \\
&= \frac{1}{2}
\q [(v\cdot\nabla_{x,y})v+w(v)\partial_z v] + \Delta_{x,y}\underline{v} + \underline{\partial_z^2 v}   + \underline{Q}\dot{W}, \quad   \text{on }\T^3.
\end{align*}
Here, we used that the pressure $-\q[f]$ does not depend on $z$ by \eqref{eq:def_p_q}.

Let $p<\frac{4}{3}$ be fixed.
For $v\in \Hs^2(\T^3)$, we have $w(v)=0$ on $\T_{x,y}^2\times \{0,1\}$ and therefore
\begin{align*}
\underline{w(v)\partial_z v}
&= \int_{\T_z} \big(w(v)\partial_z v \big) (1-z)\,\dd z\\
&= \int_{\T_z} w(v) v \,\dd z+ \int_{\T_z}(\nabla_{x,y} \cdot v) v\, (1-z) \,\dd z.
\end{align*}
Similarly, for $v\in \Hs^2(\T^3)\subseteq C(\T^3;\R^2)$,
\begin{align*}
\underline{\partial_z^2 v}
= -\partial_z v(\cdot,0) +\int_{\T_z} \partial_z v \,\dd z=  -\partial_z v(\cdot,0).
\end{align*}
It remains to estimate the pressure term $\q [(v_t\cdot\nabla_{x,y})v_t+w(v_t)\partial_z v_t]$.
Finally, let $\q_{x,y}$ be as below \eqref{eq:def_q_xy}. Note that, for $v\in \Hs^2(\T^3)$, 
\begin{align*}
\q [(v\cdot\nabla_{x,y})v+w(v)\partial_z v]
&= \q_{x,y} \Big[\int_{\T_z}\Big((v\cdot\nabla_{x,y})v+w(v)\partial_z v\Big)\, \dd z \Big]\\
&= \q_{x,y} \Big[ \int_{\T_z}(v\cdot\nabla_{x,y})v\,\dd z + 
\int_{\T_z}  (\nabla_{x,y}\cdot v ) v\,\dd z \Big],
\end{align*}
where the last equality follows from the identity
$$
\int_{\T_z} w(v)\partial_z v\,\dd z= 
\int_{\T_z}  (\nabla_{x,y}\cdot v ) v\,\dd z .
$$ 
Reasoning as the proof of Proposition \ref{prop:regularity_partial_z_inv_measure} and Step 1, one can check that the maximal $L^p$-regularity estimates (see Lemma \ref{l:max_reg_anisotropic_Laplacian} for a more general situation) for the Laplace operator $\Delta_{x,y}$ on $\T_{x,y}^2$, and using the invariance of $\mu$, it follows that
\begin{align*}
\int_{\Hs_0^1} \|\nabla_{x,y}^2 \vu\|_{L^{p}_{x,y}}^{p}\,\dd \mu(v)
&\lesssim_p 1+
\int_{\Hs_0^1}\Big( \|v\|_{H^1}^2
 +\|w(v)v\|_{L^{p}_{x,y}(L^1_z)}^{p}\Big) \,\dd \mu(v)\\
&+ \int_{\Hs_0^1}\Big(\||\nabla v| |v|\|_{L^{p}_{x,y}(L^1_z)}^{p} +
 \sup_{z\in \T}\|\partial_z v(\cdot,z)\|_{L^{p}_{x,y}}^{p}\Big)\,\dd \mu(v)+ C(\Energy)\\
&\leq C(p,\Energy)
\end{align*}
where in the last inequality follows from Lemma \ref{l:mixed_integrability_invariant_mesures} and \eqref{eq:vertical_derivative_mu_sup} in Proposition \ref{prop:regularity_partial_z_inv_measure}.
\end{proof}

\section{Global well-posedness and unique ergodicity of PEs}
\label{s:well_posedness_Hsigma}
In this section, we prove the results stated in Subsection \ref{ss:global_well_statement}.

\subsection{Global well-posedness of stochastic PEs -- Proof of Propositions \ref{prop:global_well_posedness} and \ref{prop:regularity_invariant_measure}}
Here, we prove Propositions \ref{prop:global_well_posedness} and \ref{prop:regularity_invariant_measure}.
The latter is a consequence of the following result, which includes the key \emph{new} a priori estimates for stochastic PEs \eqref{eq:primitive_full} in high-order Sobolev spaces.

\begin{theorem}[Global well-posedness in high-order regularity spaces]
\label{t:high_order_estimates}
Let $\sigma\geq 2$, $Q\in \calL_2(\Ls^2_0,\Hs^\sigma_0)$ and $v\in  L^2_{\F_0}(\O;\Hs^{\sigma}_0)$. 
There exists a global unique $H^{\sigma}$-solution $(v_t)_{t}$ such that
$
(v_t)_t\in C([0,\infty);H^{\sigma})
$ a.s.\
 for which the following hold.
\begin{enumerate}[{\rm(1)}]
\item\label{it:high_order_estimates2}
For all $T\in (0,\infty)$ and $\g\gg 1$,
$$
\P\Big(\sup_{t\in [0,T]} \|v_t\|_{H^\sigma}^2 + \int_{0}^T \|v_t\|_{H^{1+\sigma}}^2 \,\dd t\geq \g\Big)
\lesssim_T 
\frac{1+\|v\|_{H^{\sigma}}^4}{\log\log\log(\g)} .
$$
\item\label{it:high_order_estimates3} For all sequences $(v^n)_{n\geq 1}\subseteq L^2_{\F_0}(\O;\Hs^\sigma)$ satisfying $v^n\to v$ in probability in $\Hs^\sigma$ it holds that 
$$
(v^n_t)_t\to (v_t)_t \text{ in probability in }C([0,T];H^\sigma)\cap L^2(0,T;H^{1+\sigma})
$$
for all $T<\infty$, where $(v^n_t)_t$ is the global solution to \eqref{eq:primitive_full} with initial data $v^n$.
\end{enumerate}
\end{theorem}

The case $\sigma=1$ of Theorem \ref{t:high_order_estimates} is also true, see e.g.\ \cite[Theorem 3.6]{primitive1} and \cite[Remark 3.10]{primitive2}. Item \eqref{it:high_order_estimates2} follows from bootstrapping the latter estimate to high-order Sobolev regularity. The latter fact is non-trivial due to the above-mentioned criticality of $H^1$ for the PEs.
As for Proposition \ref{prop:global_well_posedness}, the above result is not present in the existing literature on stochastic PEs. 
The proof of Theorem \ref{t:high_order_estimates} is rather general and can be adapted to the case of multiplicative noise as in \cite{primitive1,primitive2}. Moreover, by localization in $\Omega$ (see e.g., \cite[Proposition 4.13]{AV19_QSEE_2}), Theorem \ref{t:high_order_estimates} also holds with $L^0_{\F_0}(\O;\Hs_0^\sigma)$-intial data and non-zero mean initial data and noise. As it will not be needed here, we do not include the details.

The arguments in the proof of Theorem \ref{t:high_order_estimates} are inspired by the corresponding results in the deterministic setting presented in \cite[Section 3]{LT19}. However, in this work, we take a more direct approach by relying on the nonlinear estimate in Lemma \ref{l:estimate_nonlinearity}. Specifically, our approach reveals that the special structure of the PEs nonlinearity plays only a minor role in improving the estimate from $H^1$ to $H^2$.

The proof of Theorem \ref{t:high_order_estimates} is postponed to Subsection \ref{ss:high_order_estimates}. In the next subsection, we analyse the mapping properties of the nonlinearities in the PEs \eqref{eq:primitive_full}. 

\subsubsection{Sobolev estimates for the PEs nonlinearity}
Set
\begin{equation}
\label{eq:def_nonlinearity_b}
b(v_1,v_2)\stackrel{{\rm def}}{=} \p\big[(v_1\cdot\nabla_{x,y})v_2+ w(v_1)\partial_z v_2\big],
\end{equation}
where $w(v)$ is as in \eqref{eq:def_w}. In this subsection, we prove the following result.

\begin{lemma}
\label{l:estimate_nonlinearity}
The following estimates hold whenever the right-hand side is finite and $v_1,v_2\in \Ls^2$:
\begin{enumerate}[{\rm(1)}]
\item\label{it:estimate_nonlinearity_L2}
$\displaystyle{
\|b(v_1,v_2)\|_{L^2}
\lesssim \|v_1\|_{H^{3/2}}\|v_2\|_{H^{3/2}};
}$
\vspace{0.1cm}
\item\label{it:estimate_nonlinearity_H10}
$\displaystyle{\|b(v_1,v_2)\|_{H^1}
\lesssim \|v_1\|_{H^{5/2}}\|v_2\|_{H^{3/2}}}
+\|v_1\|_{H^{3/2}}\|v_2\|_{H^{5/2}};
$
\vspace{0.1cm}
\item\label{it:estimate_nonlinearity_H1}
$\displaystyle{\|b(v_1,v_2)\|_{H^\rho}
\lesssim_{\rho,\delta} \|v_1\|_{H^{1+\rho+\delta}}\|v_2\|_{H^{1+\rho+\delta}}}
$
for all $\delta>0$ and $\rho\geq 1$;
\vspace{0.1cm}
\item\label{it:estimate_nonlinearity_High}
$\displaystyle{\|b(v_1,v_2)\|_{H^\rho}
\lesssim_{\rho} \|v_1\|_{H^{1+\rho}}\|v_2\|_{H^{1+\rho}}}
$
for all $\rho\geq 2$.
\end{enumerate}
\end{lemma}

Employing the terminology of \cite{AV19_QSEE_1,AV19_QSEE_2}, the estimate \eqref{it:estimate_nonlinearity_L2} shows that the nonlinearity in \eqref{eq:primitive} is critical with ground space is $L_t^2(L^2_x)$ (corresponding to initial data in $\Hs^1$), while \eqref{it:estimate_nonlinearity_L2} and \eqref{it:estimate_nonlinearity_H1}-\eqref{it:estimate_nonlinearity_High} ensure that this is not the case when the group space is $L_t^2(H^\rho_x)$ with $\rho> 0$. Finally, \eqref{it:estimate_nonlinearity_H10} will be the key to bootstrap the estimate to bootstrap the energy estimate in \cite{primitive2} from $H^1$-data to $H^{\sigma}$ with $\sigma\geq 2$ leading to Theorem \ref{t:high_order_estimates}\eqref{it:high_order_estimates2}.

\begin{proof}
Below, we deal with the most difficult part of the nonlinearity $b$, i.e.\ 
$$
b_0(v_1,v_2)\stackrel{{\rm def}}{=} \p \big[w(v_1)\partial_z v_2\big].
$$
The estimates for $(v_1\cdot\nabla_{x,y})v_2$ are similar and simpler. For simplicity, we sometimes write $H_ {x,y}^{s,p}(X)$ for the Sobolev spaces $H^{s,p}(\T^2_{x,y};X)$ (here $X$ is a Banach space).

\smallskip

\eqref{it:estimate_nonlinearity_L2}: 
The H\"{o}lder inequality and 
the Sobolev embedding $H^{1/2}_{x,y}\embed L^4_{x,y}$ yield 
\begin{align*}
\|b_0(v_1,v_2)\|_{L^2}
&\lesssim \|w(v_1)\|_{L^{\infty} (\T_z;L^4(\T^2_{x,y}))}\|\partial_z v_2\|_{L^2(\T_z;L^4(\T^2_{x,y}))}\\
&\lesssim \|\nabla_{x,y} v_1\|_{L^{2} (\T_{z};L^4(\T^2_{x,y}))}\|\partial_z v_2\|_{L^2(\T_z;L^4(\T_{x,y}^2))}\\
&\lesssim \|\nabla_{x,y} v_1\|_{L^{2} (\T_{z};H^{1/2}(\T^2_{x,y}))}\|\partial_z v_2\|_{L^2(\T_z;H^{1/2}(\T_{x,y}^2))}\\
&\lesssim \|v_1\|_{H^{3/2}(\T^3)}\|v_2\|_{H^{3/2}(\T^3)}.
\end{align*}

\smallskip

\eqref{it:estimate_nonlinearity_H10}: By density, we assume $v_1,v_2\in \Hs^3(\T^3)$. Note that, for all $v\in \Hs^3(\T^3)$,
$$
\partial_z w(v)=-\nabla_{x,y}\cdot v, \qquad \text{ as }\qquad 
\textstyle{\int}_{\T_z}\nabla_{x,y}\cdot v (\cdot,z)\,\dd z=0.
$$
Hence, the estimate follows by \eqref{it:estimate_nonlinearity_L2} and 
\begin{align}
\label{eq:gradient_indenity_b0_1}
\nabla_{x,y} \p[w(v_1)\partial_z v_2]
&= \p [w(\nabla_{x,y} v_1) \partial_z v_2] + \p[w( v_1) \partial_z \nabla_{x,y}v_2],\\  
\label{eq:gradient_indenity_b0_2}
\partial_z \p [w(v_1)\partial_z v_2]
&=\partial_z [w(v_1)\partial_z v_2]\\
\nonumber
&= -(\nabla_{x,y}\cdot v_1 )\partial_z v_2+ w(v_1)\partial_z (\partial_z v)
\end{align}
as well as 
$
\big\| (\nabla_{x,y}\cdot v_1 )\partial_z v_2\big\|_{L^2}\leq \|\nabla_{x,y}\cdot v_1\|_{L^{3}}\|\partial_z v\|_{L^6}\lesssim 
\|v_1\|_{H^{3/2}}\|v\|_{H^2}.
$
\smallskip

\eqref{it:estimate_nonlinearity_H1}: We prove the estimate only in the case $\rho=1$. The non-integer case $\rho\in (1,2)$ then follows from bilinear complex interpolation \cite[Theorem 4.4.1]{BeLo} and 
\eqref{it:estimate_nonlinearity_High} for $\rho=2$, whose proof is given below.

Using the identities \eqref{eq:gradient_indenity_b0_1}-\eqref{eq:gradient_indenity_b0_2}, for all $v_1,v_2\in \Hs^3$,
\begin{align*}
\|b_0(v_1,v_2)\|_{H^1}
&\lesssim \|w(v_1)\partial_z v_2\|_{L^2} 
+\|\nabla_{x,y} (w(v_1)\partial_z v_2)\|_{L^2}
+ \|\partial_z (w(v_1)\partial_z v_2)\|_{L^2}\\
&\lesssim \|w(v_1)\partial_z v_2\|_{L^2} \\
&+\| w(\nabla_{x,y} v_1)\partial_z v_2\|_{L^2}
+\|w(v_1)\partial_z \nabla_{x,y} v_2\|_{L^2}\\
&+ \|(\nabla_{x,y}\cdot v_1)\,  \partial_z v_2\|_{L^2}
+ \|w(v_1)\partial_z^2 v_2\|_{L^2}\stackrel{{\rm def}}{=}\textstyle{\sum}_{0\leq j\leq 4} I_j.
\end{align*}

Note that $I_0$ has been estimated in \eqref{it:estimate_nonlinearity_L2}. 
Fix $\varepsilon>0$. To estimate the remaining terms, recall that Sobolev embeddings yield
$
H^{1+\varepsilon}(\T^2_{x,y})\embed L^{\infty}(\T^2_{x,y}).
$
In particular
\begin{align*}
I_1 
&\leq \| w(\nabla_{x,y}v_1)\|_{L^{\infty}(\T_z;L^{2}(\T_{x,y}^2))}
\| \partial_z v_2\|_{L^{2}(\T_z;L^{\infty}(\T_{x,y}^2))}\\
&\leq \| \nabla_{x,y}^2 v_1\|_{L^{2}}
\| \partial_z v_2\|_{L^{2}(\T_z;L^{\infty}(\T_{x,y}^2))}\\
&\lesssim \|  v_1\|_{H^2} \| \partial_z v_2\|_{L^2(\T_z;H^{1+\varepsilon}(\T_{x,y}^2))}\\
&\lesssim \|v_1\|_{H^{2+\varepsilon}}\|v_2\|_{H^{2+\varepsilon}}.
\end{align*}

A similar strategy can be applied to $I_2$:
\begin{align*}
I_2 
&\lesssim \|w(v_1)\|_{L^{\infty}}\|v_2\|_{H^2}\\
&\lesssim \|\nabla_{x,y}\cdot v_1\|_{L^{2}(\T_z;L^{\infty}(\T^2_{x,y}))}\|v_2\|_{H^2}\\
&\lesssim \|\nabla_{x,y}\cdot v_1\|_{L^{2}(\T_z;H^{1+\varepsilon}(\T^2_{x,y}))}\|v_2\|_{H^2}\\
&\lesssim \|v_1\|_{H^{2+\varepsilon}}\|v_2\|_{H^{2+\varepsilon}}.
\end{align*}
As $I_4$ can be estimates as $I_2$, it remains to estimate $I_3$:
$$
I_3\lesssim \|\nabla_{x,y}v_1\|_{L^4} \|\partial_z v_2\|_{L^4}
\lesssim  \|v_1\|_{H^{2+\varepsilon}}\|v_2\|_{H^{2+\varepsilon}},
$$
where we used that $H^1(\T^3)\embed L^p(\T^3)$ for all $p\in [2,6]$.

\eqref{it:estimate_nonlinearity_High}: As in the previous step, by bilinear complex interpolation \cite[Theorem 4.4.1]{BeLo},  it suffices to consider the case $\rho\in \N_{\geq 2}$. In the latter case, the proof is analogous to \eqref{it:estimate_nonlinearity_H1} where one uses once more the Leibniz rule and the embedding $
H^{2}(\T^2_{x,y})\embed L^{\infty}(\T^2_{x,y}).
$
\end{proof}

\subsubsection{Proof of Theorem \ref{t:high_order_estimates}}
\label{ss:high_order_estimates}
In the proof of Theorem \ref{t:high_order_estimates}, we need the following consequence of the Lenglart domination principle \cite{L77_domination}.
 
\begin{lemma}
\label{l:lenglart}
Let $\tau$ be a stopping time. Let $f,g:[0,\tau)\to \R$ be progressively measurable and increasing stochastic process such that $f_{\tau}$ and $g_{\tau}$ are finite a.s.\ and 
$$
\E f_{\mu}\leq\E g_{\mu}
$$
for all stopping times $\mu$ such that $g_{\mu}\leq R_0$ a.s.\ for some deterministic constant $R_0\geq 1$.
Then, for all $\g,R>0$,
$$
\P(f_{\tau}\geq \g)\leq R\g^{-1} + \P(g_{\tau}\geq R).
$$
In particular, if $\P(g_{\tau}\geq \g)\leq K_g (\log^n \g)^{-1}$ for $\g\gg 1$ and some constants $n\geq 1$ and $K_g$ (here $\log^n $ is the $n$-times composition of $\log$), then there exists a universal constant $C_0>0$  such that $\P(f_{\tau}\geq \g)\leq C_0 K_g (\log^n \g)^{-1}$ for $\g\gg 1$.
\end{lemma}

\begin{proof}[Proof of Theorem \ref{t:high_order_estimates}]
We split the proof into several steps. 
\smallskip

\emph{Step 1: Existence of a maximal unique $H^\sigma$-solution $((v_t)_t,\tau)$ such that}
\begin{equation}
\label{eq:blow_up_criterion}
\P\Big(\tau<T,\, \sup_{t\in [0,\tau)}\|v(t)\|_{H^\sigma}^2+\int_0^T \|v(t)\|_{H^{1+\sigma}}^2\,\dd t <\infty\Big)=0 \ \text{ for all }T<\infty.
\end{equation} 
The existence of a maximal unique $H^\sigma$-solution $((v_t)_t,\tau)$ satisfying \eqref{eq:blow_up_criterion} immediately follows from \cite[Theorem 3.3]{AV_variational} and Lemma \ref{l:estimate_nonlinearity}.

\smallskip

\emph{Step 2: Let $((v_t)_t,\tau)$ be as in Step 1. Then, for all $T>0$ and $\g\gg 1$,}
\begin{equation}
\label{eq:estimate_Hsigma_higher}
\P\Big(\sup_{t\in [0,T\wedge\tau)} \|v_t\|_{H^\sigma}^2 + \int_{0}^{T\wedge\tau} \|v_t\|_{H^{1+\sigma}}^2 \,\dd t\geq \g\Big)\lesssim_T \frac{1+\|v\|_{H^{\sigma}}^4}{\log\log\log(\g)} .
\end{equation} 

Here, we focus on the case $\sigma=2$ as it is more complicated. We give a few comments on the case $\sigma>2$ at the end of this step.
Thus, let us first consider $\sigma=2$. 
For all $j\geq 1$, we define the following stopping time 
$$
\tau_j \stackrel{{\rm def}}{=} \inf\big\{t\in [0,T]\,:\, \|v_t\|_{H^2}+ {\textstyle{\int}}_0^t\| v_r\|^2_{H^3}\,\dd s\geq j \big\}, 
 \  \text{ where } \ \inf\emptyset\stackrel{{\rm def}}{=}\tau\wedge T.
$$
Let $\eta,\xi$ be stopping times such that $0\leq \eta\leq \xi\leq \tau_j$ a.s.\ for some $j\geq 1$. By well-known $L^2$-estimates for the stochastic heat equation, there exists $C_0>0$ independent of $(j,\eta,\xi)$ such that (see e.g., \cite[Lemma 4.1]{AV_variational} or \cite[Chapter 4]{LR15})
\begin{align*}
\E \sup_{t\in[\eta,\xi]}\|v_t\|_{H^2}^2
&+ \E\int_{\eta}^{\xi}\|v_t\|_{H^3}^2\,\dd t \leq C_0\E\|v_{\eta}\|_{H^2}^2\\
&+C_0 \E\int_{\eta}^{\xi}( \|b(v_t,v_t)\|_{H^1}^2+ \|Q\|_{\HS(\Ls^2_0,\Hs_0^{2})}^2)\,\dd t .
\end{align*}
where $b$ is as \eqref{eq:def_nonlinearity_b}.
By Lemma \ref{l:estimate_nonlinearity}\eqref{it:estimate_nonlinearity_H10} and standard interpolation inequalities,
\begin{align*}
\|b(v_t,v_t)\|_{H^1}^2
&\lesssim \|v_t\|_{H^{5/2}}^2\|v_t\|_{H^{3/2}}^2\\
&\lesssim \|v_t\|_{H^{3}} \|v_t\|_{H^2}^2\|v\|_{H^{1}}
\leq \frac{1}{2C_0} \|v_t\|_{H^3}^2 + C_{1} \|v_t\|_{H^2}^4 \|v_t\|_{H^1}^2
\end{align*}
a.s.\ for all $t\geq 0$. Letting $M_t\stackrel{{\rm def}}{=}\one_{[0,\tau)\times \O} \|v_t\|_{H^2}^2 \|v_t\|_{H^1}^2$, by \cite[Remark 3.10]{primitive2}, 
$$
\P\big({\textstyle{\int}}_0^T M_t\,\dd t \leq \g\big)\lesssim (1+\|v\|_{H^1}^4) (\log(\log(\g)))^{-1}, \ \ \g\gg 1.
$$
Since $\|v\|_{L^2(\eta,\xi;H^2)}\leq j$, the above estimates imply
\begin{align*}
\E \sup_{t\in[\eta,\xi]}\|v_t\|_{H^2}^2
+ \E\int_{\eta}^{\xi}\|v_t\|_{H^3}^2\,\dd t
 \lesssim \E\|v_\eta\|_{H^2}^2+ \E\int_{\eta}^{\xi} (1+ M_t)(1+\|v_t\|_{H^1}^2)\,\dd t ,
\end{align*}
where the implicit constant depends only on $C_0$ and $\|Q\|_{\calL_2(\Ls^2,\Hs^2)}$.
Thus, the case $\sigma=2$ of \eqref{eq:estimate_Hsigma_higher} follows from the stochastic Grownall's inequality (see e.g., \cite[Lemma A.1]{AV_variational}) and the above estimate.

The case $\sigma>2$ follows from the case $\sigma=2$ and Lemmas \ref{l:estimate_nonlinearity}\eqref{it:estimate_nonlinearity_H1}-\eqref{it:estimate_nonlinearity_High}, \ref{l:lenglart}.

\smallskip

\emph{Step 3: Conclusion}. Note that by combining the blow-up criterion \eqref{eq:blow_up_criterion} and the energy estimate \eqref{eq:estimate_Hsigma_higher}, we obtain $\tau=\infty$ a.s. Thus, the maximal unique $H^\sigma$-solution in Step 1 is \emph{global} in time. Moreover, using $\tau=\infty$ a.s.\ in \eqref{eq:estimate_Hsigma_higher}, we obtain \eqref{it:high_order_estimates2}.
Finally, \eqref{it:high_order_estimates3} follows almost verbatim the proof of \cite[Theorem 3.7]{primitive2} by employing the just-proved energy estimate in \eqref{it:high_order_estimates2}.
\end{proof}

\subsection{Random dynamical system from PEs -- Proof of Proposition \ref{prop:RDS}}
\label{ss:RDS}
The aim of this subsection is to prove the following result. For the definition of RDS, the reader is referred to the text above Proposition \ref{prop:RDS}.

\begin{proposition}[RDS from stochastic PEs II]
\label{prop:RDSII}
Let $\sigma\geq 2$ and assume that $Q\in \calL_2(\Ls^2,\Hs^{\sigma'})$ for some $\sigma'>\sigma$.
Then the solution operator $v\mapsto v_t$ associated to the stochastic {\normalfont{PEs}} induces a {\normalfont{RDS}} over $\Hs^\sigma$. Moreover, denoting by $V$ the corresponding semiflow map, we have
$ V(t,\theta^s\om,\cdot)$ is independent of $\F_t$ for all $t\geq s\geq 0$.
\end{proposition}

The above implies Proposition \ref{prop:RDS} as if Assumption \ref{ass:Q} holds, then the stochastic PEs \eqref{eq:primitive_full} preserves the mean of the initial data and one can always choose $\sigma'>\sigma$ for which \eqref{eq:coloring_assumption_noise_2} also holds for $\sigma$ replaced by $\sigma'$.

To prove Proposition \ref{prop:RDSII}, we employ the following a-priori estimate for a modified version of the PEs.
In the following, we denote by $\exp^{[n]}$ the $n$-time composition of the function $\exp$. 
Recall that $b$ is defined \eqref{eq:def_nonlinearity_b}.

\begin{lemma}
\label{l:pathwise_energy_estimate}
Fix $\sigma\geq 2$, $v_0\in \Hs^\sigma$ and
$v'\in L^\infty_{\loc}([0,\infty);\Hs^{\sigma'})$ for some $\sigma'>\sigma$. Consider the following modified {\normalfont{PEs}}:
\begin{equation}
\label{eq:modified_PE}
\partial_t v + b(v+v',v+v')=\Delta v \ \text{ for } \ t>0, \qquad v(0)=v_0.
\end{equation}
Then, for all $T<\infty$ and $\eta\in [2,\sigma]$, the unique global $H^\sigma$-solution \eqref{eq:modified_PE} satisfies
\begin{equation}
\label{eq:sigma_estimate_pathwise}
\sup_{t\in [0,T]}\|v(t)\|_{H^\eta}^2+
\int_0^T \|v(t)\|_{H^{1+\eta}}^2\,\dd s \leq \exp^{[\g_0]} (R(1+\|v\|_{H^\eta}^{\g_1}+\|v'\|_{L^\infty(0,T;H^{\sigma'})}^{\g_1})).
\end{equation}
where $\g_0,\g_1\in \N_{\geq 1}$ and $R>0$ depends only on $\sigma$ and $(\sigma,T)$, respectively.
\end{lemma}

The proof of \eqref{eq:sigma_estimate_pathwise} follows the one of Theorem \ref{t:high_order_estimates}. For brevity, we only give a sketch. It is possible to improve the dependence on the initial data 
$v$ and 
$v'$, but the current form is sufficient for our purposes. Furthermore, by adapting the proof of the 
$H^1$-estimate from \cite{HH20_fluids_pressure}, one can verify that \eqref{eq:sigma_estimate_pathwise} holds with 
$\g_0=3$ and $\g_1=4$.

\begin{proof}[Proof of Lemma \ref{l:pathwise_energy_estimate} -- Sketch]
As $\sigma'>\sigma\geq 2$, by Sobolev embedding we have 
$$
v'\in L^\infty_{\loc}([0,\infty);W^{1,\infty}(\T^2_{x,y};L^2(0,1)))
\cap L^\infty_{\loc}([0,\infty);L^2(\T^2_{x,y};W^{1,\infty}(0,1))).
$$
Hence, the case where in \eqref{eq:sigma_estimate_pathwise} we have $\sigma=1$ follows from the (by now) well-established proof of global well-posedness in $H^1$ of the deterministic PEs, see e.g., \cite{CT07} or \cite[Subsection 1.4]{HH20_fluids_pressure}. To bootstrap from the latter situation, the estimate \eqref{eq:sigma_estimate_pathwise} for $\sigma\geq 2$, one can repeat the argument almost verbatim the argument used in the proof of Theorem \ref{t:high_order_estimates} by using the nonlinear estimate of Lemma \ref{l:estimate_nonlinearity}. 
\end{proof}

\begin{proof}[Proof of Proposition \ref{prop:RDSII}]
To prove the claim, it is convenient to work with the stochastic PEs \eqref{eq:primitive_full} in a pathwise sense. 
Let $\Gamma_t=
\int_0^t e^{(t-s)\Delta} Q\,\dd W_s $ and note that the process $(z_t)_t$ defined as $z_t=v_t-\Gamma_t$ solves \eqref{eq:modified_PE} with $v_t'=\Gamma_t$. Arguing as in \cite[Subsection 2.4.4]{KS12_2d}, by Lemma \ref{l:pathwise_energy_estimate}, the claim follows by noticing that $(\Gamma_t)_t\in L^\infty_{\loc}([0,\infty);\Hs^{\sigma'})$ with $\sigma'>\sigma$ a.s.
\end{proof}

\subsection{Regularity of invariant measures -- Proof of Proposition \ref{prop:regularity_invariant_measure}}
Lemma \ref{l:pathwise_energy_estimate} also yields the regularity of invariant measures claimed in Proposition \ref{prop:regularity_invariant_measure}.

\begin{proof}[Proof of Proposition \ref{prop:regularity_invariant_measure}]
Let $\mu$ be an invariant measure of \eqref{eq:primitive_full} on $\Hs^\rho$. Without loss of generality, we only consider the case $\rho=1$. 
By \cite[Theorem 1.7]{GHKVZ14}, it follows that $\mu(\Hs^2)=1$ and 
\begin{equation}
\label{eq:H2_regularity_invariant_measure}
\textstyle
\int_{\Hs^1_0(\T^3)} \log(R(1+\|v\|_{H^2}))\,\dd \mu(v)\lesssim_{\Energy_\sigma} 1,
\end{equation}
where $R>1$ is a universal constant.
To prove \eqref{eq:integrability_invariant_measure_Hrho}, we employ a standard bootstrapping argument by combining \eqref{eq:H2_regularity_invariant_measure}, and \eqref{eq:H2_regularity_invariant_measure} as a starting point. 

Recall that $\Energy_\sigma=\|Q\|_{\calL_2(\Ls^2_0,\Hs^\sigma_0)}$ and fix $\sigma'<1+\sigma$. We claim that, if there exists $\eta\in [ 2,\sigma]$ and $\g\in \N_{\geq 1}$ for which
\begin{align}
\label{eq:improved_regularity_invariant_measures_1}
\textstyle
\int_{\Hs^1_0(\T^3)} \log^{[\g] }(R_\g(1+\|v\|_{H^\eta}))\,\dd \mu(v)\lesssim_{\Energy_\sigma} 1, 
\end{align}
then there exists $\g'$ depending only on $\g,\eta,\sigma$ and $\sigma'$ such that 
\begin{equation}
\label{eq:improved_regularity_invariant_measures_2}
\textstyle\int_{\Hs^1_0(\T^3)} \log^{[\g'] }(R_\g'(1+\|v\|_{H^{\eta'}}))\,\dd \mu(v)\lesssim_{\Energy_\sigma} 1\ \ \ 
 \text{ where }\ \ \
  \eta' \stackrel{{\rm def}}{=}\eta+\sigma'-\sigma, 
\end{equation}
and $R'_\g$ depends only on $\g,\eta,\sigma$ and $\sigma'$.

To prove \eqref{eq:improved_regularity_invariant_measures_1}$\Rightarrow$\eqref{eq:improved_regularity_invariant_measures_2}, we employ weighted maximal $L^p$-regularity estimates, see e.g., \cite[Section 3.5.2]{pruss2016moving} or \cite[Subsection 17.2.e]{Analysis3}. 
Recall that, for 
$$
\eta_0\in (\eta',\eta+1),
$$ 
the operator on $-\Delta: H^{\eta_0}\subseteq H^{\eta_0-2}\to H^{\eta_0-2}$ has weighted maximal $L^p$-regularity estimates (this follows from the periodic version of \cite[Example 10.1.5]{Analysis2}). As we will see below, we do not take $\eta_0$ equal to $\eta+1$ due to the $\delta$-loss in Lemma \ref{l:estimate_nonlinearity}\eqref{it:estimate_nonlinearity_H1}.
 
Next, let us fix $p\in (2,\infty)$ such that $\eta_0-2/p>\eta'$, and afterwards $\kappa\in [0,p-1)$ such that $\eta_0-2\frac{1+\kappa}{p}<\eta$. Note that the latter choice is always possible by using $\eta'<\eta+\frac{1}{2}$ and choosing $\kappa$ close to $ p-1$.
Hence, by either \cite[Corollary 17.2.37]{Analysis3} or \cite[Theorem 3.4.8]{pruss2016moving}, the maximal $L^p$-regularity of $-\Delta$ on $H^{\eta_0-2}$ with time weight $t^\kappa$ applied to $v'_t =v_t-\Gamma_t$ (here $\Gamma_t=\int_0^t e^{\Delta(t-s)} Q \,\dd W_s$), a.s., it holds that 
\begin{align*}
\|v_1'\|_{H^{\eta'}}^p
&
\textstyle
\stackrel{(i)}{\lesssim} \|v\|_{H^{\eta}}^p+ \int_0^1 t^{\kappa} \|b(v_t-\Gamma_t,v_t-\Gamma_t)\|^p_{H^{\eta_0-2}}\,\dd s\\
&\textstyle
\stackrel{(ii)}{\lesssim} \|v\|_{H^{\eta}}^p+  \sup_{t\in [0,1]}\big(\|v_t\|_{H^{\eta}}^{2p}+\|\Gamma_t\|_{H^{\eta}}^{2p}\big),
\end{align*}
where in $(i)$ we used $H^{\eta}\embed B^{\eta_0-2(1+\kappa)/{p}}_{2,p}$ and $B^{\eta_0-2/p}_{2,p}\embed H^{\eta'}$ 
by the above choice of $\kappa$ and $p$, and in $(ii)$ Lemma \ref{l:estimate_nonlinearity}\eqref{it:estimate_nonlinearity_H1} and $\eta_0<\eta+1$.
As usual, $B$ stands for Besov spaces.
From the above and Lemma \ref{l:pathwise_energy_estimate}, we obtain, a.s., 
\begin{align*}
\|v_1'\|_{H^{\eta'}}^p
\leq \exp^{[\g_0]} (R(1+\|v\|_{H^\eta}^{\g_1}+\|(\Gamma_t)_t\|_{L^\infty(0,1;H^{\eta})}^{\g_1})),
\end{align*}
for some $\g_0,\g_1$ and $R$ depending only on $\sigma$ and $\sigma'$. 

Now, let $\g$ be as in \eqref{eq:improved_regularity_invariant_measures_1}. From the previously displayed formula, we have, a.s.,  
$$
\log^{[\g+\g_0]} (
\|v_1\|_{H^{\eta'}})\lesssim_{\sigma,\sigma'} 1+\|\Gamma_1\|_{H^{\sigma'}}+ \|(\Gamma_t)_t\|_{L^\infty(0,1;H^{\eta})}^{\g_1}
+ \log^{ [\g]} (\|v\|_{H^\eta}).
$$
By stochastic maximal $L^p$-regularity for the Laplacian on $H^{\sigma-1}(\T^3)$ (with no time weight) and $p(\sigma,\sigma')<\infty$ such that $\sigma'<\sigma+1-\frac{2}{p}$ (see e.g., \cite[Theorem 7.16]{AV19} and \cite[Example 10.1.5]{Analysis2}), it follows that $$
\E\|\Gamma_1\|_{H^{\sigma'}}^p\lesssim_{p,\Energy_\sigma} 1.
$$ 
Hence, the implication \eqref{eq:improved_regularity_invariant_measures_1}$\Rightarrow$\eqref{eq:improved_regularity_invariant_measures_2} now follows from taking the expected values, using the standard estimate $\E \|(\Gamma_t)_t\|_{L^\infty(0,T;H^{\eta})}^{\g_1}\lesssim_{\Energy_\sigma}1$ (here we used $\eta\leq \sigma$), and integrating over $\Hs^1_0$ with respect to the invariant measure $\mu$.
\end{proof}

\subsection{Unique ergodicity of stochastic PEs -- Proof of Theorem \ref{t:uniqueness_regularity_invariant}}
The aim of this subsection is to prove Theorem \ref{t:uniqueness_regularity_invariant}. 
As usual, it is based on smoothing and irreducibility properties of the Markov semigroup induced by the stochastic PEs \eqref{eq:primitive_full}, cf., \cite{DPZ_ergodicity,FlaMas1995}.
More precisely, we employ the following results.

\begin{proposition}[Strong Feller]
\label{prop:strong_feller_v}
Under the assumptions of Theorem \ref{t:uniqueness_regularity_invariant}, the Markov process $(v_t)_t$ induced by the stochastic {\normalfont{PEs}} \eqref{eq:primitive_full} is strong Feller in $\Hs^\sigma$.
\end{proposition}

\begin{proposition}[Irreducibility]
\label{prop:irreducibility_Appendix}
Let the assumptions of Theorem \ref{t:uniqueness_regularity_invariant} be satisfied. Then 
$
\supp\,\mu=\Hs^{\sigma}
$ for all invariant measures $\mu$ of the stochastic {\normalfont{PEs}} \eqref{eq:primitive_full} on $\Hs^{\sigma}$.
\end{proposition}

For the definition of the strong Feller property, the reader is referred to the text before Proposition \ref{prop:strong_feller_projective}.
The proof of Proposition \ref{prop:strong_feller_v} is postponed in Subsection \ref{ss:proof_irreducibility}. The irreducibility result of Proposition \ref{prop:irreducibility_Appendix} is given in Subsection \ref{ss:proof_irreducibility}.
Now, we first prove that the above implies Theorem \ref{t:uniqueness_regularity_invariant}.

\begin{proof}[Proof of Theorem \ref{t:uniqueness_regularity_invariant}]
The existence of an invariant measure on $\Hs^2(\T^3)$ was established in \cite[Theorems 16 and 1.7]{GHKVZ14}. By Proposition \ref{prop:regularity_invariant_measure}, these measures are also concentrated on $\Hs^\sigma$. The uniqueness and support conditions for invariant measures in Theorem \ref{t:uniqueness_regularity_invariant} then directly follow from Propositions \ref{prop:strong_feller_v}-\ref{prop:irreducibility_Appendix} and Khasminskii's theorem (see, e.g., \cite[Proposition 4.1.1]{DPZ_ergodicity}).
\end{proof}

The results in Propositions \ref{prop:strong_feller_v} and \ref{prop:irreducibility_Appendix} lead to a stronger result. For later use, we formulate it in the following

\begin{remark}[Strongly mixing of transition semigroup for PEs]
\label{r:strongly_mixing_etc}
Doob's theorem \cite[Theorem 4.2.1]{DPZ_ergodicity} and Propositions \ref{prop:strong_feller_v}-\ref{prop:irreducibility_Appendix} imply that the Markov transition function $P_t(v,A)=\P(v_t\in A\,|\,v_0=v)$ where $A\in \Borel(\Hs^\sigma)$ is strongly mixing and the measures $P_t(v,\cdot)$, $P_t(v',\cdot)$ are equivalent for all $v,v'\in \Hs^\sigma$. 
\end{remark}

\subsubsection{Proof of Proposition \ref{prop:irreducibility_Appendix}}
\label{ss:proof_irreducibility}
In this subsection, we prove the irreducibility result of Proposition \ref{prop:irreducibility_Appendix}. 
We begin with a controllability result.

\begin{lemma}[Controllability]
\label{l:controllability_primitive}
Let the assumptions of Theorem \ref{t:uniqueness_regularity_invariant} be satisfied.
Fix $T\in (0,\infty)$, an initial data $\vin\in \Hs_0^\sigma$ and $\vf_0\in \cap_{\sigma'\geq \sigma}\,\Hs^{\sigma'}$. Then there exists a smooth control $g\in \cap_{\sigma'\geq \sigma}L^\infty(0,T;\Hs_0^{\sigma'})$ such that the unique $\Hs^\sigma$-solution to (deterministic) controlled {\normalfont{PEs}} (here $b$ is as in \eqref{eq:def_nonlinearity_b})
\begin{equation}
\label{eq:control_primitive_lemma_statement}
\partial_t v_t+ b(v_t,v_t)
= \Delta v_t+Q g_t \ \ \text{ for }t>0, \qquad v_0=\vin,
\end{equation}
satisfies $v_T=\vf$.
\end{lemma}

\begin{proof}
Let $\wt{v}$ be the unique $H^\sigma$-solution to
$$
\partial_t \wt{v}_t - \Delta \wt{v}_t+  \p[(\wt{v}_t\cdot\nabla_{x,y})\wt{v}_t + w(\wt{v}_t)\partial_z \wt{v}_t]
=0,  \qquad \wt{v}_0=z_0.
$$
By \cite{GGHHK20_analytic}, we have $z_t\in \cap_{\sigma'>\sigma} \Hs^{\sigma'}_0$ for all $t>0$. In particular 
$\wt{z}_0\stackrel{{\rm def}}{=} \wt{v}_{1/2}\in \cap_{\sigma'>\sigma} \Hs^{\sigma'}_0$. Let
$$
f_t= (\varphi_t \wt{z}_0+ (1-\varphi_t)z_1)
$$
where $\varphi$ is a smooth cut-off function such that $\varphi(1/2)=1$ and $\varphi(1)=0$. From the above construction and uniqueness of solutions to the controlled PEs, one can readily check that the following control 
$$
g_t= \one_{[1/2,1]} Q^{-1}(\partial_t f_t - \Delta f_t - \p[(f\cdot\nabla_{x,y})f_t + w(f_t)\partial_z f_t]
 ),
$$
where $Q^{-1}$ is well-defined due to
satisfies the requirement of Step 1 and $
v^{(g)}_t=\wt{v}_t  $ if $ t\in [0,1/2]$ or $
v^{(g)}_t=f_t$ if $t\in (1/2,1]$.
\end{proof}

We conclude this subsection by studying the regularity and support of the law of stochastic convolutions:
\begin{equation}
\label{eq:def_stochastic_convolution_primitive}
\textstyle
\Gamma_t\stackrel{{\rm def}}{=} \int_0^t e^{(t-s)\Delta} Q\,\dd W_s.
\end{equation}

\begin{lemma}[Regularity and support of stochastic convolutions]
\label{l:support_stoch_convolution_primitive}
Let the assumptions of Theorem \ref{t:uniqueness_regularity_invariant} be satisfied.
Set
$
\XX_T\stackrel{{\rm def}}{=}L^2(0,T;\Hs^{1+\sigma}_0)\cap C([0,T];\Hs^{\sigma}_0)
$
where $T>0$.
Then the following hold:
\begin{itemize}
\item $(\Gamma_t)_t \in \XX_T$ a.s.;
\item $\P((\Gamma_t)_t \in \Op)>0$ for all open sets $\Op\subseteq \XX_T$.
\end{itemize}
\end{lemma}

Although the above might be standard, we included some details for the reader's convenience.

\begin{proof}
For all $N\geq 1$, let $Q_N\stackrel{{\rm def}}{=}\Pi_{\leq N}Q\Pi_{\leq N}$ where $\Pi_{\leq N}$ is the projection onto the Fourier modes $\leq N$, and $\Gamma_t^N\stackrel{{\rm def}}{=} \int_0^t e^{(t-s)\Delta} Q_N\,\dd W_s$. 
It is clear that $(\Gamma_t^N)_t:\O\to \XX_T$ is a Gaussian random variable for all $N\geq 1$, and $(\Gamma_t^N)_t\to (\Gamma_t)_t$  in $L^2(\O;\XX_T)$ for all $T<\infty$ as $N\to \infty$. 
In particular, $(\Gamma_t)_t:\O\to \XX_T$ is Gaussian as well. 
The final claim follows from the fact that Gaussian random variables have full support.
\end{proof}

\begin{proof}[Proof of Proposition \ref{prop:irreducibility_Appendix}]
The strategy is well-known to experts, cf., the one of \cite[Lemma 7.3]{BBPS2022}. However, some modifications are needed because of the mapping property of the nonlinearity in the PEs, cf., Lemma \ref{l:estimate_nonlinearity}.
For the reader's convenience, we include the details.
It suffices to show that for all $z_0\in \Hs^\sigma_0$, $z_1\in \cap_{\sigma'>\sigma}\Hs^{\sigma'}_0$ and $\g>0$,
\begin{equation}
\label{eq:approx_controllability_time_1}
\P(v_1\in B_{\g}(z_1) \,|\, v_0=z_0)>0.
\end{equation}
Indeed, if \eqref{eq:approx_controllability_time_1} holds, then the invariance of $\mu$ yields
$$
\mu(B_\g(z_1))=
\int_{\Hs^\sigma} 
\P(v_1\in B_{\g}(z_1) \,|\, v_0=z_0)\,\dd \mu(z_0)>0
$$
for all $z_1\in\cap_{\sigma'>\sigma} \Hs^{\sigma'}_0$ and $\g>0$. Hence, 
$\supp\,\mu \supseteq\cap_{\sigma'>\sigma} \Hs^{\sigma'}_0$, and by density, 
$$
\supp\,\mu \supseteq \overline{\cap_{\sigma'>\sigma}\Hs^{\sigma'}_0}=\Hs^{\sigma}_0,
$$
where the closure is taken in the $\Hs^{\sigma}_0$-topology.

\smallskip

Now, fix $z_0\in \Hs^\sigma$ and $z_1\in \cap_{\sigma'>\sigma} \Hs^{\sigma'}$. Let $(g_t)_t$ be the control given in Lemma \ref{l:controllability_primitive} with $T=1$, $\eta=z_0$ and $\xi=z_1$. 
Let  $(v_t^{(g)})_t$ the solution to the corresponding control problem \eqref{eq:control_primitive_lemma_statement}.
By Lemma \ref{l:support_stoch_convolution_primitive}, it follows that 
$$
\P ( \O_\varepsilon)>0 \quad \text{ where }\quad
\O_\varepsilon\stackrel{{\rm def}}{=}
\Big\| \Big(\Gamma_t - \int_0^t e^{(t-s)\Delta} Qg_s\,\dd s\Big)_t\Big\|_{\XX_T}\leq \varepsilon,
$$ 
where $\Gamma_t$ and $\XX_T$ are as in \eqref{eq:def_stochastic_convolution_primitive} and Lemma \ref{l:support_stoch_convolution_primitive}, respectively.

By the definition of $\O_\varepsilon$ and $\varepsilon<1$, we have $\sup_{\O_\varepsilon}\|(\Gamma_t)_t\|_{\XX_T}\leq R$ where $R$ is independent of $\varepsilon$. Hence, from Lemma \ref{l:pathwise_energy_estimate} (and its extension to the case of a control $g$), it follows the existence of a constant $R_0>0$ independent of $\varepsilon$ such that 
\begin{equation}
\label{eq:apriori_estimate_appendix_R0}
\|(v_t)_t\|_{\XX_T}+ \|(v_t^{(g)})_t\|_{\XX_T}\leq R_0.
\end{equation}

It is clear that that difference $v_t-v_t^{(g)}$ satisfies, a.s.,
\begin{align*}
v_t-v^{(g)}_t = \int_0^t e^{(t-s)\Delta}\big[ b(v_t,v_t)-b(v_t^{(g)},v_t^{(g)})\big] \,\dd s + \Gamma_t - \int_0^t e^{(t-s)\Delta} Q g_s\,\dd W_s
\end{align*}
where $b$ denotes the bilinear nonlinearity in the PEs \eqref{eq:def_nonlinearity_b}.
Thus, by parabolic regularity and Lemma \ref{l:estimate_nonlinearity}, we have, for all $t\in (0,1)$, $\delta\in (0,1)$ and a.e.\ on $\O_{\varepsilon}$, 
\begin{align*}
\sup_{s\in (0,t)}\|v_s-v^{(g)}_s\|_{H^\sigma}^2
&+
\int_0^T\|v_s-v^{(g)}_s\|_{H^{\sigma+1}}^2\,\dd t \\
&\leq C_0 \int_0^t \| b(v_s,v_s)-b(v_s^{(g)},v_s^{(g)})\|_{H^{\sigma-1}}^2 \,\dd s +C_0 \varepsilon\\
&\leq C_{0}C_{\delta} \int_0^t 2R\|v_s-v_s^{(g)}\|_{H^{2+\delta}}^2\,\dd s+C_0\varepsilon\\
&\leq\frac{1}{2} \int_0^t\|v_s-v^{(g)}_s\|_{H^{\sigma+1}}^2\,\dd s + C_{1,\delta}  \int_0^t \|v_s-v_s^{(g)}\|_{H^{\sigma}}^2\,\dd s+C_0\varepsilon.
\end{align*}
Now, by the Gr\"onwall inequality,
$$
\sup_{t\in (0,1)}\|v_t-v_t^{(g)}\|_{\Hs^\sigma}\leq C\varepsilon  \quad \text{ on }\ \O_\varepsilon.
$$
Thus, \eqref{eq:approx_controllability_time_1} follows from the above with $\varepsilon=C/\g$ and $\P(\O_\varepsilon)>0$.
\end{proof}

\section{Strong Feller for the Lagrangian and projective processes}
\label{s:strong_Feller}
The purpose of this section is to establish Propositions \ref{prop:strong_feller_projective} and \ref{prop:strong_feller_v}, which guarantee the strong Feller property for the projective and velocity processes, respectively. For the projective processes, we focus primarily on the process $(v_t, \x_t, \xi_t)_t$, as the argument for the process $(v_t, \x_t, \wc{\xi}_t)_t$ follows similarly. To show the strong Feller property, we adopt a standard approach that reduces the problem of proving a gradient estimate for a truncated SPDE. Central to this reduction is the energy estimate provided in Theorem \ref{t:high_order_estimates}. The gradient estimate itself is established using Malliavin calculus.

This section is organised as follows. In Subsection \ref{ss:strong_feller_via_cut_off}, we reduce the proof of the strong Feller property for the original system to a truncated one. The latter is then further reduced to a suitable gradient estimate for the truncated problem. Subsection \ref{ss:Malliavin_calculus} introduces the necessary notation for Malliavin calculus. Subsection \ref{ss:unique_ergodicity_PEs} provides an elementary proof of Proposition \ref{prop:strong_feller_v}, which also sets the stage for the more intricate proof of Proposition \ref{prop:strong_feller_projective}. Finally, in Subsections \ref{ss:gradient_estimates} and \ref{ss:nondegeneracy}, we prove the required gradient estimate for the projective process $(v_t, \x_t, \xi_t)_t$ and the estimate for the inverse of the reduced Malliavin matrix, respectively.

\subsection{Strong Feller of the projective process via its cutoff}
\label{ss:strong_feller_via_cut_off}
As in \cite[Subsection 6.1]{BBPS2022}, we begin by considering the augmented system:
\begin{equation}
\label{eq:augumented_system}
\left\{
\begin{aligned}
\partial_t v_t &= \Delta v_t-\p[(v_t\cdot\nabla_{x,y})v_t+w(v_t)\partial_z v_t]+ Q \dot{W}^{(v)}_t,\\
\partial_t \x_t&= u_t(\x_t),\\
\partial_t \xi_t &= \Pi_{\xi_t} [\nabla u_t (\x_t)\xi_t],\\
\partial_t \eta_t &=\dot{W}_t^{(\eta)},
\end{aligned}
\right.
\end{equation}
where $W^{(v)}_t=W_t$ is the cylindrical Brownian motion on $\Ls^2_0$, described in Subsection \ref{sss:probabilistic_set_up}, $W_t^{(\eta)}$ is a $\R^6$-Brownian motion in $\R^6$ independent of $W^{(v)}_t$ and $Q$ is as Subsection \ref{sss:probabilistic_set_up}. Moreover, as above, $u=(v,w)$ where $w$ is as in \eqref{eq:def_w}.
The system \eqref{eq:augumented_system} is complemented with the initial conditions
$v_0=v\in \H$,  $\x_0=\x\in \Dom$, $ \xi_0 =\xi\in \S^2$ and $ \eta_0=\eta\in \R^6$.

In contrast to Subsection \ref{ss:proofs_further_results}, for convenience, here we treat $(\xi_t)_t$ as an $\S^2$-valued process. However, to establish the strong Feller property as in Proposition \ref{prop:strong_feller_projective}, this approach is sufficient. Indeed, any observable $\psi:\Hs^\sigma_0\times \Dom\times \Pr\to \R$ can be lifted to an obsevable $\wt{\psi}$ on $\Hs_0^\sigma\times \Dom\times \S^2$ via the formula $\wt{\psi}(v,\x,\xi)=\psi(v,\x,\pi(\xi))$, where $\pi:\S^2\to \Pr$ is the standard quotient map.

\smallskip

To proceed further, we introduce some more notation. We denote by $\v_t$ the augmented process:
$$
\v_t\stackrel{{\rm def}}{=}
\begin{pmatrix}
v_t,
\x_t,
\xi_t,
\eta_t
\end{pmatrix}^\top
\in \HL
\qquad \text{ where }\qquad 
\HL\stackrel{{\rm def}}{=}\H\times \Dom\times \S^2\times \R^6.
$$
Note that $\v_t$ satisfies the abstract SPDE on $\HL$:
\begin{equation}
\label{eq:abstract_SPDE}
\left\{
\begin{aligned}
&\partial_t \v_t 
+A \v_t= F(\v_t)+ \wh{Q} \dot{W}_{t},\\
&\v_0=\v,
\end{aligned}
\right.
\end{equation}
where $\v\in \HL$ and 
\begin{equation*}
A \v =
\begin{pmatrix}
-\Delta v\\
0\\
0\\
0
\end{pmatrix}, 
\quad
F(\v) = 
\begin{pmatrix}
- \p[(v\cdot\nabla_{x,y})v+w(v)\partial_z v]\\
u(\x)\\
\Pi_\xi [\nabla u(\x)\xi]\\
0
\end{pmatrix}, \quad
\wh{Q}\dot{W}_t=
\begin{pmatrix}
Q \dot{W}_t^{(u)}\\
0\\
0\\
\dot{W}_t^{(\eta)}
\end{pmatrix}.
\end{equation*}

Let $\chi$ be a smooth and nonnegative increasing function satisfying 
\begin{equation}
\label{eq:choice_cutoff}
\chi(t)=0 \ \text{ for }t\leq 1, \qquad \text{ and }\qquad \chi(t)=1\ \text{ for }t\geq 2.
\end{equation}
For $r\geq 0$, we let $\chi_r\stackrel{{\rm def}}{=}\chi(\cdot/r)$. Consider the following truncation of the nonlinearity:
\begin{equation}
\label{eq:Ftruncation_nonlinearity}
F^{(r)}(\v)\stackrel{{\rm def}}{=} \big[1-\chi_{3r}(\|v\|_{H^\sigma})\big] F(\v)+ \chi_{r}(\|v\|_{H^\sigma}) H(\x,\xi,\eta), 
\end{equation}
where, for $\v=(v,\x,\xi,\eta)\in \HL$, 
\begin{equation*}
H(\x,\xi,\eta)= 
(\dist(\x))^4\,
\begin{pmatrix}
0\\\sum_{j=1}^3 a_j \frac{
\eta_j}{\sqrt{1+|\eta_j|^2}}\\
\Pi_{\xi}\sum_{j=1}^3 a_{j} \frac{\eta_{j+3}}{\sqrt{1+|\eta_{j+3}|^2}}\\
0
\end{pmatrix},
\end{equation*}
$(a_j)_{j=1}^3$ denotes the canonical basis of $\R^3$ and 
$\dist(\x)$ the distance from $\partial\Dom$. The exponent used to quantify the distance to the boundary may not be optimal; our choice is guided by the exponent that appears in the main algebraic lemmas below—see Lemma \ref{l:spanning_condition}. 

The above truncation (slightly) differs from the one used in \cite[Section 6]{BBPS2022}. Indeed, the additional term $\dist(\x)$ ensures that the $\x_t$-dynamics remains within $\Dom$. As in \cite{BBPS2022}, the process $\eta$ is used in combination with the truncation to transfer some `randomness' from the $\eta$-equation to $\x_t$ and $\xi_t$ via the nonlinearity $\chi_{r}(\|v\|_{H^\sigma})H(\x,\xi,\eta)$ when $\|v\|_{\Hs^\sigma}$ is large. However, for small values of $\|v\|_{\Hs^\sigma}$, the coupling is not efficient enough, and a hypoellipticity argument has to be exploited; see Subsection \ref{sss:hypoelliptic}. 

Consider the solution to the following version of \eqref{eq:abstract_SPDE} with cutoff:
\begin{equation}
\label{eq:abstract_SPDE_r}
\left\{
\begin{aligned}
&\partial_t \v_t^{(r)}+ A \v_t^{(r)} = F^{(r)}(\v_t^{(r)})+ \wh{Q} \dot{W}_t,\\
&\v_0^{(r)}=\v,
\end{aligned}
\right.
\end{equation}
where $\v\in \HL$.
The global well-posedness of \eqref{eq:abstract_SPDE_r} follows from standard arguments due to the presence of the truncation, Lemma \ref{l:estimate_nonlinearity} and the smoothness condition $\sigma\geq 2$.
For all $r,t\geq 0$ and bounded Borel maps $\phi:\HL\to \R$, we let
$$
[ P_t \phi](\v)\stackrel{{\rm def}}{=}\E[\phi(\v_t)] \qquad \text{ and }\qquad
[ P_t^{(r)} \phi](\v)\stackrel{{\rm def}}{=}\E[\phi(\v_t^{(r)})] ,
$$
where $\v_t$ and $\v_t^{(r)}$ is the solution to \eqref{eq:abstract_SPDE} and \eqref{eq:abstract_SPDE_r} with initial data $V$, respectively. As in Subsection \ref{ss:proofs_further_results}, the process $(V_t)_t$ (resp.\ $(V_t^{(r)})_t$) is strong Feller in $\HL$ if $P_t \phi\in C(\HL;\R)$ (resp.\ $P_t^{(r)} \phi\in C(\HL;\R)$) for all bounded Borel maps $\phi:\HL\to \R$.

The following result enables us to deduce the strong Feller property of the process $(\v_t)_{t\geq 0}$  from that of the truncated process $(\v^{(r)}_t)_{t\geq 0}$ provided $r$ is sufficiently large.

\begin{lemma}[Strong Feller of the projective process -- Reduction to cutoff]
\label{l:strong_feller_reduction}
Suppose that there exists $r_0>0$ such that $(\v^{(r)}_t)_{t\geq 0}$ is strong Feller in $\HL $ for all $r\geq r_0$. Then $(\v_t)_{t\geq 0}$ is strong Feller in $\HL$.
\end{lemma}

\begin{proof}
Fix a bounded Borel observable $\phi:\HL\to \R$, a time $t\geq 0$ and an initial data $V\in \HL$. To prove the claim, we need to show that, for each $\varepsilon>0$, there exists $\delta=\delta(\varepsilon)>0$ such that
\begin{equation}
\label{eq:continuity_Pt_reduction}
|P_t\phi(V')-P_t\phi(V)|\leq \varepsilon\ \ \text{
for all } \ \|V-V'\|_{\H}\leq \delta.
\end{equation}
Let $\v=(v,\x,\xi,\eta)\in \HL$. By \eqref{eq:augumented_system}, it is clear that $\v_t=(v_t,\x_t,\xi_t,\eta+ W^{(\eta)}_t)$ for $t\geq 0$, where $v_t$ is the unique solution to \eqref{eq:primitive_full} and $(\x_t,\xi_t)$ the corresponding processes uniquely determined by $v_t$ as in \eqref{eq:augumented_system}. 
By Theorem \ref{t:high_order_estimates}\eqref{it:high_order_estimates2}, we can choose $r_1\geq r_0$ (depending on $\|v\|_{\Hs^\sigma}$) such that, for $\wh{v}\in \H$ satisfying $\|v-\wh{v}\|_{\H}\leq 1$, the solution $\wh{v}_t$ of the stochastic PEs \eqref{eq:primitive_full} with initial data $\wh{v}\in \H$ satisfies
$$
\P\Big(\sup_{s\in [0,t]}\|\wh{v}_s\|_{\Hs^\sigma}\geq  r_1\Big)\leq \frac{\varepsilon}{8\|\phi\|_{L^\infty}}.
$$
Hence, defining the following stopping time 
$$
\wh{\tau}\stackrel{{\rm def}}{=}\inf\{s\in [0,t]\,:\, \|\wh{v}_s\|_{\H}\geq r_1\} , \quad \text{ and }\quad \inf\emptyset\stackrel{{\rm def}}{=}t,
$$
we have $\P(\wh{\tau}<t)\leq  \varepsilon/(8\|\phi\|_{L^\infty})$.
By uniqueness of solutions to the stochastic PEs (cf., Definition \ref{def:solution}), it follows that $v^{(r)}=v$ on $[0,\tau)\times \O$. 
Hence, for all $r\geq r_1$ and $\wh{\v}=(\wh{v},\x,\xi,\eta)\in \HL$ such that $\|v-\wh{v}\|_{\Hs^\sigma}\leq 1$,
\begin{equation}
\label{eq:estimate_Pt_Ptr}
| P_t \phi(\wh{V})- P_t^{(r)} \phi(\wh{V})|\leq \E\big[\one_{\{\wh{\tau}<t\}}(|\phi (\wh{V}_t)| +|\phi(\wh{V}^{(r)}_t)|) \big]
\leq \varepsilon /4 .
\end{equation}
Thus, for all $r\geq  r_1$ and $\v'\in \HL$ satisfying $\|\v-\v'\|_{\HL}\leq 1$,
\begin{align*}
| P_t \phi(\v)- P_t \phi(\v')| 
&\leq 
| P_t \phi(\v)- P_t^{(r)} \phi(\v)|\\
&+| P_t^{(r)} \phi(\v)- P_t^{(r)} \phi(\v')|  +
| P_t \phi(\v')- P_t^{(r)} \phi(\v')| \\
&\stackrel{\eqref{eq:estimate_Pt_Ptr}}{\leq} \varepsilon/2+| P_t^{(r)} \phi(\v)- P_t^{(r)} \phi(\v')| .
\end{align*}
The conclusion follows by the continuity of $P_t^{(r)} \phi(\cdot)$ for $t>0$ provided $r\geq r_1$.
\end{proof}

In light of Lemma \ref{l:strong_feller_reduction}, it is enough to prove the Markov process $(V^{(r)}_t)_t$ is strong Feller in $\HL$ for $r$ sufficiently large.
The latter fact is a consequence of the following result. 
In its formulation, we use the following notation:
$$
T_\v \HL\stackrel{{\rm def}}{=}
\H(\T^3)\times \R^3 \times T_{\xi} \S^2\times \R^6 \ \ \text{ for }\ \ \v=(v,\x,\xi,\eta)\in \HL.
$$

\begin{proposition}[Gradient estimate for the cutoff projective process]
\label{prop:gradient_estimate}
There exist positive constants $a_0,b_0,r_0,T_0>0$ depending only on the parameters in Subsection \ref{sss:probabilistic_set_up} such that for all $r\geq r_0$, $t\leq T_0$ and $\phi\in C^2(\HL;\R)$ the mapping 
$$
\HL\ni \v\mapsto P_t^{(r)}\phi(\v)
$$ 
is differentiable and 
$$
| D P^{(r)}_t \phi (\v) h |\lesssim t^{-a_0}(\dist (\x))^{-b_0}(1+|\eta|^{b_0})\|\phi\|_{L^{\infty}(\HL)}\|h\|_{T_V\HL}  
$$
for all $\v\in \HL$ and $h\in T_V \HL$, with implicit constant independent of $h$ and $V$.
\end{proposition}

As in \cite[Subsection 6.1]{BBPS2022_AOP}, the above readily implies the strong Feller property of $(V^{(r)}_t)_t$ with $r\geq r_0$.
Indeed, fix $\phi\in C^2(\HL;\R)$ and let $\v=(v,\x,\xi,\eta),\v'=(v',\x',\xi',\eta')\in \HL$ be such that $\|V-V'\|_{\HL}\leq \varepsilon$ with $\varepsilon\in (0,\dist(\x)/2)$. In particular, $\dist(\x')\leq \dist(\x)/2$. 
Let $[0,1]\ni \lambda \to \xi^\lambda$ be a geodesic on $\S^2$ such that $\xi^0=\xi$ and $\xi^1=\xi'$, and $|\frac{\dd}{\dd t}\xi^\lambda|\lesssim |\xi-\xi'|$. Letting 
$V^\lambda= ((1-\lambda)v+\lambda v', (1-\lambda)\x+\lambda \x', \xi^\lambda,(1-\lambda)\eta+\lambda \eta)$, from Proposition \ref{prop:gradient_estimate}, for all $t\leq T_0$, we have
\begin{align*}
|P^{(r)}_t\phi(\v)-P^{(r)}_t\phi(\v')|
&\textstyle\leq \int_0^1 \big| D P^{(r)}_t \phi (\v_\lambda) \frac{\dd}{\dd \lambda  }V^\lambda\big| \,\dd \lambda\\
&
\textstyle\leq t^{-a_0}(\dist (\x))^{-b_0}(1+|\eta|^{b_0})\|\phi\|_{L^{\infty}(\HL)} \|V-V'\|_{\HL}.
\end{align*}
Since the right-hand side of the above depends only on $\|\phi\|_{L^\infty}$, the previous extends from $\phi\in C^2(\HL;\R)$ to $\phi$ being only bounded and measurable (see \cite[Lemma 7.1.5]{DPZ_ergodicity} for details). The case $t>T_0$ follows from the semigroup property and the Feller property (see Proposition \ref{prop:global_well_posedness}).  
Hence, $(V^{(r)}_t)_t$ is strong Feller for all $r\geq r_0$.

In the rest of this section, we focus on proving Proposition \ref{prop:gradient_estimate}. 
Our proof of Proposition \ref{prop:gradient_estimate} follows the line of \cite[Proposition 6.1]{BBPS2022} done in Section 6 of the same work. However, adjustments are needed for the stochastic PEs \eqref{eq:primitive_full}.
In particular, because of the noise degeneracy of the third component $w_t$ of the vector field $u_t$ (see \eqref{eq:def_w}), which creates additional difficulties in the study of the invertibility of the partial Malliavin matrix, see Subsection \ref{sss:hypoelliptic}.

\smallskip

\emph{For notational convenience, in the rest of this section, we drop superscripts in the cutoff \emph{SPDE} \eqref{eq:abstract_SPDE_r} and related objects, e.g., in the following, $P_t$ stands for $P_t^{(r)}$. In particular, all the processes appearing below are solutions to \eqref{eq:abstract_SPDE_r} for some $r>0$, which will be clear from the context.}

\subsection{Setup for the gradient estimate and Malliavin calculus}
\label{ss:Malliavin_calculus}

\subsubsection{Basics on Malliavin calculus}
In this subsection, we introduce the key concepts in the context of Malliavin calculus, primarily to fix the notation for the subsequent subsections. For a more detailed discussion, the reader is referred to the monograph \cite{Nualart}, and for extensions to the Banach-valued case, see \cite{PV14_Malliavin}. Malliavin calculus is typically introduced in the context of isonormal Gaussian processes \cite[Definition 1.1.1]{Nualart}. For our purposes, we focus here on $\uu$-cylindrical Brownian motion $\ww$ on a separable Hilbert space; see e.g., \cite[Definition 2.11]{AV19_QSEE_1} for the definition. Recall that such a cylindrical process is equivalent to an isonormal Gaussian process on $L^2(0,T;\uu)$. To see the correspondence between the latter situation and the one analysed here, let us mention that the sequence of standard independent Brownian motions $(W^{(\kk,\ell)})_{(\kk,\ell)\in \z}$ in \eqref{eq:L20cylindrical_noise} uniquely induced a $\Ls^2_0$-cylindrical Brownian motion via the formula: For $ f\in \Ls^2_0$ and $ 0\leq s<t\leq T$,
\begin{equation}
\label{eq:noise_isonormal}
\textstyle
\ww_{\Ls^2_0} (\one_{[s,t]} \otimes f)\stackrel{{\rm def}}{=}\sum_{(\kk,\ell)\in\z} \g_{(\kk,\ell)}e_{\kk} (f,  \g_{(\kk,\ell)}e_{\kk})_{L^2}\,(W^{(\kk,\ell)}_t-W^{(\kk,\ell)}_s).
\end{equation} 
Indeed, it is routine to show that $\ww_{\Ls^2_0}$, initially defined on simple functions by linearly extending the above formula, can be extended uniquely as a bounded linear operator $\ww_{\Ls^2_0} :L^2(0,T;\Ls^2_0)\to L^2(\O)$ satisfying the properties in \cite[Definition 2.11]{AV19_QSEE_1}. 

In the context of the study of the augmented process \eqref{eq:abstract_SPDE_r}, the space $\uu=\Ls^2_0\times \R^6$, where one also takes into account the noise in the $\eta$-component. Moreover, for the sake of proving Proposition \ref{prop:strong_feller_v}, we use $\uu=\Ls^2$. For the latter situation, the association between noise and isonormal process as in \eqref{eq:noise_isonormal} extends naturally. 

\smallskip

Using the previous identification, the Malliavin derivative $\D$ can be defined as in \cite[Subsection 1.2]{Nualart} or \cite[Subsection 2.2]{PV14_Malliavin} on an Hilbert space $\hh$ as an unbounded linear operator
\begin{equation}
\label{eq:malliavin_derivative}
\D: \Do(\D)\subseteq L^2(\O;\hh)\to L^2(\O\times (0,T);\calL_2(\uu,\hh)),
\end{equation}
where $\Do(\D)$ denotes the domain of the operator.
In the above we used the identification $L^2(0,T;\calL_2(\uu,\hh))=\calL_2(L^2(0,T;\uu),\hh)$ (see e.g., \cite[Theorem 9.4.8]{Analysis2}) and the vector valued extension of the Fubini identity $L^2(\O;L^2(0,T))=L^2(\O\times (0,T))$. Here, with a slight abuse of notation, we did not display the dependence on the ground space $\hh$ and on $T>0$. 

For a random variable $\Phi\in \Do(\D)$, we use the more suggestive notation $\D\Phi=(\D_s \Phi)_{s}$ to denote its Malliavin derivative, where $\D_s\Phi\in L^2(\O;\calL_2(\Ls_0^2,\Hs^\sigma))$ for a.a.\ $s\in (0,T)$. 

In general, computing the Malliavin derivative is not straightforward. Below, we will need to compute the Malliavin derivative of $V_T$, where $V_T$ is the solution to the truncated SPDE \eqref{eq:abstract_SPDE} for a fixed $r > 0$. However, as shown in \cite[Subsection 1.2.1]{Nualart}, when $\O$ is the standard space of paths, the Malliavin derivative of the random variable $V_T$ be computed using the Fr\'echet derivative of the \emph{It\^o-map} (if available) associated to \eqref{eq:abstract_SPDE}, i.e., the map which associated to the noise realisation $W$ the solution $V_T(W)$, see \cite[Subsection 1.2]{Nualart}.
To make this concrete, we need to introduce some more notation. 
To this end, as in Subsection \ref{ss:proofs_further_results} and without loss of generality, we may assume that 
\begin{equation}
\label{eq:path_space_assumption_omega}
\O = C_0([0,T]; \WW \times \R^6),
\end{equation}
where $\WW$ is a Hilbert space such that the embedding $\Ls^2_0 \hookrightarrow \WW$ is Hilbert-Schmidt (for example, $\WW = \Hs^{-3/2 - \varepsilon}$ for $\varepsilon > 0$), and $\R^6$ accounts for the auxiliary process $\eta$. In the following, we let $\WW\stackrel{{\rm def}}{=} \Hs^{-3/2-\varepsilon}$ where $\varepsilon>0$ is chosen such that $\varepsilon<\alpha-\sigma-\frac{3}{2}$. The latter choice is always possible by \eqref{eq:coloring_assumption_noise_2}. Such a choice, in particular, implies
\begin{equation}
\label{eq:WW_Q_relation}
Q: \WW \to \Hs^{\sigma}.
\end{equation}

The It\^o-map associated to \eqref{eq:abstract_SPDE} is the mapping 
\begin{equation}
\label{eq:Ito_map}
\Ito_T:C_0([0,T];\WW\times \R^6)\to \HL
\end{equation} 
which satisfies
$
\Ito_T(W_T)= V_T(W) \text{ for a.a.\ }W\in C_0([0,T];\WW\times \R^6),
$
where $V_T$ is the solution at time $T$ of \eqref{eq:abstract_SPDE}. In particular, $\Ito_0\equiv V\in \HL$.

The existence of such a map can be performed as in the proof of Proposition \ref{prop:RDSII}, by considering the process $Z_t= V_t-(\Lambda_t,0,0,W^{(\eta)}_t)^\top$, where 
\begin{equation}
\label{eq:gamma_pathwise}
\textstyle
\Lambda_t= Q W_t - \int_0^t \Delta e^{(t-s)\Delta} QW_t\,\dd s,
\end{equation} 
and noticing that $\Lambda_t= \int_0^t e^{(t-s)\Delta} Q\, \dd W_s$ a.s.\ for all $t\geq 0$.

\begin{lemma}[Malliavin derivative from the It\^o-map]
\label{l:from_Malliavin_Ito}
Fix $T>0$. Let Assumption \ref{ass:Q} be satisfied and suppose that  \eqref{eq:path_space_assumption_omega} and \eqref{eq:WW_Q_relation} hold. Then the It\^o-map $\Ito_T$ is Fr\'echet differentiable. Moreover, the random variable $V_T\in \Do(\D)$ where $\D$ is the Malliavin derivative on $\HL$, and for all $(g_t)_t\in L^2(0,T;\Ls^2_0\times \R^6)$ and a.a.\ $W\in C_0([0,T];\WW\times \R^6)$, it is given as:
\begin{equation}
\label{eq:derivative_Ito_Malliavin}
\D V_T(W) g =  D\Ito_T(W) G
\end{equation}
where $G(\cdot)=\int_0^{\cdot} g_t'\,\dd t'$. 
Finally, for each $(g_t)_{t}\in L^2(\O\times (0,T);\Ls^2_0\times \R^6)$, the process $(\D_g V_t)_t \stackrel{{\rm def}}{=}
(\D V_t(W) g)_t$ solves (in the strong sense):
\begin{equation}
\label{eq:differential_problem_DgV}
\left\{
\begin{aligned}
\partial_t \D_g V_t &=  A \D_g V_t +D F(V_t) \D_g V_t +
Qg_t,\\
\D_g V_0&=0.
\end{aligned}
\right.
\end{equation}
\end{lemma}

Using \eqref{eq:gamma_pathwise}, the Fr\'echet differentiability of the It\^o map can be demonstrated as in \cite[Proposition 4.1]{HM11_theory}. For the equality in \eqref{eq:derivative_Ito_Malliavin}, the proof can be obtained by arguing as in \cite[Subsection 1.2.1]{Nualart}.

The key significance of \eqref{eq:differential_problem_DgV} is that it enables the computation of Mallavin derivatives of 
$V_T$ without requiring explicit calculation.
Finally, let us mention that from \eqref{eq:derivative_Ito_Malliavin} it follows that for a.a.\ $W\in C_0([0,T];\WW\times \R^6)$,
$$
\D_g V_T(W)  = \frac{\dd}{\dd \varepsilon}\Big|_{\varepsilon=0} \Ito_T \Big(W+\int_0^{\cdot} g_{t'}\,\dd t'\Big).
$$
Thus, the Malliavin derivative can be interpreted as a derivative in the Cameron-Martin-type direction $\int_0^{\cdot} g_{t'}\,\dd t'\in \prescript{}{0}{H}^{1}(0,T;\Ls^2_0\times \R^6)$. 

\subsubsection{Gradient estimate via controllability and integration by parts}
As standard in the study of strong Feller property for Markov semigroups (see e.g., \cite{DPZ_ergodicity,F99,FlaMas1995}), we prove Proposition \ref{prop:gradient_estimate} by finding for each $T>0$ and $h\in T_V\HL$ a suitable \emph{control} $g=(g_t)_t$ (possibly depending on $h$, $V$ and $T$) for which 
\begin{equation}
\label{eq:approximation_gradient_phiVT}
\D_g V_T  \approx D V_T h,
\end{equation}  
where $\approx$ means that the remainder $\D_g V_T - D V_T h$ is small in an appropriate sense, cf., Lemma \ref{l:bounds_control} below.
As in Lemma \ref{l:from_Malliavin_Ito}, in the above, we set 
$
\D_g V_T = \D V_T g .
$

From \eqref{eq:approximation_gradient_phiVT}, it follows that 
\begin{align}
\label{eq:estimate_malliavian_integration_by_parts}
D P_T \phi(V) h 
& =\E [D\phi(V_T) DV_T h] \\
\nonumber
& \approx \E\big(
D\phi(V_T) \D_g V_T \big)\\
\nonumber
&\stackrel{(i)}{=} \E\big(
\D_g [\phi(V_T)]\big)\\
\nonumber
&\stackrel{(ii)}{=} \E[
\phi(V_T) \delta(g)]
\end{align}
where in $(i)$, we applied the chain rule for the Malliavin derivative (see \cite[Proposition 1.2.4]{Nualart}), and in $(ii)$, we denote by $\delta$ the adjoint operator of $\D$ on $L^2(\O;\Hs^\sigma)$, usually referred to as the Skorohod operator, with its domain denoted by $\Do(\delta)$. Since
\begin{equation}
\label{eq:estimate_malliavian_integration_by_parts2}
\big|\E[
\phi(V_T) \delta(g)]\big|
\leq\|\phi\|_{L^\infty}\, \E|\delta(g)|,
\end{equation}
it follows that in the case \eqref{eq:approximation_gradient_phiVT} holds, the proof of Proposition \ref{prop:gradient_estimate} reduces to prove $g\in \Do(\delta)$ and estimating $\E|\delta(g)|$. The latter will be performed using the following result, which can be proven as in \cite[Propositions 1.3.1 and 1.3.11]{Nualart}. Below, we denote the subspace of progressively measurable processes with the subscript `${{\rm prog}}$'.

\begin{lemma}
\label{l:delta_lemma}
Let $\ww$ be an isonormal process on $L^2(0,T;\uu)$, where $T>0$ and $\uu$ is a separable Hilbert space. 
Let $\hh$ be another separable Hilbert space and let $\delta$ be the adjoint operator of the Malliavin derivative \eqref{eq:malliavin_derivative} with $\hh=\R$. Then the following hold.
\begin{enumerate}[{\rm(1)}]
\item\label{it:delta_lemma_1} $L^2_{{\rm prog}}(\O\times (0,T);\uu)\subseteq \Do(\delta)$, and for all $g\in L_{{\rm prog}}^2(\O\times (0,T);\uu)$, the random variable $\delta (g)$ coincide with the It\^o-integral $\int_0^{T}(  g_t,\dd \ww_t)_{\uu}$ and 
$$
\E|\delta(g)|^2= \E\|g\|_{L^2(0,T;\uu)}^2.
$$
\item\label{it:delta_lemma_2}  For all $(g_t)_t\in L^2(\O\times (0,T);\uu)$ such that $g_t\in \Do(\D)$ for a.a.\ $t\in (0,T)$ and $(\D_s g_t)_{s,t}:\O\times (0,T)^2 \to \calL_2(\uu)$ has a measurable modification, we have $g\in \Do(\delta)$ and
$$
\E| \delta(g)|^2
\leq 
\E \int_0^T \|g_t\|_{\uu}^2\,\dd t 
+
\E \int_0^T\int_0^T \| \D_s g_t \|_{\calL_2(\uu)}^2\,\dd s \, \dd t .
$$
\end{enumerate}
\end{lemma}  

Because of \eqref{it:delta_lemma_1}, the operator $\delta(g)$ can be considered as an enxtesion of the classical It\^o-integral and its action on $g$ is often denoted by $\int_0^T ( g_t ,\delta \,\ww_t )_{\uu}^2$.

\smallskip

With the above lemma, we now have all the necessary tools to prove Proposition \ref{prop:gradient_estimate} using the argument outlined in \eqref{eq:approximation_gradient_phiVT}-\eqref{eq:estimate_malliavian_integration_by_parts}. However, before proving the general case, we first apply this reasoning to the special case where the observable $\phi$ depends only on the velocity field $(v_t)_t$, thereby proving Proposition \ref{prop:strong_feller_v}. This will provide insight into the choice of the control $g$ in the general case, as discussed in Subsection \ref{ss:gradient_estimates} below.

\subsection{Intermezzo: Strong Feller for PEs -- Proof of Proposition \ref{prop:strong_feller_v}}
\label{ss:unique_ergodicity_PEs}
The aim of this subsection is to prove Proposition \ref{prop:strong_feller_v} via the argument described in \eqref{eq:estimate_malliavian_integration_by_parts} for the cutoff process. For the sake of completeness, let us recall the setting. Here $(v_t)_t$ is the solution to the \emph{truncated} stochastic PEs \eqref{eq:primitive_full}: 
\begin{equation}
\label{eq:v_primitive_cutoff}
\left\{
\begin{aligned}
\partial_t v_t &=\Delta v_t + \wt{F}(v_t)+ Q\dot{W}_t,\\ 
v_0&=v\in \Hs^\sigma,
\end{aligned}
\right.
\end{equation}
where 
$$
\wt{F}(v_t)=\chi_r(\|v_t\|_{H^\sigma})\,b(v_t,v_t)\quad\text{ for }\  r>0,
$$ 
and $b$ as in \eqref{eq:def_nonlinearity_b}. Here, as commented below Proposition \ref{prop:gradient_estimate}, we dropped the dependence on the parameter $r$.
The reduction of the strong Feller property from the original problem \eqref{eq:primitive_full} to the one with cutoff \eqref{eq:v_primitive_cutoff} follows from Lemma \ref{l:strong_feller_reduction}. 

We prove Proposition \ref{prop:strong_feller_v} via the gradient bound as in Proposition \ref{prop:gradient_estimate}, where the observable $\phi$ depends only on $v$. To prove the latter, we argue as in \eqref{eq:estimate_malliavian_integration_by_parts}. 
As mentioned in the previous subsection, the idea is to find for each $h\in \Hs^\sigma$ (recall that $\Hs^\sigma$ is the state space for \eqref{eq:v_primitive_cutoff}) a suitable control $(g_t)_t\in L^2(\O\times (0,T);\Ls^2)$ for which $\D_g v_T$ is a `good' approximation of $Dv_T h$, see \eqref{eq:approximation_gradient_phiVT} holds. 
To this end, let us $D v_T h$ and $\D_g v_T$ compute separately. 
Firstly, differentiating \eqref{eq:v_primitive_cutoff}, we obtain $D v_T h=\Jv_{T,0} h $ where $\Jv_{t,s}\in\calL(\Hs^\sigma)$ is the Jacobian operator for the velocity equation, which solves (in a strong operator sense) 
\begin{equation}
\label{eq:increments_v_J}
\left\{
\begin{aligned}
\partial_t \Jv_{s,t}
&=\Delta \Jv_{s,t} + D \wt{F}(v_t)\Jv_{s,t} ,\\
\Jv_{s,s}
&=\Id,
\end{aligned}
\right.
\end{equation}
for all $0\leq s<t<\infty$.
Secondly, by the analogue of Lemma \ref{l:from_Malliavin_Ito} for the system \eqref{eq:v_primitive_cutoff}, we have that ($\D_g v_t)_{t}$ exists for any direction $(g_t)_t\in L^2(0,T;\Ls^2)$ and it solves
\begin{equation}
\label{eq:increments_v_D}
\left\{
\begin{aligned}
\partial_t \D_g v_t
&=\Delta \D_g v_t + D F_b(v_t)\D_g v_t+Q g_t ,\\
\D_g v_0
&=0.
\end{aligned}
\right.
\end{equation}
By comparing \eqref{eq:increments_v_J} with \eqref{eq:increments_v_D}, the Duhamel principle imply
$$
\D_g v_T =\int_0^T \Jv_{t,T} Qg_t\,\dd t . 
$$
Now, if we choose
\begin{equation}
\label{eq:control_velocity}
g_t=2T^{-1}\one_{[T/2,T]}(t) Q^{-1}\Jv_{0,t}h,
\end{equation}
then the the semigroup-type property $\Jv_{s,t}=\Jv_{r,t}\Jv_{s,r}$ for $s<r<t$, which follows from the uniqueness of the linear problem \eqref{eq:increments_v_J}, yields 
\begin{equation}
\label{eq:derivative_g_equal_to_shift_malliavin_velocity}
\D_g v_T=\int_{T/2}^T \Jv_{t,T} \Jv_{0,t} h \,\dd t
=  
\Jv_{0,T} h= Dv_T h.
\end{equation}
Note that $Q^{-1}: \Ls^\sigma\to \Hs^{\sone}$ in \eqref{eq:control_velocity} is well-defined due to the assumption \eqref{eq:generalization2} and the density of the range of $Q$.
Given \eqref{eq:derivative_g_equal_to_shift_malliavin_velocity}, in the current situation and with choice of the control $(g_t)_t$ in \eqref{eq:control_velocity}, the formula \eqref{eq:approximation_gradient_phiVT} holds with equality.

To obtain the gradient estimate via \eqref{eq:estimate_malliavian_integration_by_parts}, we need to control $(g_t)_t$ in the domain of $\delta$. This will be accomplished via Lemma \ref{l:delta_lemma} and the following result.

\begin{lemma}[Regularity estimates for the Jacobian]
\label{l:instantaneous_regularization_J_appendix}
Suppose that the assumptions of Theorem \ref{t:uniqueness_regularity_invariant} hold.
Let $(\Jv_{s,t})_{0\leq s<t< \infty}\subseteq \calL(\Hs^\sigma)$ be the strong solution to \eqref{eq:increments_v_J} in $\Hs^\sigma$ such that, for all $h\in \Hs^\sigma$ and a.s., 
\begin{equation}
\label{eq:Jv_pathwise_regularity}
\Jv_{s,t} h \in C([s,\infty);\Hs^\sigma)\cap L^2([0,\infty);\Hs^{\sigma+1}) .
\end{equation}
Then, for all $\sigma_0<\sigma+2$ and $T\in (0,1]$, there exist deterministic constants $C_0,r_0,p_0\geq 0$ such that, for all $v,h\in \Hs^\sigma$ and a.s., 
$$
\E\int_{T/2}^T \|\Jv_{t,0} h\|_{H^{\sigma_0}}^2\,\dd t 
\leq C_0 T^{-r_0}(1+\|v\|_{H^\sigma}^{p_0}) \|h\|_{H^\sigma}^2.
$$
\end{lemma}

As we will see in the proof of Proposition \ref{prop:strong_feller_v} below, the possibility of choosing $\sigma_0>\sigma+\frac{3}{2}$is crucial for accommodating diagonal operators, as demonstrated in Example \ref{ex:operator_Q}. While this may not be surprising, it is important to emphasise that reaching the regime $\sigma_0\in(\sigma+1,\sigma+2)$ 
cannot be done directly and requires a two-step argument. Moreover, the regularity threshold is optimal because $Q\in \calL_2(\Ls^2,\Hs^\sigma)$, and therefore the velocity $v$ belongs to $\Hs^{1+\sigma}$ at most in space.

\begin{proof}
The existence of a strong solution to \eqref{eq:increments_v_J} that satisfies \eqref{eq:Jv_pathwise_regularity} is a standard result, thanks to the cutoff, and thus is omitted for brevity. However, we note that this can also be established using the estimate below, alongside the method of continuity. The proof is divided into two steps.

\smallskip

\emph{Step 1: Regularization estimates for $v$. For all $\sigma_1<\sigma+1$ there exist constants $C_1,\kappa_1,\kappa_2>0$ and $p_1\in (2,\infty)$ such that, for all $T\in (0,1]$,}
\begin{align}
\label{eq:Step1_regularization_proof_strong}
\E \sup_{t\in [0,T]} t^{\kappa_1}\|v_t\|_{H^{\sigma_1}}^{p_1}
&\lesssim C_1(1+\|v\|_{H^\sigma}^{p_1}),\\
\label{eq:Step1_regularization_proof_strong2}
\sup_{t\in [0,T]}t^{\kappa_2} \|\Jv_{0,T}h\|_{H^{\sigma_1}}&\lesssim_r \|h\|_{H^{\sigma}}.
\end{align}
As commented above, \eqref{eq:Step1_regularization_proof_strong} cannot be further improved since $Q$ is only in $\calL_2(\Ls^2,\Hs^\sigma)$. Below, we only prove \eqref{eq:Step1_regularization_proof_strong} as \eqref{eq:Step1_regularization_proof_strong2} follows similarly.  

To prove the bounds, here we apply stochastic maximal $L^p$-regularity estimates. To this end, fix $\sigma_1<\sigma_0$,  $p_0\in (2,\infty)$ and $\kappa_0\in [0,\frac{p_0}{2}-1)$ such that $\sigma_0<\sigma_1+1-\frac{2}{p_0}$ and $\sigma\geq \sigma_1+1-2\frac{1+\kappa_0}{p_0}$. Now, by stochastic maximal $L^p$-regularity for the Laplacian on the ground space $H^{\sigma_1-1}(\T^3;\R^2)$ with time weight $\kappa_0$ (see e.g., \cite[Theorem 7.16]{AV19} and \cite[Example 10.1.5]{Analysis2}), we obtain
\begin{align*}
&\E \sup_{t\in [0,T]} t^{\kappa_0}\|v_t\|_{B^{\sigma_1+1-2/p_0}_{2,p_0}}^{p_0}
+
\E \int_0^T t^{\kappa_0}\|v_t\|_{H^{\sigma_1+1}}^{p_0}\,\dd t \\
&\leq 
N \E\|v\|_{B^{\sigma_1+1-2(1+\kappa_0)/{p_0}}_{2,p_0}}^{p_0}
+ N \E\int_0^T \big(\|\wt{F}(v_t)\|_{\Hs^{\sigma_1-1}}^{p_0}+\|Q\|_{\calL_2(\Ls^2,\Hs^{\sigma_1})}^{p_0}\big)\,\dd t,
\end{align*}
Now, as $\sigma_1<\sigma$, by Lemma \ref{l:estimate_nonlinearity}\eqref{it:estimate_nonlinearity_H1} applied with $\varepsilon=\sigma-\sigma_1$,
\begin{align*}
\|\wt{F}(v_t)\|_{\Hs^{\sigma_1-1}}^{p_0}
&\leq \chi_r(\|v_t\|_{\Hs^\sigma})\|v_t\|_{H^{\sigma}}^{2p_0}\lesssim_r 1.
\end{align*}
The conclusion of Step 1 now follows from the fact that $B^{\sigma_1+1-2/p_0}_{2,p_0}\embed H^{\sigma_0}$ and 
$H^{\sigma_0} \embed B^{\sigma_1+1-2(1+\kappa_0)/p_0}_{2,p_0} $ by construction.

\smallskip

\emph{Step 2: Conclusion}.
The idea here is to use again maximal $L^2$-regularity estimates working on $[T/2,T]$ instead of $[0,T]$, and exploiting the estimates of Step 1. More precisely, for all $\sigma_0\in (\sigma+1,\sigma+2)$, maximal $L^2$-regularity estimates for the Laplace operator on $H^{\sigma_0-1}(\T^3)$ (see e.g., \cite[Theorem 3.5.7]{pruss2016moving}) ensures the existence of a constant $M>0$ independent of $T,v$ and $h$ such that
\begin{align}
\label{eq:estimate_Jv_T2T}
\int_{T/2}^T \|\Jv_{0,t} h\|_{H^{\sigma_0}}^2\,\dd t
&\leq M \|J_{0,T/2}h\|_{H^{\sigma_0-1}}^2
+ \int_{T/2}^T \|D\wt{F}(v_t)J_{0,t}\|_{\sigma_0-2}^2\,\dd t\\
\nonumber
&\stackrel{\eqref{eq:Step1_regularization_proof_strong2}}{\leq} M_1\|h\|_{H^{\sigma}}^2
+M \int_{T/2}^T \|D\wt{F}(v_t)J_{0,t}\|_{\sigma_0-2}^2\,\dd t,
\end{align}
where $M_1$ depend only on $M$ and $r$.
From Lemma \ref{l:estimate_nonlinearity}\eqref{it:estimate_nonlinearity_H1}, it follows that 
\begin{align*}
\|D\wt{F}(v_t)J_{0,t}\|_{\sigma_0-2}
\lesssim \|v_t\|_{H^{\sigma_0-1}}\|J_{0,t}h\|_{H^{\sigma_0-1}}.
\end{align*}
Now, the claim follows from the above, \eqref{eq:estimate_Jv_T2T} and the estimates in Step 1.
\end{proof}

\begin{proof}[Proof of Proposition \ref{prop:strong_feller_v}]
Let us begin by recalling that $Q^{-1}:\Hs^\g\to \Ls^2$, see the comments below \eqref{eq:derivative_g_equal_to_shift_malliavin_velocity}.
As mentioned at the beginning of this subsection, it is enough to prove the gradient estimate of Proposition \ref{prop:gradient_estimate} with $\phi\in C^1_{{\rm b}}(\Hs^\sigma)$.

Let $(g_t)_t$ be as in \eqref{eq:control_velocity}. Note that $(g_t)_t$ is progressively measurable and 
\begin{align*}
\E \int_0^T \|g_t\|_{\Ls^2}^2 \,\dd t 
&\eqsim \E \int_{T/4}^T \|Q^{-1}\Jv_{0,t} h\|_{\Ls^2}^2 \,\dd t \\
&\lesssim \E \int_{T/4}^T \|\Jv_{0,t} h\|_{\Hs^{\g}}^2 \,\dd t 
\lesssim T^{-r}(1+\|v\|_{H^\sigma}^{p}) \|h\|_{\Hs^\sigma}^2,
\end{align*}
where we used \eqref{eq:generalization2} and Lemma \ref{l:instantaneous_regularization_J_appendix} with $\sigma_0=\g<\sigma+2$. In particular, the constants $r$ and $p$ are independent of $T$, $v$ and $h$.

As observed above, in the current case, we have $\D_g v_T= Dv_T h$. Therefore, \eqref{eq:approximation_gradient_phiVT} holds as an equality. By reasoning similarly to the arguments used in \eqref{eq:estimate_malliavian_integration_by_parts}-\eqref{eq:estimate_malliavian_integration_by_parts2}, we can conclude that
\begin{align*}
|D P_t \phi(v)h|
&\leq\|\phi\|_{L^\infty} (\E \|(g_t)_t\|_{L^2(0,T;\Ls^2)}^2 )^{1/2}\\
&\lesssim 
\|\phi\|_{L^\infty}
T^{-r}(1+\|v\|_{H^\sigma}^{p}) \|h\|_{\Hs^\sigma}.
\end{align*}
Thus, Proposition \ref{prop:gradient_estimate} with $\phi\in C^1_{{\rm b}}(\Hs^\sigma)$ follows, and as explained at the beginning of this subsection, this yields Proposition \ref{prop:strong_feller_v}.
\end{proof}

\subsection{Gradient bound for the cutoff projective process -- Proof of Proposition \ref{prop:gradient_estimate}}
\label{ss:gradient_estimates}
To prove the latter, we adapt the arguments in \cite[Subsection 6.6.2]{BBPS2022} to the current situation.
More precisely, we obtain the following result where we denote by $r_T$ the remainder in \eqref{eq:approximation_gradient_phiVT}:
\begin{equation}
\label{eq:remainder}
r_T\stackrel{{\rm def}}{=} 
\D_g V_T- D V_T h.
\end{equation}

\begin{lemma}[Bounds on the control and remainder]
\label{l:bounds_control}
There exists $r_0$ depending only on the parameters introduced in Subsection \ref{sss:probabilistic_set_up} for which the following holds. 
For all $r>r_0$, there exist constants $a_0,b_0,T_0>0$ such that for all $0<T\leq T_0$, $V\in \HL$ and $h\in T_V \HL$ there exists a control $(g_t)_t\in \Do(\D)$ for a.a.\ $t\in (0,T)$ and $(\D_s g_t)_{s,t}:\O\times (0,T)^2 \to \calL_2(\Ls^2_0)$ has a measurable modification, and
\begin{align}
\label{eq:bounds_control_1}
\E\int_0^T \Big(\|g_t\|_{\Ls^2}^2+ \int_0^T \|\D_s g_t\|_{\calL(\Ls^2,\Hs^\sigma)}^2\,\dd s\Big) \,\dd t
&\lesssim_r T^{-a_0} (1+\|V\|_{\HL})^{b_0}\|h\|_{T_V\HL},\\
\label{eq:bounds_control_2}
\E\|r_T\|_{T_V\HL}^2&\lesssim_r T \|h\|_{T_V\HL},
\end{align} 
where $r_T$ is the remainder as in \eqref{eq:remainder}.
\end{lemma}

Note that $\calL(\Ls^2,\Hs^\sigma)\subseteq \calL_2(\Ls^2)$ as $\Hs^\sigma\embed \Ls^\infty$ due the fact that we are working on $\T^3$ and our assumption on $\sigma$, see \eqref{eq:coloring_assumption_noise_2}. Therefore, \eqref{eq:bounds_control_1} and Lemma \ref{l:delta_lemma}\eqref{it:delta_lemma_2} gives the control on $\delta(g)$ which is needed to conclude the argument in \eqref{eq:estimate_malliavian_integration_by_parts}. 

The fact that Lemma \ref{l:bounds_control} implies the gradient estimate of Proposition \ref{prop:gradient_estimate} is standard, and it is a slightly elaborated version of the argument in \eqref{eq:estimate_malliavian_integration_by_parts}. For details, the reader is referred to \cite[Proposition 6.1]{BBPS2022}.

\smallskip

Next, we construct the control $(g_t)_t$ for which the bounds in Lemma \ref{l:bounds_control} hold.

\subsubsection{Construction of the control}
To prove Lemma \ref{l:bounds_control}, we follow the splitting in high and low frequencies as in \cite{BBPS2022,EckHai2001,RX11}. 
The idea is to split the system \eqref{eq:abstract_SPDE_r} into a high-frequency part where the noise is active in all components and a low-order part that is \emph{finite} dimensional and only a \emph{few} modes are active but they are still sufficient to obtain improvement on the smoothness via an H\"ormander type condition, see Lemma \ref{l:spanning_condition}. 

To begin, let $\z_\l\subseteq \z$ be the `low' mode given by 
\begin{equation}
\label{eq:def_zl}
\z_\l\stackrel{{\rm def}}{=}\{(\kk,\ell)\in\z\,:\, |\kk|_{\infty}=1\}.
\end{equation}
Let $\Pi_\L:\H(\T^3)\to \H(\T^3)$ be the orthogonal projection onto the low modes, i.e., onto the space 
$$\sp\big\{\g_{(\kk,\ell)}e_{\kk}\,:\, (\kk,\ell)\in\z_\L\big\},$$ 
see Subsection \ref{sss:probabilistic_set_up} for the notation. Moreover, we let $\Pi_\h\stackrel{{\rm def}}{=}\Id_{\H(\T^3)} -\Pi_\l$. 
For future convenience, we also extend such projections on the state space $\HL$ for the augmented system \eqref{eq:augumented_system}, i.e., for $V=(v,\x,\xi,\eta)\in \HL$,
$$
\Pi_\h V= (\Pi_\h v,0,0,0)^{\top} , \qquad \text{ and }\qquad 
\Pi_\l V=(\Pi_\l v,\x,\xi,\eta)^\top ;
$$
where, with a slight abuse of notation, we still denoted by $\Pi_\h$ and $\Pi_\l$ the above extension to the product space $\HL$ of the latter projections initially defined on $\Hs^\sigma$. 

By applying the projections $\Pi_\l$ and $\Pi_\h$ to \eqref{eq:abstract_SPDE_r}, it is clear that the processes 
$$
V^\h_t\stackrel{{\rm def}}{=}\Pi_\h V_t\qquad \text{ and }\qquad V^\l_t\stackrel{{\rm def}}{=}\Pi_\l V_t
$$ 
satisfy the following system of SPDEs (recall that we dropped superscripts in \eqref{eq:abstract_SPDE_r}):
\begin{equation}
\label{eq:high_low_splitting}
\left\{
\begin{aligned}
&\partial_t \v_t^\l + A_\l \v_t^\l  = F_\l(\v_t)+ \wh{Q}_\l \dot{W}_t,\\
&\partial_t \v_t^\l + A_\h \v_t^\h  = F_\h(\v_t)+ \wh{Q}_\h \dot{W}_t,
\end{aligned}
\right.
\end{equation}
with initial data $V_0^\l = \Pi_\l V$, $V_0^\h = \Pi_\h V$, where for $V\in \HL$,
\begin{align*}
A_\h V&= \Pi_\h A V, \qquad 
F_\l (V)=\Pi_\l F(V), \qquad  Q_\l =\Pi_\l \wh{Q},\\
A_\h V&= \Pi_\h A V, \qquad 
F_\l (V)=\Pi_\l F(V), \qquad  Q_\l =\Pi_\l \wh{Q}.
\end{align*}

Let us remark that the SPDE in \eqref{eq:high_low_splitting} are coupled via the nonlinearities $F_\l(V)$ and $F_\h (V)$.
Moreover,
\begin{equation}
\label{eq:def_QL}
Q_\l = \Pi_\l \wh{Q}=(\Pi_\l Q, 0,0,\Id_{\R^6})^\top,
\end{equation}
where $\Id_{\R^6}$ is the identity matrix in $\R^6$. 
The state space for the unknown $V^\l_t$ and $V^\h_t$ are
\begin{equation}
\HL_\l \stackrel{{\rm def}}{=}
(\Pi_\l \H(\T^3))\times \Dom\times  \S^2\times \R^6
\qquad \text{ and }\qquad 
\HL_\h \stackrel{{\rm def}}{=}
\Pi_\h\H(\T^3),
\end{equation}
respectively. 
Accordingly, we define low and high mode approximation of the Jacobian for \eqref{eq:abstract_SPDE_r}. More precisely, we let 
\begin{equation}
\label{eq:high_low_splitting_jacobian}
\left\{
\begin{aligned}
&\partial_t R_{s,t}^\l + A_\l R_{s,t}^\l  = D_\l F_\l(\v_t)R_{s,t}^\l, \qquad R_{s,s}^\l =\Id\\
&\partial_t R_{s,t}^\h + A_\h R_{s,t}^\h  =D_\h F_\h(\v_t)R_{s,t}^\h, \qquad R_{s,s}^\h =\Id.
\end{aligned}
\right.
\end{equation}
Note that $R_{0,t}^\l $ and $R_{0,t}^\h$ are \emph{not} the projection of the Jacobian $J_{0,t}=D V_t $ onto the spaces $\HL_\l$ and $\HL_\h$, respectively. Indeed, the projections $\Pi_\l$ and $\Pi_\h$ do not commute with the operator $J\mapsto D F(v) J$. 
 However, as we expect $J_{0,t}\approx \Id$ for small times, then $\Pi_\l R_{0,t}^\h \approx 0 $ and $\Pi_\h R_{0,t}^\l \approx 0 $ for small times and therefore $R^\h_{0,t}$ and $R_{0,t}^\l$ are good approximation of the full Jacobian $J_{0,t}$.

For notational convience, for $V^\l\in \HL_\l$, we let 
$$
T_{V^\L} \HL_\l =\Pi_{\L} (\H(\T^3))\times \R^3\times T_{\xi} \S^2\times \R^6.
$$
From the above definitions, for all $s,t\geq 0$, we have  
$
R_{t,s}
$ 
is a linear mapping from $T_{V_s^\L} \HL_\l$ to $
T_{V_t^\L} \HL_\l$, and $Q_\l:T_{V^\l_s}\HL\to T_{V^\l_t}\HL$.  
Finally, we also need the inverse of the operator $R^\l_{s,t}$. 
Let 
\begin{equation}
\label{eq:inverse_R}
S^\l_{s,t}=(R^{\l}_{s,t})^{-1}\in\calL(T_{V_t^\L} \HL_\l,T_{V_s^\L} \HL_\l).
\end{equation}
Note that, from the resolvent identity, $\partial_t S^\l_{s,t}= (R^{\l}_{s,t})^{-1}\partial_t R^\l_{s,t} (R^{\l}_{s,t})^{-1}$ and therefore the first in \eqref{eq:high_low_splitting_jacobian} yields
\begin{equation}
\label{eq:inverse_low_jacobian}
\partial_t S_{s,t}^\l = S_{s,t}^\l A_\l + S_{t,s}^\l D_\l F(V), \qquad S_{s,s}^\l =\Id.
\end{equation}

Next, let us comment on the well-definedness of the inverse in \eqref{eq:inverse_R}. 
Note that $R^{\l}_{s,t}$ is a-priori invertible only for short times as $R^\l_{s,t}\to \Id$ as $t\downarrow s$. However, the global well-posedness of the problem \eqref{eq:inverse_low_jacobian} shows the global in-time invertibility of the approximate Jacobian $R^\l_{s,t}$.

\smallskip

Finally, we define the last object needed to construct the control $(g_t)_t$, i.e., the \emph{partial Malliavin matrix} $\Ma$. Recall that $V=(V_\l,V_\h)\in \HL$ is fixed and $(V^\l_t)_t$ is the solution to the low-mode system for such initial data. 
Then, the partial Malliavin matrix 
$\Ma:T_{V_\l}\HL \to 
T_{V_\l}\HL$ is given by
\begin{align}
\label{eq:partial_malliavin_matrix_def}
\Ma 
\stackrel{{\rm def}}{=}
\int_0^t S^\L_{0,s} Q_\L(S_{0,s}^\L Q_\L)^\top \,\dd s, 
\end{align} 
where we used that $Q_\L:T_{V^\l}\HL\to T_{V^\l_s}\HL$ for all $s\geq 0$ and $V^\l \in \HL_\l$, see \eqref{eq:def_QL}.

The following result shows that $\Ma$ is nondegenerate. The nondegeneracy of the Malliavin matrix allows us to build the control $(g_t)_t$.  

\begin{lemma}[Bounds and nondegeneracy of the partial Malliavin matrix]
\label{l:invertibility_matrix}
There exists $T_1>0$ depending only on the parameters introduced in Subsection \ref{sss:probabilistic_set_up} and $r>0$ for which the following holds. For all $T\in (0,T_1]$, the partial Malliavin matrix $\Ma$ is a.s.\ invertible on $T_{V_\l}\HL$. Moreover, there exist constants $a,b>0$ such that
$$
\E |(\Mat)^{-1}|^p\lesssim_p T^{-ap}\dist(\x)^{-b}(1+|\eta|)^{bp} \  \text{ for all } p>1.
$$
\end{lemma}

The proof of Lemma \ref{l:invertibility_matrix} is given at the beginning of Section \ref{ss:nondegeneracy} below.
Next, with the aid of the partial Malliavin matrix, we construct the control $(g_t)_t$, which will be used to proof of Lemma \ref{l:bounds_control}.

To define the control $(g_t)_t$, we start with the low-frequency part. In light of Lemma \ref{l:invertibility_matrix}, we can choose the low mode component of the control $(g_t)_t$ as
\begin{align}
\label{eq:control_glow_part}
g_t^\l &\stackrel{{\rm def}}{=} (S^\L_{0,t}Q_\l)^\top (\Mat)^{-1} S_{0,T}^\l D V_T^\l h_{\l} .
\end{align}

Before defining the high part of the control, let us first comment on the structure of the low-frequency control. Arguing as in Lemma \ref{l:from_Malliavin_Ito} (and assuming the high-frequency of the control to be given), it follows that $\D_g V^\l_t$ solves
\begin{align}
\label{eq:DgVlow_equation}
\partial_t \D_g V^\l_t 
+A_\l V^\l_t 
&= D F_\l (V_t) \D_g V_t + \wh{Q}^\l g_t\\
\nonumber
&= D_\l F_\l (V_t) \D_g V_t^\l +D_\h F_\l (V_t) \D_g V_t^\h + \wh{Q}^\l g_t^\l 
\end{align}
with zero initial condition, i.e., $\D_g V^\l_t=0$. 
Reasoning as in \eqref{eq:increments_v_J}-\eqref{eq:increments_v_D} and using $Q_\l g =Q_\l g^\l$, we can compute $
\D_{g} V_T^\l $ via \eqref{eq:high_low_splitting_jacobian} and the Duhamel principle:  
\begin{align}
\label{eq:malliavin_derivative_low1}
\D_{g} V_T^\l 
=\int_0^T R_{t,T}^\l D_\h F_\l (V_t)\D_g V^\h_t \,\dd t +
\int_0^T R_{t,T}^\l Q g^\l_t \,\dd t.
\end{align}
For the last term on the right-hand side of the previous, we have
\begin{align}
\label{eq:malliavin_derivative_low2}
\int_0^T R_{t,T}^\l Q g^\l_t \,\dd t
&= 
R_{0,T}^\l \int_0^T  S_{0,t}^\l  Q g^\l_t \,\dd t\\
\nonumber
&= 
R_{0,T}^\l \Big(\int_0^T  S_{0,t}^\l  Q (S_{0,t}^\l  Q)^\top \,\dd t\Big)(\Mat)^{-1} S_{0,T}^\l D V_T^\l h_{\l}\\
\nonumber
&\stackrel{{(i)}}{=}
R_{0,T}^\l  S_{0,T}^\l D V_T^\l h_{\l}
\stackrel{{(ii)}}{=}D V_T^\l h_{\l},
\end{align}
where $(i)$ and $(ii)$ follow from the definition of $\Mat$ and \eqref{eq:inverse_R}, respectively.
In particular, the second term in \eqref{eq:malliavin_derivative_low1} and the choice of $g^\l$ in \eqref{eq:control_glow_part} allows us to reproduce a variation in the initial condition $h$ in the low mode part of the dynamics of \eqref{eq:high_low_splitting}, while the second term $\int_0^T R_{t,T}^\l D_\h F_\l (V_t)\D_g V^\h_t \,\dd t$ will be treated as a remainder. 

As for the high-mode part, we would like to choose a control similar to the one used in \eqref{eq:control_velocity}, where, ideally, $(Q,\Jv)$ are replaced by $(Q^\h,R^\h)$. However, in the current situation, it turns out that the error of the low frequencies compared to the high ones is too large.
This issue arises because the parameter $a>0$ in Lemma \ref{l:invertibility_matrix} can be quite large, making the control of the partial Mallavin matrix, which governs the low frequencies, so singular as $T \downarrow 0$ that the smoothing effect of $A_\h$ cannot compensate its contribution. To solve this issue, as in \cite[Section 6]{BBPS2022}, we use a frequency cutoff in the high-frequencies control:
\begin{equation}
\label{eq:control_glow_part1}
g_t^\h\stackrel{{\rm def}}{=}-Q_\h^{-1} \Pi_{\leq N} D_\l F_\h (V_t) \gamma_t +\one_{[T/2,T]}(t) 2T^{-1} Q_{\h}^{-1} R_t^\h h_\h ,
\end{equation}
where $\Pi_{\leq N}$ is the projection onto frequencies $\leq N$, $N\stackrel{{\rm def}}{=} T^{-2a} (1+|\eta|)^{2b}$ (with $a$ and $b$ as in Lemma \ref{l:invertibility_matrix}) and $\g_t$ is a $T_{V_t^\l} \HL$-valued process solving the following system:
\begin{align}
\label{eq:partial_malliavin_derivative_low_high_system}
\left\{
\begin{aligned}
&\dot{\gamma}_t+A_\l \gamma_t
= D_\l F_\l (V_t) \gamma_t + Q_\l g_t^\l +D_\h F_\l (V_t)\delta_t,\\
&\dot{\delta}_t
+A_\l \gamma_t=   D_\h F_\h (V_t) \delta_t + \Pi_{>N} D_\l F_\h (V_t)\g_t +\one_{[T/2,T]}(t) 2T^{-1}  R^\h_t h_\h ,\\
&\gamma_0=0, \quad\quad \delta_0=0.
\end{aligned}
\right.
\end{align}
The system \eqref{eq:partial_malliavin_derivative_low_high_system} is constructed in such a way that the natural candidates for the solutions $\gamma_t$ and $\delta_t$ are $\D_g V^\l_t$ and $\D_g V^\h_t$, respectively; cf., \eqref{eq:DgVlow_equation} for the first equation. 
Note that we could not directly define \eqref{eq:control_glow_part1} with 
$\gamma_t$ replaced by 
$\D_g V^\l_t$, because defining the latter requires knowing the high part of the control 
$(g_t)_t$, which we aim to define. Therefore, by considering the system \eqref{eq:partial_malliavin_derivative_low_high_system}, we avoid the previous circular argument.
The usefulness of the first term in the high-frequency control \eqref{eq:control_glow_part1} will be clear below when constructing the remainder $r_T=DV_T h - \D_g V_T$, cf., Lemma \ref{l:bounds_control}. 

Next, we discuss the existence of solutions to \eqref{eq:partial_malliavin_derivative_low_high_system} and the consistency property with the Malliavin derivative. 

\begin{lemma}
\label{l:solvability_system_gamma_delta}
There exists $T_2>0$ depending only on the parameters introduced in Subsection \ref{sss:probabilistic_set_up} and $r>0$ for which the following holds. For all $T\in (0,T_2]$ and $p\geq 2$, there exists a unique solution $(\gamma_t,\delta_t)_t\in L^p(\O;C([0,T];T_{V_t}\HL))$ satisfying
\begin{align*}
\E \sup_{t\in [0,T]} \|(\g_t,\delta_t)\|_{\HL}^p + \E \sup_{s,t\in [0,T]} \|\D_s (\g_t,\delta_t)\|_{\calL(\Ls^2,T_{V_t}\HL)}^p 
\lesssim_p T^{-2ap} (1+|\eta|)^{2bp} \|h\|_{\HL}^p.
\end{align*}
Moreover, $\gamma_t=\D_g V^\l_t$ and $\delta_t=\D_g V^\h_t$ for a.a.\ $t\in [0,T]$.
\end{lemma}

In the above, $C([0,T];T_{V_t}\HL)$ stands for the set of maps continuous maps $f$ on $[0,T]$ such that $f_t\in T_{V_t}\HL$ for all $t\in [0,T]$. 
The above result is verbatim from the argument in \cite[Lemma 6.10]{BBPS2022}, where instead of the nonlinear estimates for the Navier-Stokes nonlinearity, one uses Lemma \ref{l:estimate_nonlinearity} and the truncation in $\Hs^\sigma$ of the nonlinearity in \eqref{eq:Ftruncation_nonlinearity}. 
Therefore, we omit the details.

\subsubsection{Proof of Lemma \ref{l:bounds_control}}
In this subsection, we finally give the proof of Lemma \ref{l:bounds_control}. 
We begin by giving the explicit formula for the remainder 
$$
r_T= D V_T h - \D_g V_T .
$$
Arguing as in \eqref{eq:malliavin_derivative_low1}-\eqref{eq:malliavin_derivative_low2} and in light of Lemma \ref{l:solvability_system_gamma_delta}, we have that 
$$
\D_g V^\h_t = \int_0^T R_{t,T}^\h D_\l F_\h (V_t) \Pi_{>N}\D_g V_t^\l \,\dd t  + R_{0,T}^\h h_\h, 
$$ 
where we used $\wh{Q}_\h g=\wh{Q}_\h g^\h $ and the last term arises from the identity $ R_{t,T}^\h  R_{0,t}^\h =R_{0,T}^\h$ for all $t\in [0,T]$ as in \eqref{eq:derivative_g_equal_to_shift_malliavin_velocity}.
Similarly, for $h_\h\in \HL_\h$, we have
\begin{equation}
\left\{
\begin{aligned}
\partial_t DV^\h_t h_\h +A_\h DV^\h_t h_\h
&= D F_\h (V_t)DV_t h_\h \\
&= D_\h F_\h (V_t)DV_t^\h h+D_\l F_\h (V_t)DV^\l_t h,\\
DV^\h_0 h_\h &=h_\h.
\end{aligned}
\right.
\end{equation}
Thus, if $h=(h_\l,h_\h)\in \HL$, then 
\begin{align*}
D V^\h_T h
&= D_\l V^\h_T h_\l + D_\h V^\h_T h_\h\\
&=D_\l V^\h_T h_\l + R^\h_{0,T} h_\h+ \int_0^T R^\h_{T,t} D_\l F_\h (V_t)DV^\l_t h_\h\, \dd t.
\end{align*}

By the above formulas and \eqref{eq:malliavin_derivative_low1}-\eqref{eq:malliavin_derivative_low2}, we obtain $r_T=(r_T^\h,r_T^\l)$ where
\begin{align}
\label{eq:remainder_1}
r_T^\l
&=
\int_0^T R^\l_{t,T} D_\h F_\l (V_t)\delta_t \,\dd t ,\\
\label{eq:remainder_2}
r_T^\h &= \int_0^T R_{t,T}^\h\Pi_{>N} D_\l F_\h (V_t) \g_t \,\dd t \\
\nonumber
& -D_\l V^\h_T h_\l-\int_0^T R^\h_{T,t} D_\l F_\h (V_t)DV^\l_t h_\h\,\dd t.
\end{align}

\begin{proof}[Proof of Lemma \ref{l:bounds_control}]
From the expression of the control $g_t=(g_t^\l,g_t^\h)$ defined in \eqref{eq:control_glow_part} and \eqref{eq:control_glow_part1}, as well as the remainder $r_T=(r_T^\l,r_T^\h)$ in \eqref{eq:remainder_1}-\eqref{eq:remainder_2}, the estimates in Lemma \ref{l:bounds_control} follow verbatim from the argument used in \cite[Lemma 6.5]{BBPS2022}, by employing analogous results from \cite[Lemmas 6.22-6.24]{BBPS2022} in the present context. As above, the extension of the latter results follows from the truncation in the nonlinearity \eqref{eq:Ftruncation_nonlinearity} and from Lemma \ref{l:estimate_nonlinearity}.

Let us emphasise that, as in the proof of Proposition \ref{prop:strong_feller_v} given at the end of Subsection \ref{ss:unique_ergodicity_PEs} (see also Example \ref{ex:operator_Q}), the lower bound $\sigma>\alpha-2$ in assumption in \eqref{eq:coloring_assumption_noise_2} is needed to prove the claimed estimate for $\E\int_0^T \|g_t\|_{\Ls^2}^2\,\dd t$.
\end{proof}

\subsection{Nondegeneracy of the partial Malliavin matrix -- Proof of Lemma \ref{l:invertibility_matrix}}
\label{ss:nondegeneracy}
The proof of Lemma \ref{l:invertibility_matrix} of the above result via the following

\begin{lemma}
\label{l:invertibility_matrix_suff}
There exists $T_3>0$ depending only on the parameters introduced in Subsection \ref{sss:probabilistic_set_up} and $r>0$ for which the following holds. Then, for all $p\geq 1$ and $T\in (0,T_3]$, 
$$
\E|\Mat|^p\lesssim_p 1.
$$
Moreover, there exist costants $a_1,b_1>1$ such that, for all $\varepsilon>0$ and $V\in \HL_\l$, 
\begin{equation}
\label{eq:inf_bound_malliavin_matrix}
\sup_{|h|_{\HL_\l}=1} \P \big( \langle\Mat h,h\rangle <\varepsilon\big) 
\lesssim_{p} t^{-a_1p}\dist(\x)^{-b_1} (1+|\eta|)^{b_1p} \varepsilon^p,
\end{equation}
where $\langle\cdot,\cdot\rangle$ denotes the inner product in $\HL_\l$.
\end{lemma}

Let us first show how Lemma \ref{l:invertibility_matrix_suff} implies Lemma \ref{l:invertibility_matrix}. 
The argument below largely follows the one in \cite[Lemma 4.7]{Hai11}, where we keep track of the constants which are needed because of the explosion as $T\to 0$.

\begin{proof}[Proof of Lemma \ref{l:invertibility_matrix}]
Let $N=\dim \HL_\L$, and recall that $T_{V} \HL_\l =\Pi_{\L} (\H(\T^3))\times \R^3\times T_{\xi} \S^2\times \R^6$ for $V\in \HL_\l$. 
Arguing as in \cite[Lemma 4.7]{Hai11}, by a covering argument, there exists an absolute constant $C_0>0$ such that, for all $\varepsilon>0$, 
$$
\textstyle
\P\big(\inf_{|h|=1} \langle \Mat h,h\rangle<\varepsilon\big) \leq C_0 \varepsilon^{2-2N} 
\P\big(\langle \Mat h,h\rangle<4\varepsilon\big)
+ \P\big(|\Mat|>\varepsilon^{-1}\big).
$$
Since $|\Mat|^{-1}=(\inf_{|h|=1} \langle \Mat h,h\rangle)^{-1}$ and $t\leq 1$, we obtain
\begin{align*}
\E|\Mat|^{-1}
&
\textstyle
\lesssim 1
+ \int_1^\infty  \lambda^{p+2N-1} \big[\P(\langle \Mat h,h\rangle<(4\lambda)^{-1})
+ \P(|\Mat|>\lambda)\big]\,\dd \lambda\\
&\lesssim t^{-a_1(p+2N+1)}\dist(\x)^{-b_1(p+2N+1)} (1+|\eta|)^{b_1(p+2N+1)},
\end{align*}
where in the last step we applied Lemma \ref{l:invertibility_matrix_suff} with $p$ replaced by $p+2N+1$.
\end{proof}

\subsubsection{Proof of Lemma \ref{l:invertibility_matrix_suff}}
\label{sss:hypoelliptic}
The proof of Lemma \ref{l:invertibility_matrix_suff} requires some preparation. 
To begin, let us introduce some notation. Recall that $\z_\l$ is defined in \eqref{eq:def_zl}. Set
\begin{equation}
\label{eq:augumented_indeces}
\zl= \z_\l \cup \{1,\dots,6\},
\end{equation}
and
\begin{equation}
\label{eq:qq_j_def}
\qq_{j} = 
\left\{
\begin{aligned}
&(q_{(\kk,\ell)}\g_{(\kk,\ell)} e_\kk, 0,0,0)^\top\quad &\text{ if }& j= (\kk,\ell)\in \z_\l,\\
&(0,0,0,a_j^{(\eta)})^\top\quad &\text{ if }&  j\in \{1,\dots,6\},
\end{aligned}
\right.
\end{equation}
where $(a_j^{(\eta)})_{j=1}^6$ is the standard basis of $\R^6$. The zeros in the above definition ensure that $\qq_j\in T_{V_\l}\HL$ for all $V_\l\in \Pi_\l (\H(\T^3))$.
Note that $\zl$ gather all the `directions' in the low-part of the system \eqref{eq:high_low_splitting} where the noise is present.

Next, we turn to the proof of \eqref{eq:inf_bound_malliavin_matrix}. Let us begin by noticing that, for $h\in \HL_\l$, 
\begin{equation}
\label{eq:identity_malliavin_matrix}
\langle \Ma h,h\rangle
=
\sum_{j\in \zl}\int_0^t \langle S^\l_s \qq_j h,h\rangle^2 \,\dd s
\end{equation}
where, with a slight abuse of notation, on the LHS and RHS of \eqref{eq:identity_malliavin_matrix} we indicate with $\langle \cdot,\cdot \rangle$ the scalar product in $T_{V_\l}\HL$ and in $T_{V_\l(s)}\HL$, respectively. A similar abuse of notation will also be used below.
Now, for $ j\in \zl$, we let 
$$
X_t^j\stackrel{{\rm def}}{=}
\langle S_{0,t}^\L \qq_j,h\rangle,
$$
and using that \eqref{eq:inverse_low_jacobian}, we obtain
\begin{equation*}
\left\{
\begin{aligned}
\partial_t X_t^j &=  -\langle A_\l \qq_j ,h\rangle+\langle S^\l_{0,t} D_\l F_\l (V_t)\qq_j,h\rangle ,\\
X_0^j &= \langle \qq_j,h\rangle.
\end{aligned}
\right.
\end{equation*}

Now, arguing as in \cite[Proposition 6.12]{BBPS2022}, the key point is to study H\"ormander type conditions for the vectors $(\qq_j,A_\l \qq_j -D_\l F_\l (V_t)\qq_j)_{j\in \zl}$ appearing in the above. 
The following is the key algebraic result to prove \eqref{eq:inf_bound_malliavin_matrix}, cf., \cite[Lemma 6.15]{BBPS2022}.

\begin{lemma}
\label{l:spanning_condition}
There exists a universal constant $r_0>0$ for which the following assertion holds provided $r\geq r_0$.  
For all $V=(v,\x,\xi,\eta)\in \HL$,
\begin{equation*}
\svec\big\{ \qq_{j}, [A_\l  -D_\l F_\l (v,\x,\xi)]\qq_j\,:\, j\in \zl\big\} = \HL_\l.
\end{equation*}
Moreover, there exists $C_0>0$ independent of $V=(v,\x,\xi,\eta)\in\HL$ such that
\begin{equation}
\label{eq:uniform_bound_below_algebraic_condition_full}
\max_{j\in \zl} \big\{ |\langle \qq_j ,h \rangle|, |\langle [A_\l  -D_\l F_\l (v,\x,\xi)]\qq_j ,h \rangle|\big\}
\geq  \frac{C_0(\dist(\x))^{4}}{(1+|\eta|)^{3}} |h|
\end{equation}
for all $h\in T_{V}\HL_\l$.
\end{lemma}

The implication from Lemma \ref{l:spanning_condition} to Lemma \ref{l:invertibility_matrix_suff} is well-known and follows verbatim from the proof of \cite[Lemma 6.15]{BBPS2022}. The details are omitted for brevity. 
We note that the Lie brackets, which typically appear in algebraic lemmas of this kind, are absent in Lemma \ref{l:spanning_condition} because the functional $\langle \cdot,\qq_j\rangle$ used to obtain the system for $
X_t^j$ does not depend on the solution $(V_t)_t$ of the truncated SPDE \eqref{eq:abstract_SPDE_r}. As a result, the first-order commutators reduce to $[A_\l  -D_\l F_\l (v,\x,\xi)]\qq_j$.

\smallskip

In the remaining part of this subsection, we focus on the proof of Lemma \ref{l:spanning_condition}. In the latter result, the structure of the PEs plays a key role. 
To begin, let us decompose $F_\l$ as follows:
\begin{align*}
F_\l (V)
&=\big[ 1-\chi_{3r}(\|v\|_{\Hs^\sigma(\T^3)}) \big]\big (\fvel_\l(v)+\fpr_\l(v,\x,\xi) \big)\\
&+\chi_{r}(\|v\|_{\Hs^\sigma(\T^3)}) H(\x,\xi,\eta), 
\end{align*}
where $\fvel_\l$ and $\fpr_\l$ are the velocity and projective part of the nonlinearity low mode projection of the nonlinearity $F$ in \eqref{eq:abstract_SPDE}, i.e., 
\begin{align}
\label{eq:def_Fvv}
\fvel_\l(v)
&=
\big(- \Pi_\l\p[(v\cdot\nabla_{x,y})v+w(v)\partial_z v], 0,0,0\big)^\top,\\
\fpr_\l(v,\x,\xi)
&=
\big(0, u(\x), \Pi_\xi [\nabla u(\x)\xi], 0\big)^{\top}
\end{align}
for $V=(v,\x,\xi,\eta)$. Recall that $u=(v,w(v))$ where $w(v)$ is as in \eqref{eq:def_w}.

\smallskip

The following constitutes the main ingredient in the proof of Lemma \ref{l:spanning_condition}.

\begin{lemma}
\label{l:spanning_condition1}
There exist positive constants $r_0,c_0$ such that, for all $r\geq r_0$, $V=(v,\x,\xi,\eta)\in \HL$ and  
$h=(h_v ,h_\x,h_\xi,h_\eta)\in T_{V}\HL_\l$, 
\begin{align}
\label{eq:uniform_bound_below_algebraic_condition_1}
\max_{j\in \zl}|\langle \qq_j ,h \rangle|&\geq c_0\big( |h_v| + |h_\eta|\big),\\
\label{eq:uniform_bound_below_algebraic_condition_2}
\max_{j\in \z_\l}|\langle D_\l \fpr_\l (v,\x,\xi)\qq_j  ,h \rangle|&\geq c_0(\dist(\x))^{4}\big( |h_\x| + |h_\xi|\big).
\end{align}
\end{lemma}

We emphasise that in \eqref{eq:uniform_bound_below_algebraic_condition_2}, the maximum is taken exclusively over $\z_\l$, deliberately excluding the $\eta$-directions in $\zl\setminus \z_\l$. This will be crucial later to prove \eqref{eq:uniform_bound_below_algebraic_condition_full} for low values of $\|v\|_{\H}$ at which the truncation in \eqref{eq:Ftruncation_nonlinearity} is not active.

Before going into the proof of Lemma \ref{l:spanning_condition1}, we collect some facts.
If $v\in \Pi_\l (\H(\T^3))$, then for some constants $(v_{(\kk,\ell)})_{(\kk,\ell)\in \z_\l}$ it holds that (see also the beginning of Subsection \ref{ss:solutions_def} for the notation)
\begin{equation}
\label{eq:v_decomposed_in_low_modes}
\textstyle
v(\x) =\sum_{(\kk,\ell)\in \z_\l} v_{(\kk,\ell)} \g_{(\kk,\ell)} e_{\kk}(\x) \ \ \text{ for } \ \x\in \T^3.
\end{equation} 
By \eqref{eq:def_w}, the third component of the corresponding velocity field is given by 
\begin{align*}
\textstyle
w(\x)
=- \int_0^{z} \nabla_{x,y} \cdot v (x,y,z')\,\dd z'
= 
\sum_{(\kk,\ell)\in \z_\l}  \g^{(w)}_{(\kk,\ell)} v_{(\kk,\ell)}\big( e_{\kk}(\x) 
 - e_{\kk}(\x_0)\big)
\end{align*}
where $\x=(x,y,z)\in \Dom$, $\kk=(k_x,k_y,k_z)\in \Z^3_0$, and we set
\begin{align*}
\x_0&\stackrel{{\rm def}}{=} (x,y,0), \\
\g^{(w)}_{(\kk,\ell)}
&\stackrel{{\rm def}}{=}
- \, k_z^{-1}(k_x,k_y)\cdot a_\ell \ \ \text{ if }  k_z\neq 0,\quad \text{ or }
\g^{(w)}_{(\kk,\ell)}
\stackrel{{\rm def}}{=}0  \  \text{ otherwise. } 
\end{align*}
Note that $\g^{(w)}_{(\kk,\ell)}=\g^{(w)}_{(-\kk,\ell)}$ by symmetry of $\g_{(\cdot,\ell)}$. Letting 
$$
\g^{(u)}_{(\kk,\ell)}\stackrel{{\rm def}}{=} (\g_{(\kk,\ell)}, \g^{(w)}_{(\kk,\ell)})\in \R^3
\quad  \text{ and }
\quad
\beta^{(u)}_{(\kk,\ell)}\stackrel{{\rm def}}{=} (0, 0,\g^{(w)}_{(\kk,\ell)})\in \R^3,
$$
we obtain that the velocity field $u$ corresponding to an horizontal velocity $v\in \Pi_\l(\H(\T^3))$ decomposed as \eqref{eq:v_decomposed_in_low_modes} has the form 
$$
\textstyle
u (\x)=\sum_{(\kk,\ell)\in \z_\l} v_{(\kk,\ell)} \big(\g^{(u)}_{(\kk,\ell)} e_{\kk}(\x)-\beta^{(u)}_{(\kk,\ell)} e_{\kk}(\x_0) \big)  .
$$ 
Therefore, for each $\xi\in \S^2$ and $\x\in\Dom$, 
\begin{align*}
\Pi_{\xi} \big[\nabla u(\x)\xi\big] 
\textstyle
= \sum_{(\kk,\ell)\in \z_\l}v_{(\kk,\ell)}  
\big((\kk\cdot \xi) e_{-\kk}(\x) \,\Pi_{\xi}\g_{(\kk,\ell)}^{(u)} - (\kk\cdot \xi_0) e_{-\kk}(\x_0)\,\Pi_{\xi}\beta_{(\kk,\ell)}^{(u)}\big).
\end{align*}
where $\xi_0\stackrel{{\rm def}}{=} (\xi_x,\xi_y,0)$. With the above preparation, we first establish Lemma \ref{l:spanning_condition1}, and then turn to Lemma \ref{l:spanning_condition}.

\begin{proof}[Proof of Lemma \ref{l:spanning_condition1}]
The estimate \eqref{eq:uniform_bound_below_algebraic_condition_1} follows clearly from $\inf_{(\kk,\ell)\in\z_\l}q_{(\kk,\ell)}>0$ and \eqref{eq:qq_j_def}.
Next, we focus on \eqref{eq:uniform_bound_below_algebraic_condition_2}. Since for all $j=(\kk,\ell)\in \zl$ and $\x\in \Dom$
$$
\textstyle
D_\l \fpr_\l (v,\x,\xi)\qq_j 
= 
\frac{\dd}{\dd \varepsilon}
D_\l \fpr_\l (v+\varepsilon q_{(\kk,\ell)}\g_{(\kk,\ell)} e_{\kk},\x,\xi)\big|_{\varepsilon=0},
$$ 
it follows that $D_\l \fpr_\l (v,\x,\xi)\qq_j $ is given by
\begin{align}
\label{eq:value_of_direction_derivative_FL_proof_algebraic_lemma}
q_{(\kk,\ell)}
\begin{pmatrix}
0\\
\g^{(u)}_{(\kk,\ell)}e_{\kk}(\x) -
\beta^{(u)}_{(\kk,\ell)}e_{\kk}(\x_0)  \\
\Pi_{\xi}\big[(\kk\cdot \xi) e_{-\kk}(\x) \,\g_{(\kk,\ell)}^{(u)} - (\kk\cdot \xi_0) e_{-\kk}(\x_0)\,\beta_{(\kk,\ell)}^{(u)}\big]\\
0
\end{pmatrix}.
\end{align}
For exposition convenience, we now split the proof into two steps.

\smallskip

\emph{Step 1: There exist positive constants $r_1,c_1$ such that, for all $r\geq r_1$, $V=(v_\l,\x,\xi,\eta)\in \HL_\l$ and $h=(h_v ,h_\x,h_\xi,h_\eta)\in T_{V}\HL_\l$, 
}
\begin{equation}
\label{eq:lower_bound_step_1_algebraic_lemma}
\max_{j\in \z_\l}|\langle D_\l \fpr_\l (v,\x,\xi)\qq_j  ,h \rangle|\geq c_1(\dist(\x))^{4}|h_\x|.
\end{equation}

For $j=(\kk,\ell)\in \z_\l$, we denote by $j'$ the reflected index $j'=(-\kk,\ell)\in \z_\l$. By standard computations, one can check that the vector $e_{\kk}(\x)
D_\l \fpr_\l (v,\x,\xi)\qq_j 
+
e_{-\kk}(\x)
D_\l \fpr_\l (v,\x,\xi)\qq_{j'}
$ 
is given by 
\begin{align*}
\Phi_{(\kk,\ell)}(z,\xi)\stackrel{{\rm def}}{=}
q_{(\kk,\ell)}
\begin{pmatrix}
0\\
\g^{(u)}_{(\kk,\ell)}-
\beta^{(u)}_{(\kk,\ell)} \cos(2\pi k_z z) \\
\sign(\kk)\,[
 (\kk\cdot \xi_0) \,\Pi_{\xi}\beta_{(\kk,\ell)}^{(u)}\, \sin (2\pi k_z z)]\\
0
\end{pmatrix},
\end{align*}
where, as in \eqref{eq:definition_gamma_kk_ell}, $\sign (\kk)=1$ if $\kk\in \Z^3_+$ or $\sign(\kk)=0$ otherwise. 
In particular, the vector
$
\tfrac{1}{2}(\Phi_{(\kk,\ell)}(z,\xi)+\Phi_{(-\kk,\ell)}(z,\xi)) 
$
has only one non-trivial component, namely the one in the $\x$-direction, which takes the form
$$
\phi_{(\kk,\ell)}(z)\stackrel{{\rm def}}{=}
\g^{(u)}_{(\kk,\ell)}-
\beta^{(u)}_{(\kk,\ell)} \cos(2\pi k_z z) .
$$
To prove \eqref{eq:lower_bound_step_1_algebraic_lemma}, it is enough to show that, for all vectors $h\in \R^3$, 
\begin{equation}
\label{eq:sufficient_condition_step_1_algebraic_lemma}
\textstyle
\max_{(\kk,\ell)\in \z_\l}|\langle 
\phi_{(\kk,\ell)}(z)  ,h \rangle|\gtrsim (|z|^2\wedge |1-z|^2)|h|.
\end{equation}
To prove the above, let $\kk_1=(1,0,1)$ and $\kk_2=(1,-1,0)$. It is clear that 
$$
\g_{(\kk_1,1)}
=
\begin{pmatrix}
1\\
0\\
\cos(2\pi z)-1
\end{pmatrix},
\qquad
\g_{(\kk_1,2)}
=
\begin{pmatrix}
0\\
1\\
0
\end{pmatrix},
\qquad 
\g_{(\kk_2,1)}
=
\frac{1}{\sqrt{2}}
\begin{pmatrix}
1\\
-1\\
0
\end{pmatrix}.
$$ 
Since $\svec\{\g_{(\kk_1,2)},\g_{(\kk_2,1)}\}=\R^2\times \{0\}$, the claim \eqref{eq:sufficient_condition_step_1_algebraic_lemma} follows by noticing that $(\kk_1,1), (\kk_1,2), (\kk_2,1)\in \z_\l$ and 
$1-\cos(2\pi z)\gtrsim |z|^2\wedge |1-z|^2$ by Taylor's expansion.

\smallskip

\emph{Step 2: There exist positive constants $r_2,c_2$ such that, for all $r\geq r_2$, $V=(v_\l,\x,\xi,\eta)\in \HL_\l$ and $h=(h_v ,h_\x,h_\xi,h_\eta)\in T_{V}\HL_\l$, 
}
\begin{equation}
\label{eq:lower_bound_step_2_algebraic_lemma}
\max_{j\in \z_\l}|\langle D_\l \fpr_\l (v,\x,\xi)\qq_j  ,h \rangle|\geq c_2(\dist(\x))^4 |h_\xi|.
\end{equation}

By \eqref{eq:value_of_direction_derivative_FL_proof_algebraic_lemma} and the argument in Step 2, it is enough to show that
\begin{equation}
\label{eq:span_step2_algebraic_lemma}
\svec \big\{\Psi_{(\kk,\ell)}(z,\xi)\,:\, (\kk,\ell)\in \z_\l\big\}= T_{\xi}\S^2,
\end{equation} 
with a corresponding lower bound, where 
$$
\Psi_{(\kk,\ell)}(\x,\xi)\stackrel{{\rm def}}{=}\Pi_{\xi}\big[
(\kk\cdot \xi) e_{-\kk}(\x) \,\g_{(\kk,\ell)}^{(u)} - (\kk\cdot \xi_0) e_{-\kk}(\x_0)\,\beta_{(\kk,\ell)}^{(u)}\big].
$$
Similar in Step 2, the vector $e_{-\kk}(\x) \Psi_{(\kk,\ell)}(\kk,\xi)- e_{\kk}(\x)\Psi_{(\kk,\ell)}(\kk,\xi)$ takes the form
$$
\psi_{(\kk,\ell)}(z,\xi)\stackrel{{\rm def}}{=}
\Pi_{\xi}\big[(\kk\cdot \xi) \g_{(\kk,\ell)}^{(u)} - (\kk\cdot \xi_0) \,\beta_{(\kk,\ell)}^{(u)} \, \cos(2\pi k_z z)\big].
$$
To prove \eqref{eq:span_step2_algebraic_lemma}, we distinguish two cases depending on a parameter $\delta>0$, which will be chosen depending on $\dist(\x)$. Below, we write $\xi=(\xi_x,\xi_y,\xi_z)\in \R^3$, and for $y\in \R$, we let $\sign (y)=1$ if $y\geq 0$, or $\sign(y)=-1$ if $y<0$. 

\smallskip

\emph{Claim 1: There exists $K>0$ such that \eqref{eq:lower_bound_step_2_algebraic_lemma} holds with RHS replaced by $c_2(\dist(\x))^2 |h_\xi|$ provided $|\xi_z|\leq K(\dist (\x))^2$.}
By a standard perturbation argument, in this case, it suffices to prove Claim 1 for $\xi_z=0$. Note that, as $\xi_z=0$ and $|\xi|=1$, either $\xi_x\geq \frac{1}{\sqrt{2}}$ or $\xi_y\geq \frac{1}{\sqrt{2}}$. 
In the following, we only discuss the case $\xi_x\geq \frac{1}{\sqrt{2}}$, as the other is similar. 
Let $\kk_1=(\sign(\xi_x),\sign(\xi_y),0)$ and $\kk_2=(\sign(\xi_x),0,1)$. Thus,
\begin{align*}
\psi_{(\kk_1,1)}(z,\xi)
=
\sqrt{2}(|\xi_x|+|\xi_y|)\Pi_\xi\begin{pmatrix}
\kk_1^\perp\\
0
\end{pmatrix}
, 		
\quad
\psi_{(\kk_2,1)}(z,\xi)
=
|\xi_x|\,\Pi_{\xi}
\begin{pmatrix}
1\\
0\\
-1+\cos(2\pi z)
\end{pmatrix}.
\end{align*}
Due to $\xi_z=0$, it holds that $T_{\xi}\R^2= \svec\{(\xi^\perp,0)^\top, (0,0,1)^\top\}$. Moreover, from $\kk_1\cdot\xi \neq 0$ it follows that $\Pi_\xi(\kk_1^\perp,0)^\top=\svec\{(\xi^\perp,0)\}$. 
Now, in the case $\xi_z=0$, the claim \eqref{eq:lower_bound_step_2_algebraic_lemma} holds with the RHS replaced by $c_2(\dist(\x))^2 |h_\xi|$. As noted earlier, by perturbation, this is sufficient to prove Claim 1.

\smallskip

\emph{Claim 2: Let $K>0$ be as in Claim 1. Then \eqref{eq:lower_bound_step_2_algebraic_lemma} holds for $|\xi_z|\geq  K(\dist (\x))^2$.}
Here we proceed as in Step 1 of the current proof. Consider 
$$
\wt{\psi}_{(\kk,\ell)}(z,\xi)\stackrel{{\rm def}}{=}
(\kk\cdot \xi) \g_{(\kk,\ell)}^{(u)} - (\kk\cdot \xi_0) \,\beta_{(\kk,\ell)}^{(u)} \, \cos(2\pi k_z z).
$$
Here, instead of \eqref{eq:span_step2_algebraic_lemma}, we prove $\svec \big\{\wt{\psi}_{(\kk,\ell)}(z,\xi)\,:\, (\kk,\ell)\in \z_\l\big\}= \R^3$ with a corresponding lower bound as in \eqref{eq:lower_bound_step_2_algebraic_lemma}. Indeed, for $\kk_1=(\sign(\xi_x), 0,\sign(\xi_z))$ and $\kk_2=(0, \sign(\xi_y),\sign(\xi_z))$, one has
\begin{align*}
\wt{\psi}_{(\kk_1,1)}(z,\xi)&=
(|\xi_x|+|\xi_z|)
\begin{pmatrix}
1\\
0\\
-\frac{\sign(\xi_x)}{\sign(\xi_z)}
\end{pmatrix}
+
|\xi_x|
\begin{pmatrix}
1\\
0\\
\frac{\sign(\xi_x)}{\sign(\xi_z)}
\end{pmatrix}
\cos(2\pi z), \\ 
\wt{\psi}_{(\kk_2,2)}(z,\xi)&=
(|\xi_y|+|\xi_z|)
\begin{pmatrix}
0\\
1\\
-\frac{\sign(\xi_y)}{\sign(\xi_z)}
\end{pmatrix}
+
|\xi_y|
\begin{pmatrix}
0\\
1\\
\frac{\sign(\xi_y)}{\sign(\xi_z)}
\end{pmatrix}
\cos(2\pi z), \\ 
\wt{\psi}_{(\kk_1,2)}(z,\xi)&=
(|\xi_x|+|\xi_z|)
\begin{pmatrix}
0\\
1\\
0
\end{pmatrix},
\quad \text{ and }\quad  
\wt{\psi}_{(\kk_2,1)}(z,\xi)=
(|\xi_x|+|\xi_z|)
\begin{pmatrix}
1\\
0\\
0
\end{pmatrix}.
\end{align*}
The above readily implies the estimate \eqref{eq:lower_bound_step_2_algebraic_lemma} as $|\xi_z|(1-\cos(2\pi z))\gtrsim (\dist(\x))^4$. 

The claim \eqref{eq:lower_bound_step_2_algebraic_lemma} of Step 2 hence follows from Claims 1 and 2. This finished the proof of Lemma \ref{l:spanning_condition1}.
\end{proof}

\begin{proof}[Proof of Lemma \ref{l:spanning_condition}]
Here, we essentially follow \cite[Lemma 6.15]{BBPS2022}. As in the latter proof, we distinguish two cases.
The parameter $r_0>0$ is chosen in Step 2.

\smallskip

\emph{Step 1: $\|v\|_{\Hs^\sigma(\T^3)}\geq 2r$}. Note that the non-decreasingness of $\chi$ and \eqref{eq:choice_cutoff} imply 
\begin{equation*}
\chi_{r}(\|v\|_{\Hs^\sigma})= 1.
\end{equation*}
In particular, as in Case 1 of \cite[Lemma 6.15]{BBPS2022}, we can use the regularising $\eta$ process to span the missing directions $(\x,\xi)$. 
More precisely, for $j\in \{1,\dots,6\}$,  
since $
D_\l F_\l (\v)\qq_j = \frac{\dd}{\dd \varepsilon} H(\x,(\xi,\eta)+\varepsilon a_j) |_{\varepsilon=0}
$, 
it holds that
$$
|\langle 
D_\l F_\l (\v)\qq_j , h \rangle| \geq \frac{(\dist(\x))^{4}}{(1+|\eta|^2)^{3/2}} |\langle a_j,h \rangle|
\geq \frac{(\dist(\x))^{4}}{(1+|\eta|)^{3}} |\langle a_j,h \rangle|.
$$
In the above $(a_j)_{j=1}^6$ is the standard basis of $\R^6$.
This concludes Step 1.

\smallskip

\emph{Step 2: $\|v\|_{\Hs^\sigma(\T^3)}\leq 2r$}. In this case, we cannot rely on the regularisation induced by $\eta$. Instead, we exploit Lemma \ref{l:spanning_condition1} to the presence of $\fpr_\l$ in the drift, while the remaining terms are treated as lower-order perturbations. 
Firstly, 
due to \eqref{eq:choice_cutoff},
$$
\chi_{3r}(\|v\|_{\Hs^\sigma(\T^3)})= 0.
$$  
Hence, in case $j=(\kk,\ell)\in \z_\l$, we have
\begin{align*}
D_\l F_\l (V)\qq_{j}
&= D_\l \fvel_\l(v) \qq_j+D_\l \fpr_\l(v,\x,\xi)\qq_j\\
 &
 -\frac{1}{r}\,\chi'\Big(\frac{\|v\|_{\Hs^\sigma(\T^3)}}{r}\Big) \frac{q_{(\kk,\ell)} ( v, \g_{(\kk,\ell)} e_\kk)_{\Hs^\sigma(\T^3)}}{\|v\|_{\Hs^\sigma(\T^3)}} H(\x,\xi,\eta).
\end{align*}
where $v_{(\kk,\ell)} =\int_{\T^3}v(\x)\cdot \g_{(\kk,\ell)}e_\kk(\x)\,\dd \x $.
Now, we apply the operator $\langle \cdot,h \rangle$, and we estimate each term separately. As above, here $h=(h_v,h_\x,h_\xi,h_\eta)\in \H\times \R^3\times T_{\xi} \S^2\times \R^6$. 
Firstly, by \eqref{eq:def_Fvv}, $\|v\|_{\Hs^\sigma}\leq 2r$ and Lemma \ref{l:estimate_nonlinearity},
$$
|\langle D_\l \fvel_\l(v) \qq_j,h\rangle |= 
|\langle D_\l \fvel_\l(v) \qq_j,h_v\rangle|\lesssim_r |h_v|\lesssim \max_{j\in \zl} |\langle \qq_j ,h\rangle|. 
$$
Secondly, 
let $c_0\geq 0$ be as in Lemma \ref{l:spanning_condition1}. As $\chi'\in L^\infty(\R_+)$, there exists $r_0>0$ depending only on $\|\chi'\|_{L^\infty(\R_+)}$, $\sup_{(\kk,\ell)}q_{(\kk,\ell)}<\infty$ and $c_0$ such that, for all $r\geq r_0$,
$$
\frac{1}{r}
\sup_{j\in \z_\l}
\Big|\chi'\Big(\frac{\|v\|_{\Hs^\sigma(\T^3)}}{r}\Big) \frac{q_{(\kk,\ell)}( v, \g_{(\kk,\ell)} e_\kk)_{\Hs^\sigma(\T^3)} }{\|v\|_{\Hs^\sigma(\T^3)}} H(\x,\xi,\eta)\Big|
\leq \frac{c_0}{2} (\dist(\x))^{4}.
$$
Thus, there exists $K$ depending only on $r$ such that 
\begin{align*}
|\langle
&[A_\l -D_\l F_\l (V)]\qq_{j},h\rangle|\\
&
\geq |\langle D_\l \fpr_\l(v,\x,\xi)\qq_j,h\rangle |- K \max_{j\in \zl} |\langle \qq_j ,h\rangle| - \tfrac{1}{2}c_0\, (\dist(\x))^{4}(|h_\x|+|h_\xi|),
\end{align*}
where $h_\x$ and $h_\xi$ are the $\x$- and $\xi$-components of $h \in T_V\HL$, respectively. Additional, we used the fact that $H(\x,\xi,\eta)$ has only non-trivial $\x$ and $\xi$-components. 

Therefore, from \eqref{eq:uniform_bound_below_algebraic_condition_2}, we get
\begin{equation*}
\max_{j\in \z_\l}|\langle
[A_\l -D_\l F_\l (V)]\qq_{j},h\rangle|
\geq  \tfrac{1}{2}c_0 (\dist(\x))^{4} (|h_\x|+|h_\xi|)-K \max_{j\in \zl} |\langle \qq_j ,h\rangle| .
\end{equation*}
The claim \eqref{eq:uniform_bound_below_algebraic_condition_full} now follows from \eqref{eq:uniform_bound_below_algebraic_condition_1} in Lemma \ref{l:spanning_condition1}.
\end{proof}

\section{Approximate controllability and weak irreducibility}
\label{s:irreduc_controllability}
In this section, we prove Propositions \ref{prop:approximate_controllability} and \ref{prop:weak_irreducibility}. We begin by recalling the main notation. Let $(u_t)_t = (v_t, w(v_t))_t$, where $w(v_t)$ is as defined in \eqref{eq:def_w}, represent the Markov process on $\Hs^\sigma(\T^3)$ induced by the stochastic PEs \eqref{eq:primitive_full}, as provided by Proposition \ref{prop:global_well_posedness}. Here, $\sigma$ is as in \eqref{eq:coloring_assumption_noise_2}.
Furthermore, the processes $(u_t, \x_t, \xi_t)$, $(u_t, \x_t, \wc{\xi}_t)$, and $(u_t, \x_t, \A_t)$ correspond to the projective and Jacobian processes, respectively. These are governed by the following equations:
\begin{align}
\label{eq:lagrangian_flow_second}
\partial_t \x_t&= u_t(\x_t),\\
\partial_t \A_t &=\nabla u_t(\x_t) \A_t,\\ 
\partial_t \xi_t&= \Pi_{\xi_t} (\nabla u_t(\x_t)\xi_t),\\
\partial_t \wc{\xi}_t&= -\Pi_{\wc{\xi}_t} ([\nabla u_t(\x_t)]^{\top}\wc{\xi}_t),
\end{align}
subject to the initial conditions $\x_0=\x\in \Dom$, $\xi_0=\xi\in \S^2$, $\wc{\xi}_0=\wc{\xi}\in \S^2$ and $\A_0=\Id_{\R^{3\times 3}}$. 
In these equations, $\Pi_{\xi} = \Id - (\xi \otimes \xi)$ denotes the orthogonal projection from $\R^3$ onto the tangent space of $\S^2$ at the point $\xi$, where $\xi$ is viewed as an element of $\R^3$ with unit length. Additionally, $\top$ represents the matrix transpose.

\smallskip

Similar to Subsection \ref{ss:proof_irreducibility}, the proofs of Propositions \ref{prop:approximate_controllability} and \ref{prop:weak_irreducibility} follow from controllability. We begin by looking at the projective processes.

\begin{lemma}[Controllability -- Projective processes]
\label{l:controllability}
Let $Q$ be such that Assumption \ref{ass:Q} holds. Consider the control problem
\begin{equation*}
\partial_t v_t +\p\big[(v_t\cdot\nabla_{x,y})v_t+ w(v_t)\partial_z v_t\big]= \Delta v_t + Qg_t,
\end{equation*}
where $w(v)$ is as in \eqref{eq:def_w}. Set  $u_t=(v_t,w(v_t))_t$ and
let $(\a,\eta)$, $(\a',\eta')$ be two elements of $\Tcal\times \S^{2}$. Then, for all $T\in (0,\infty)$, there exists a control $g\in C^\infty([0,T];C^\infty(\T^3;\R^2)\cap \Ls_0^2(\T^3))$ such that 
$$
(v_0,\x_0,\xi_0)=(0,\a,\eta) \quad \text{ and }\quad 
(v_T,\x_T,\xi_T)=(0,\a',\eta').
$$
The above also holds with $(v_t,\x_t,\xi_t)$ replaced by $(v_t,\x_t,\wc{\xi}_t)$.
\end{lemma}

As noted in \cite[Remark 7.2]{BBPS2022}, the above also holds for the projective processes $(v_t,\x_t,\xi_t)$ and $(v_t,\x_t,\wc{\xi}_t)$ where the last components are considered as elements of the projective space $\Pr$. To see this, it suffices to use arbitrary representatives of the elements in $\Pr$ on the sphere $\S^2$.

\begin{proof}
For notational simplicity, we assume $T=1$; the general case follows similarly. Additionally, we focus on the case of $(v_t,\x_t,\xi_t)$, as the argument for the other case is analogous. 

The proof extends the ``shear flows'' method in \cite[Lemma 7.1]{BBPS2022} to the context of PEs. 
The main idea is roughly as follows. First we prove that during the time interval $[0,1/2]$ we can move the point $\x_t$ from $\a$ to $\a'$ while $\xi_t$ is moved to some $\xi'\in \S^2$, and in a second step we prove that during the time interval $[1/2,1]$ we can move $\xi_t$ from $\xi'$ to $\eta'$ without changing the position of $\x_t$. Moreover, the construction also shows that $v_0=v_1=0$.  
The proof naturally splits into two steps.

\smallskip

\emph{Step 1: There exists a smooth control $g$ such that support in time of $g$ is contained in $ (0,1/2)$ and a unit vector $\xi'\in \S^2$ such that }
$$
(u_{1/4},\x_{1/4},\xi_{1/4})=(0,\a,\eta) \quad  \text{ and }\quad
(u_{1/2},\x_{1/2},\xi_{1/2})=(0,\a',\xi'); 
$$
\emph{i.e.\ the control changes the position $\x_t$ from $\a$ to $\a'$.}

Note that solutions to the 2D Navier-Stokes equations also solve the PEs (by considering $z$ as a dummy variable). Therefore, in light of the proof of \cite[Lemma 7.1]{BBPS2022}, we can assume that during the interval $(0,1/4)$, we have moved the position of the horizontal variables. In particular, it is enough to show the claim of Step 1 whenever  
$
\Pi_{x,y}\a=\Pi_{x,y}\a',
$
where $\Pi_{x,y}$ is the projection onto the first two components and $(0,1/2)$ is replaced by $(1/4,1/2)$. 
Let $\a=(a,b,c)$ and $\a'=(a,b,c')$ for some $a,b\in\T$ and $c,c'\in (0,1)$. We can assume that either $c<c'$ or $c'<c$; otherwise, there is nothing to prove. 
Without loss of generality, we may assume $c< c'$; the other case is analogous.
Let $f:\T\to \R$ be a mean-zero smooth function, which will be fixed later. Furthermore let us fix $\varphi\in C_{{\rm c}}((1/4,1/2))$ such that 
\begin{equation}
\label{eq:choice_phi_t}
\varphi_t \geq 0 \quad \text{ and }\quad \int_{1/4}^{1/2}\varphi_t\, \dd t=\log \Big(\frac{c'}{c}\Big).
\end{equation}
Consider the flow
$$
v_{1,t}(\x)\stackrel{{\rm def}}{=}\varphi_t
\begin{pmatrix}
0\\
 f(z) \sin (y-b)
\end{pmatrix}, \qquad \x=(x,y,z)\in \T^3.
$$
Since $f$ is mean zero, we have $\int_0^1 \nabla_{x,y}\cdot v_{1,t}(\cdot,z)\,\dd z=0$ on $\T^2_{x,y}$. 

Set $F(z)\stackrel{{\rm def}}{=}\int_{0}^{z} f(z')\,\dd z'$ and note that
$$
[w(v_{1,t})](\x)=-\int_{0}^{z} (\nabla_{x,y}\cdot v_{1,t})(x,y,z') \,\dd z'
=\varphi_t\cos(y-b) F(z) .
$$
Hence, the induced vector field $u_{1,t}$ is given by 
$$
u_{1,t}(\x)=\begin{pmatrix}
v_{1,t}(\x)\\
[w(v_{1,t})](\x)
\end{pmatrix}
=\varphi_t
\begin{pmatrix}
0\\
 f(z) \sin (y-b)\\
F(z)\cos(y-b)
\end{pmatrix}, \qquad \x=(x,y,z)\in \T^3.
$$ 
Let us consider the associated flow:
$$
\dot{\x}_t= u_{1,t}(\x_t), \quad \text{with }\quad \  \x_0=\a.
$$ 
It is clear that $x_t\equiv a$ and 
\begin{align*}
\dot{y}_t &= \varphi_t f(z_t) \sin (y_t-b),  &  y_{1/4}&=b,\\
\dot{z}_t &=\varphi_t F(z_t)\cos(y_t-b), & z_{1/4}&=c.
\end{align*}
From the above, it also follows that $y_t\equiv b$, and the last equation simplifies to
\begin{equation}
\label{eq:dot_z_ode_F}
\dot{z}_t =\varphi_t F(z_t), \qquad z_{1/4}=c.
\end{equation}
Next, we choose a smooth mean-zero map $f$ for which the unique solution to the former ODE satisfies $z_{1/2}=c'$. 
To this end, recall that we are assuming $0< c<c'<1$. 
Let $f$ be such that $f|_{[0,c']}=1$ and $f$ is a mean zero smooth periodic map.
The existence of such $f$ is evident because $0\leq c<c'<1$. 
In particular, $F|_{[0,c']}=z$ and the unique solution of the ODE \eqref{eq:dot_z_ode_F} is given by 
$$
\textstyle
z_t= c \exp(\int_{1/4}^{t} \varphi_s \,\dd s )
$$ 
for all $t$ such that $z_t<1$. Since $\varphi_t$ satisfies \eqref{eq:choice_phi_t}, we have $z_{1/2}=c'$. The claim of Step 1 now follows by taking:
\begin{equation}
\label{eq:def_g1}
g_{1,t}\stackrel{{\rm def}}{=} Q^{-1} \big(\partial_t v_{1,t}-\Delta v_{1,t}-\p [(v_{1,t}\cdot\nabla_{x,y})v_{1,t}+w(v_{1,t})\partial_z v_{1,t}] \big).
\end{equation}
Indeed, since $Q^{-1}:\Hs^{r+\alpha}_0\to \Hs^{r}_0$ for all $r>0$ by Assumption \ref{ass:Q}, the above formula is well-defined due to the regularity of $(v_{1,t})_t$, the mean zero of $v_{1,t}$ and $\int_{\T^3}\p [(v_{1,t}\cdot\nabla_{x,y})v_{1,t}+w(v_{1,t})\partial_z v_{1,t}]\,\dd x=0$ as $\nabla\cdot u_{1,t}=0$.

\smallskip

\emph{Step 2: Let $\xi'$ be as in Step 1. There exists a smooth control $g_2$ such that support in time of $g_2$ is contained in $ [1/2,3/4]$ and}
$$
(u_{1/2},\x_{1/2},\xi_{1/2})=(0,\a',\xi') ,\qquad 
(u_{3/4},\x_{3/4},\xi_{3/4})=(0,\a',\eta');
$$
\emph{i.e.\ the control changes the direction $\xi_t$ from $\xi$ to $\xi'$ while keeping fixed the position $\a$.}
As in Step 1, following the proof of \cite[Lemma 7.1]{BBPS2022}, it suffices to demonstrate that we can construct flows that enable us to rotate the direction $\xi_t$ around the $x$- and $y$-axis during the time interval $(1/2,3/4)$.  We provide the details only for the $x$-axis, as the case for the $y$-axis is analogous.

Fix $\theta\in (0,2\pi)$ and pick $\psi\in C_{{\rm c}}((1/2,3/4))$ such that 
\begin{equation}
\label{eq:choice_psi_t}
\psi_t \geq 0 \quad \text{ and }\quad \int_{1/4}^{1/2}\psi_t\, \dd t=\theta.
\end{equation}
Next, we produce a flow that represents a rotation of angle $\theta$ around the $x$-axis at time $t=3/4$.
As Step 1, let us write $\a'=(a',b',c')\in \Dom$. Hence, $c'\in (0,1)$.
Let $f:\T\to \R$ be a smooth map such that 
\begin{equation}
\label{eq:condition_f_rotation}
\int_0^1 f(z)\,\dd z=0, \qquad
\int_0^{c'} f(z)\,\dd z =1,\qquad  f(c')=0,\qquad \dot{f}(c')=1.
\end{equation}
The proof of the existence of such an $f$ is elementary (indeed, it is enough to take a smooth function $f$ satisfying the last two pointwise conditions in \eqref{eq:condition_f_rotation}, and then modify appropriately the mapping far from the point $c'\in (0,1)$).
Let $F(\zeta)\stackrel{{\rm def}}{=}\int_{0}^{z} f(z)\,\dd z$. Arguing as in Step 1, for $\x=(x,y,z)\in \T^3$,
$$
v_{2,t}(\x)=\psi_t 
\begin{pmatrix}
0\\ 
f(z)\cos(y-b')
\end{pmatrix}
\quad \Longrightarrow\quad
u_{2,t}(\x)=  \psi_t
\begin{pmatrix}
0\\
f(z) \cos(y-b')\\
F(z) \sin(y-b')
\end{pmatrix}.
$$ 
As $f(c')=0$, we have $u_{2,t}(\a')=0$. Therefore, the flow induced by $u_{2,t}$ fixes the point $\a'$.
Next note, that for all $\x=(x,y,z)\in \T^3$,
$$
\nabla_{\x} u_{2,t}(\x)=  \psi_t
\begin{pmatrix}
0 & 0 & 0 \\
0&  \ \ \ f(z) \sin(y-b') & \dot{f}(z)\cos(y-b')\\
0&- F(z) \cos(y-b') & f(z)\sin(y-b')
\end{pmatrix}.
$$
In particular
$$
\nabla_{\x} u_{2,t}(\a')=  \psi_t
\begin{pmatrix}
0 & 0 & 0 \\
0& 0 & \dot{f}(c')\\
0&-F(c') & 0
\end{pmatrix}
\stackrel{\eqref{eq:condition_f_rotation}}{= }\psi_t 
\underbrace{\begin{pmatrix}
0 & 0 & 0 \\
0& 0 & 1\\
0&-1 & 0
\end{pmatrix}}_{\Gamma \stackrel{{\rm def}}{=}}.
$$
Finally, let us note that for all $\n\in\S^2$ we have $\Gamma\n=(0,-n_z,n_y)\perp \n$ and therefore
$$
\Pi_{\n} (\Gamma \n) = \Gamma\n - (\Gamma\n\cdot \n)\n = \Gamma\n . 
$$
Hence the direction flow is given by $\xi_t = \psi_t \Gamma \xi_t $ with $\xi_{1/2}=\xi'$ and therefore
$$
\xi_t=
\exp \big( \Gamma  \Psi_t \big)\xi'
= 
\begin{pmatrix}
1 & 0 &0\\
0 & \ \ \ \cos(\Psi_t ) & \sin (\Psi_t ) \\
0 & -\sin (\Psi_t )  & \cos(\Psi_t ) 
\end{pmatrix}\xi',
$$
where $\Psi_t\stackrel{{\rm def}}{=}\int_{1/2}^t \psi_s \,\dd s$.
Due to \eqref{eq:choice_psi_t},
the above shows that $\xi'$ is rotated of an angle $\theta$ around the $x$-axis. Since $\theta\in (0,2\pi)$ was arbitrary, this completes the proof of Step 2 by choosing $g_{2}$ as in \eqref{eq:def_g1} with $u_1$ replaced by $u_2$.
\end{proof}

Next, we obtain a controllability result for the Jacobian process.

\begin{lemma}[Controllability -- Jacobian processes]
\label{l:controllability_jacobian}
Let $Q$ be such that Assumption \ref{ass:Q} holds. Consider the control problem
\begin{equation}
\label{eq:controlled_jacobian}
\partial_t v_t +\p\big[(v_t\cdot\nabla_{x,y})v_t+ w(v_t)\partial_z v_t\big]= \Delta v_t + Qg_t,
\end{equation}
where $w(v)$ is as in \eqref{eq:def_w}. Fix $\a\in \Dom$ and $M\in (0,\infty)$. Then, for all $T\in (0,\infty)$, there exists a control $g\in C^\infty([0,T];C^\infty(\T^3;\R^2)\cap \Ls_0^2(\T^3))$ such that 
$$
(v_0,\x_0,\A_0)=(0,\a,\Id_{\R^{3\times 3}}) \quad \text{ and }\quad 
(v_T,\x_T)=(0,\a) \ \text{ and }\ |\A_T|>M.
$$
\end{lemma}

\begin{proof}
The proof of \eqref{eq:non_degeneracy_4} follows from \cite[Proposition 7.4]{BBPS2022} as solutions to the 2D Navier-Stokes equations are $z$-independent solutions to the 3D primitive equations. In other words, we consider solution $v_t$ to \eqref{eq:controlled_jacobian} that are independent of $z\in (0,1)$.
For simplicity, we take $t=1$. Fix $\psi\in C^{\infty}_{{\rm c}}(0,t)$ such that $\int_0^1 \psi_r \,\dd r=\log(1+ M)$ and let $\a=(a,b,c)$. To prove \eqref{eq:non_degeneracy_4} one takes for all $t\in [0,1]$ and $\x=(x,y,z)\in \T^3$,
$$
v_t(\x)= -\psi_t 
\begin{pmatrix}
\sin(y-b)\\
\sin(x-a)
\end{pmatrix}  \quad \Longrightarrow \quad 
u_t(\x)=- \psi_t 
\begin{pmatrix}
\sin(y-b)\\
\sin(x-a)\\
0
\end{pmatrix}.
$$
Note that $v_0=v_1=0$ and $u_t(\a)=0$. Hence, $\x_t\equiv \a$ and therefore $A_t$ satisfies $A_0=\Id$ as well as
$$
\partial_t \A_t =\psi_t \begin{pmatrix}
0 & 1 &0\\
1 & 0& 0\\
0 & 0 &0
\end{pmatrix}
\A_t \quad \Longrightarrow \quad  \A_t 
= \begin{pmatrix}
\vspace{0.1cm}
\cosh (\int_0^t \psi_r \,\dd r) & \sinh (\int_0^t \psi_r \,\dd r) &0\\
\vspace{0.1cm}
\sinh (\int_0^t \psi_r \,\dd r) & \cosh (\int_0^t \psi_r \,\dd r)& 0\\
0 & 0 &0
\end{pmatrix}.
$$
In particular $|\A_1|\geq  e^{\int_0^1 \psi_r \,\dd r}> M$, as desired.
\end{proof}

The following follows directly from Lemmas \ref{l:controllability} and \ref{l:controllability_jacobian}, along with a stability argument as the one used in the proof of \cite[Lemma 7.3]{BBPS2022} or Proposition \ref{prop:irreducibility_Appendix}.

\begin{lemma}
\label{l:non_degeneracy}
For all $T>0$ and $\varepsilon>0$, there exists $\varepsilon'>0$ such that for all elments $(\a,\eta)$ and $(\a',\eta')$ in $\Tcal\times \S^2$ and all $v\in B_{\Hs^\sigma_0}(0,\varepsilon')$,
\begin{align}
\label{eq:non_degeneracy_1}
\P\big( (v_T,\x_T)\in B_{\Hs^\sigma}(0,\varepsilon)\times B_{\Dom}(\a',\varepsilon)\,|\, (v_0,x_0)=(v,\a) \big)>0,&\\
\label{eq:non_degeneracy_2}
\P\big( (v_T,\x_T,\xi_T)\in K_{\sigma,\varepsilon}(\a',\eta')\,|\, (v_0,\x_0,\xi_0)=
(v,\a,\eta) \big)>0,&\\
\label{eq:non_degeneracy_3}
\P\big( (v_T,\x_T,\wc{\xi}_T)\in K_{\sigma,\varepsilon}(\a',\eta')\,|\, (v_0,\x_0,\wc{\xi}_0)=(v,\a,\eta) \big)>0,
\end{align}
where $K_{\sigma,\varepsilon}(\a',\eta')=B_{\Hs^\sigma}(0,\varepsilon)\times B_{\Dom}(\a',\varepsilon)\times B_{\S^2}(\eta',\varepsilon)$.
Finally, for all $M>0$, $\a\in \T^3$ and $(t,\varepsilon)$ as above and $J_{\sigma,\varepsilon,M}=B_{\Hs^\sigma}(0,\varepsilon)\times B_\Dom(\a,\varepsilon)\times (B_{\R^{3\times 3}}(0,M))^\complement$,
\begin{equation}
\label{eq:non_degeneracy_4}
\P\big((v_T,\x_T,\A_T)\in J_{\sigma,\varepsilon,M} \,|\, (0,\x,\A_0)=
(0,\a,\Id)\big)>0.
\end{equation} 
\end{lemma}

We are ready to prove Propositions \ref{prop:approximate_controllability} and \ref{prop:weak_irreducibility}.

\begin{proof}[Proof of Proposition \ref{prop:approximate_controllability}.]
Items \eqref{it:approximate_controllability1} and \eqref{it:approximate_controllability2} correspond precisely to \eqref{eq:non_degeneracy_2} and \eqref{eq:non_degeneracy_4}, respectively.
As for \eqref{it:approximate_controllability0}, note that \eqref{eq:non_degeneracy_1} implies $\supp\Pl_t\supseteq \{0\}\times \Dom$ and therefore \eqref{it:approximate_controllability0} from the closedness of $\supp\Pl_t$.
\end{proof}

\begin{proof}[Proof of Proposition \ref{prop:weak_irreducibility}]
Here, we partially follow the proof in \cite[Proposition 2.15, pp.\ 1976]{BBPS2022}.
We only prove the assertion for the Lagrangian process; the other claim follows similarly. 
Let $\invl$ be an arbitrary invariant measure for $(u_t,\x_t)$ on $\H\times \Dom$. 
By Proposition \ref{prop:regularity_invariant_measure}, one can check that $\invl$ is concentrated on $ \Hs^{\sigma'}_0\times \Dom$ for all $\sigma'<\sigma+1$, where $\sigma$ as in \eqref{eq:coloring_assumption_noise_2}. 
Fix $\sigma'\in (\sigma,\sigma+1)$ and let $N\gg 1$ be such that 
$$
\invl(\B)>1/2\qquad \text{ and }\qquad \B\stackrel{{\rm def}}{=}\{\|v\|_{\Hs^{\sigma'}}\leq N\}\times \Dom
$$

Next, we collect some facts about the (deterministic) PEs without forcing, i.e.,
\begin{equation}
\label{eq:deterministic_primitive_proof_stability}
\partial_t \vd_t+ \p[(\vd_t\cdot \nabla_{x,y})\vd_t+w(\vd_t)\partial_z \vd_t]=\Delta \vd_t, \quad \vd_t=v.
\end{equation}
where $v\in \Hs^{\sigma'}_0(\T^3)$ satisfies $\|v\|_{H^{\sigma'}}\leq N$ and $N$ and $\sigma'\in (\sigma,\sigma+1)$ are as above.
The well-posedness of the above is standard and follows, for instance, by Proposition \ref{prop:global_well_posedness}. 
By standard energy methods, $\|\vd_t\|_{L^2}\leq e^{-t} \|v\|_{L^2}$ for all $t>0$. Moreover, by combining \cite[Proposition 3.5]{LT19} and the nonlinear estimates in Lemma \eqref{l:estimate_nonlinearity}, we get the uniform bound $\|\vd_t\|_{H^{\sigma'}}\leq C(N,\sigma')<\infty$. By interpolation we get $\|\vd_t\|_{\Hs^\sigma}\leq C(\sigma,\sigma', N) e^{-t \beta}$ for some $\beta(\sigma,\sigma')\in (0,1)$. 
Thus, the previous facts on \eqref{eq:deterministic_primitive_proof_stability} by a stability argument as the one in \cite[Lemma 7.3]{BBPS2022} or Proposition \ref{prop:irreducibility_Appendix}, one easily obtains the following assertion: For all $\delta>0$, there exists a deterministic time $T_\delta>0$ such that, for all $(v,\x)\in \B$, 
\begin{equation}
\label{eq:T_delta_proof_decay}
\P((v_{T_\delta},\x_{T_\delta})\in B_{\Hs^\sigma}(0,\delta)\times B_{\Dom}(\x,\delta) \,|\, (v_0,\x_0)=(v,\x))>0.
\end{equation}
In the following, if no confusion seems likely, we write $B(0,\delta)$ and $B(\x,\delta)$ instead of $B_{\Hs^\sigma}(0,\delta)$ and $ B_{\Dom}(\x,\delta)$, respectively.

Now, fix $\varepsilon>0$, $(v,\x)\in \B$ and $\x'\in \Dom$. 
Let $\varepsilon'>0$ be as in Lemma \ref{l:non_degeneracy} for $T=1$. Moreover, let $T_{\varepsilon'}$ be such that \eqref{eq:T_delta_proof_decay} holds with $\delta=\varepsilon'$. Finally, recall that $\Pl_t$ denotes the Markov semigroup associated with the Lagrangian process $(v_t,\x_t)$.
By the semigroup property, it follows that
\begin{align*}
&\P((v_{T_{\varepsilon'}+1}, \x_{T_{\varepsilon'}+1}) \in B(0,\varepsilon)\times B(\x',\varepsilon)\,|\, (v_0,x_0)=(v,\x))\\
& =[\Pl_{T_\delta} (\Pl_1(\one_{ B(0,\varepsilon)\times B(\x',\varepsilon)}))](v,\x)\\
& \textstyle
=\int_{\O} \P\big((v_1,\x_1)\in  B(0,\varepsilon)\times B(\x',\varepsilon) \, | \,v_0=v_{T_{\varepsilon'}}(\wt{\om}),\x_0=\x_{T_{\varepsilon'}}(\wt{\om})\big)\,\dd \P(\wt{\om})\stackrel{(i)}{>}0.
\end{align*}
where $v_{T_{\varepsilon'}}(\wt{\om})$ and $\x_{T_{\varepsilon'}}(\wt{\om})$ are the solution to the stochastic PEs \eqref{eq:primitive_full} with initial data $(v,\x)$ for the noise realization $\wt{\om}\in \O$ and the associated solution to the Lagrangian flow \eqref{eq:lagrangian_flow_second}. Moreover, in $(i)$ we used \eqref{eq:T_delta_proof_decay} as well as \eqref{eq:non_degeneracy_1}.

Hence, by invariance of $\invl$,
$$\textstyle
\invl(B(0,\varepsilon)\times B(\x',\varepsilon))
\geq \int_{\B} \Pl_{T_{\varepsilon'+1}}((v,\x),\, B(0,\varepsilon)\times B(\x',\varepsilon))\,\dd \invl(v,\x)>0.
$$
Hence, the arbitrariness of $\varepsilon$ and $\x'\in \Dom$ implies $\supp\,\wt{\mu}\supseteq \{0\}\times \Dom$.
\end{proof}

\appendix

\section{Some useful results on Feller Markov semigroups}
\label{app:useful_results}
Here, we gathered some results on the general theory of Markov processes, which were used in the proof of our main results in Subsection \ref{ss:proofs_further_results}. 
The reader is referred to the text before Proposition \ref{prop:strong_feller_projective} for the definition of the strong Feller property. 

\begin{lemma}
\label{l:ergodicity_supp_product}
Let $Z$ and $W$ be Polish spaces.  
Let $(z_t)_t$ and $(z_t,w_t)_t$ be Markov processes on $Z$ and $Z\times W$, respectively. Assume that $(z_t,w_t)_t$ is strongly Feller on $Z\times W$. Let $\nu$ be an ergodic measure for $(z_t)_t$. Suppose there exists a measure $\mu$ on $W$ with connected support and such that $\nu\times \mu$ is invariant for $(z_t,w_t)_t$. Then $\nu \otimes \mu$ is ergodic.
\end{lemma}

As the case analysed in Subsection \ref{ss:proofs_further_results}, the above can be useful in case the uniqueness of invariant measures (and therefore ergodicity) cannot be easily proven via the Krein-Milman and a variant of the Khasminskii's theorems (see \cite[Proposition 4.1.1]{DPZ_ergodicity} for the latter). 

The above result immediately implies

\begin{corollary}
\label{cor:ergodicity_supp}
Let $(z_t)_t$ be a strongly Feller Markov process on a Polish space $Z$. Let $\nu$ be an invariant measure for $(z_t)_t$ on $Z$ such that $\supp\nu$ is connected. Then $\mu$ is ergodic.
\end{corollary}

The connectedness assumption cannot be omitted in the above results.
Indeed, in the situation of Corollary \ref{cor:ergodicity_supp}, if a Feller Markov process $(z_t)_t$ has two dinstict ergodic measures, then $\mu=\frac{1}{2}(\mu_1+\mu_2)$ is invariant but not ergodic. However, if the strong Feller property holds, it has disconnected support as $\supp\mu_1\cap \supp\mu_2=\emptyset$.

\begin{proof}[Proof of Lemma \ref{l:ergodicity_supp_product}]
By a well-known characterisation of ergodic measures (see e.g., \cite[Theorem 3.2.4]{DPZ_ergodicity}), it is sufficient to show that if a set $E\subseteq Z$ satisfies
\begin{equation}
\label{eq:P1_one_E_product}
P_t \one_{E} (z,w)= \one_{E}(z,w)
\end{equation}
for all $t>0$ and $\nu\otimes \mu$-almost all $(z,w)$, then either $\nu\otimes \mu(E)=0$ or $\nu\otimes \mu(E)=1$. In the previous, $(P_t)_t$ denotes the semigroup associated to the Markov semigroup $(z_t,w_t)_t$, i.e., $P_t \one_{E} (z,w)=\E[\one_{E}(z_t,w_t)]$.

Next, we assume that \eqref{eq:P1_one_E_product} holds. 
Let $F_0\subseteq Z\times W$ be a full measure set on which the equality \eqref{eq:P1_one_E_product} holds for $t=1$. 
Clearly, $F_0$ is dense in $\supp\nu\times \supp\mu$. 
By the strong Feller and \eqref{eq:P1_one_E_product} with $t=1$, the mapping $F_0\ni (z,w)\mapsto P_1 \one_{E} (z,w)=\one_{E}(z,w)$ can be extended to a continuous mapping on $\supp\nu\times\supp\mu$; which we denoted by $\psi_E$. From the equality \eqref{eq:P1_one_E_product} with $t=1$, the mapping $\psi_E$ takes values in $\{0,1\}$. In particular, the sets $A_j=\psi_{E}^{-1}(\{j\})$ for $j\in \{0,1\}$ are open and closed in $\supp\nu\times\supp\mu$, $A_0\cup A_1=\supp\nu\times\supp\mu$ and 
\begin{equation}
\label{eq:A1_E_approximate}
\nu\otimes \mu(A_1\Delta E)=0 \ \ \text{ where }\ \ A_1\Delta E\stackrel{{\rm def}}{=} A_1\setminus E\cup E\setminus A_1.
\end{equation}
By connectedness of $\supp\mu$, it follows that $A_j=\wt{A}_j\times \supp\mu$ where $\wt{A}_j\in \Borel(Z)$. Indeed, if for instance, $A_0$ has not such a product structure, then there exist $z\in Z$ and $w,w'\in W$ such that $(z,w)\in A_0$ and $(z,w')\in A_1$. Since $A_0,A_1$ are open and closed in $\supp\nu\times \supp\mu$, it follows that the sections $A_j^z\stackrel{{\rm def}}{=}\{w\in \supp\mu\,:\, (z,w)\in A_j\}$ are also open and closed. By connectedness of $\supp\mu$, either $A_0^z=\supp\mu$ or $A_0^z=\emptyset$; this contradicts the fact that $A_0$ does not have a product structure. 

Hence, we proved that $A_j=\wt{A}_j\times \supp\mu$ for some $\wt{A}_j\in \Borel (Z)$. Due to \eqref{eq:A1_E_approximate} and the invariance of $\nu\otimes \mu$, it follows that $P_t \one_{A_1\Delta E}= 0$, $\nu\otimes \mu$-almost everywhere. Therefore, \eqref{eq:P1_one_E_product} implies $\wt{P}_t \one_{\wt{A}_1} = \one_{\wt{A}_1}$ for all $t>0$ and $\nu$-almost everywhere. Here, $(\wt{P}_t)_t$ denotes the Markov semigroup induced by $(z_t)_t$. The ergodicity of $\nu$ implies either $\nu(\wt{A}_1)=0$ or $\nu(\wt{A}_1)=1$; and therefore either $\nu\otimes \mu(E)=0$ or $\nu\otimes \mu(E)=1$ due to  \eqref{eq:A1_E_approximate}, as desired.
\end{proof}

Finally, we prove a uniqueness result of invariant measures with the same marginals. With a slight abuse of notation, for a measure $\mu$ on the product space $Z\times W$, we denote by $\mu|_{Z}$ the marginal of $\mu$ on $Z$, i.e., $\mu|_Z (A)= \mu(A\times W)$ for all $A\in \Borel (Z)$.

\begin{lemma}
\label{l:uniqueness_z_factor}
Let $Z$ and $W$ be Polish spaces.  
Let $(z_t)_t$ and $(z_t,w_t)_t$ be Markov processes on $Z$ and $Z\times W$, respectively. 
Let $\nu$ be an ergodic measure for the Markov process $(z_t)_t$, and let $\mu$ be an extreme point of the convex set
$$
\Set=\{\mu\text{ is an invariant measure for $(z_t,w_t)_t$ such that $  \mu|_{Z}=\nu$}\}.
$$ 
Then $\mu$ is ergodic. 
\end{lemma}

By the Krein-Milman theorem, the extremality of $\mu$ in $\Set$ follows trivially in case $W$ is compact. This situation appears in our application of Lemma \ref{l:uniqueness_z_factor} where $W=\Pr$ is the two-dimensional projective space, see Subsection \ref{ss:proofs_further_results}.

\begin{proof}
Here, we employ the characterisation of ergodic measures as the extreme point of invariant measures, see e.g., 
\cite[Proposition 3.2.7]{DPZ_ergodicity}. By contradiction, assume that $\mu$ is not ergodic. Hence, there exist $t\in (0,1)$ and distinct invariant measures $\mu_1$ and $\mu_2$ of $(z_t,w_t)_t$ such that 
$
\mu=(1-t)\mu_1+t\mu_2.
$
It is clear that the marginals measures $\mu_1|_Z$ and $\mu_2|_Z$ are invariant for $(z_t)_t$. Moreover, since $\mu|_Z=\nu$, the marginals measures $\mu_1|_Z$ and $\mu_2|_Z$ satisfy $
\nu=(1-t)\mu_1|_Z+t\mu_2|_Z$.
The ergodicity of $\nu$ implies $\nu=\mu_1|_Z=\mu_2|_Z$ (\cite[Proposition 3.2.7]{DPZ_ergodicity}). Hence, $\mu_1,\mu_2\in \Set$ and the assumption $\mu=t\mu_1+(1-t)\mu_2$ with $\mu_1\neq \mu_2$ and $t\in(0,1)$ contradicts the extremality of $\mu$ in $\Set$.
The proof is therefore complete. 
\end{proof}

\subsubsection*{Acknowledgments.}
The author thanks Jacob Bedrossian for his valuable comments and encouragement in writing this manuscript.
The author thanks Lucio Galeati, Umberto Pappalettera and Joris van Winden for fruitful discussions.
The author is grateful for the hospitality provided by the Math Department of UCLA during his stay in October 2024. 

\medskip

\noindent
\textbf{Data availability.} This manuscript has no associated data.

\smallskip

\noindent
\textbf{Declaration -- Conflict of interest.} The author has no conflict of interest.

\def\polhk#1{\setbox0=\hbox{#1}{\ooalign{\hidewidth
  \lower1.5ex\hbox{`}\hidewidth\crcr\unhbox0}}} \def\cprime{$'$}

\end{document}